\documentclass[12pt,a4paper]{article}
\usepackage{graphicx}
\usepackage{amsmath}
\usepackage{amssymb}
\usepackage{amsfonts}
\usepackage{amsthm}
\usepackage{hyperref}
\numberwithin{equation}{section}

\renewcommand\Im{\mathop{\rm Im}}
\renewcommand\Re{\mathop{\rm Re}}
\newcommand\p{\mathbf{x}}
\def\func#1{\mathop{\rm #1}}
\newtheorem{theorem}{Theorem}[section]
\newtheorem{corollary}[theorem]{Corollary}

\newtheorem{lemma}[theorem]{Lemma}
\newtheorem{proposition}[theorem]{Proposition}
\newtheorem{remark}[theorem]{Remark}

\title{Asymptotic series for the splitting of separatrices\\ near a Hamiltonian bifurcation}
\author{V.~Gelfreich  and N. Br\"annstr\"om\thanks{This work is partially supported by the EPSRC grant EP/C000595/1.}
\\[6pt]
Mathematics Institute\\ University of Warwick\\[6pt]
V.Gelfreich@warwick.ac.uk\\
N.L.A.Brannstrom@warwick.ac.uk}
\begin{document}
\maketitle

\begin{abstract}
This is a proof of an asymptotic formula which describes exponentially small
splitting of separatrices in a generic analytic family of
area-preserving maps near a Hamiltonian saddle-centre bifurcation.
As a particular case and in combination with an earlier work on a Stokes constant for the
H\'enon map (Gelfreich, Sauzin (2001)), it implies exponentially small transversality of
separatrices in the area-preserving H\'enon family 
when the multiplicator of the fixed point is close to one.
\end{abstract}

{
\tableofcontents
}
\newpage

There has been a substantial progress in the study of exponentially small splitting
of separatrices (see for example \cite{Delshams,Seara,Ramirez,Treschev98a,Treschev98b}).
In most cases exponentially small asymptotic expansions are very sensitive to the form of the
system. In the case of local bifurcations with an integrable normal
form some generically valid formulae have been obtained~\cite{Lombardi,Gel00,Gel02}.
In degenerate cases the splitting of separatrices may be studied by the Melnikov
method \cite{Seara,Delshams} but application of this method to
exponentially small phenomena add substantial restrictions onto
systems under consideration (see e.g. survey \cite{GL01}). 
Since the pioning work by Lazutkin, it is accepted that 
in the case when Melnikov method fails the asymptotic formulae
contain a special pre-factor which comes from a parameter-independent
problem and can be interpreted as a Stokes constant (see e.g. \cite{HakimMallock}).
For the class of system considered in the present paper, that problem
is considered separately in \cite{GN2008}.

We note that transversality of separatrices plays an important role in some
applications (\cite{Duarte99,Duarte00}).

\section{Introduction}

Let $f_\varepsilon$ be an analytic family of area-preserving maps
defined in a neighbourhood of the origin. 
We assume that when $\varepsilon=0$
the origin is a parabolic fixed point of the map.
This means that $f_0(\mathbf{0})=\mathbf{0}$ and one is a double eigenvalue
of the Jacobian $Df_0(\mathbf{0})$. These eigenvalues are called
{\em multipliers\/} of the fixed point. In the generic case the Jacobian
is not diagonalisable. Then we may assume that
\begin{equation}\label{Eq:Df0}
Df_0(\mathbf{0})=\left(\begin{array}{ll}1&1\\0&1\end{array}\right)\,,
\end{equation}
as this form can be achieved by a linear area-preserving
change of variables. Since the multipliers are equal to one
the fixed point is not structurally stable and can be destroyed
by a small perturbation of $f_0$. We will consider 
the non-degenerate bifurcation.
This means that two leading coefficients
in a normal form of $f_\varepsilon$ do not vanish: 
one of them is a leading non-linear term of $f_0$ and the second one
guarantees that $f_\varepsilon$ is a ``generic unfolding" of~$f_0$. 

Let us state these two assumptions in a more precise way. We will
see that the bifurcation can be described by an integrable normal form.
We will provide a detailed description of the normal form later but at
the moment we note that the map $f_\varepsilon$ can be
approximated by a time one map of an autonomous Hamiltonian flow.
The leading order of the Hamiltonian has the form
\begin{equation}\label{Eq:h6non}
h_6=\frac{y^2}2 +\frac{ax^3}3 -bx\varepsilon\,.
\end{equation}
In this paper we assume
\begin{equation}\label{Eq:nondegeneracy}
a,b>0.
\end{equation}
The positivity assumption is not more restrictive than $ab\ne 0$. Indeed,
the sign of $a$ and $b$ can be changed by substitutions
$x\mapsto -x$ and $\varepsilon\mapsto-\varepsilon$.
Using assumption (\ref{Eq:nondegeneracy}) 
we ensure that the phenomena we study in this paper are on 
the side of positive $\varepsilon$.

Near the bifurcation the map $f_\varepsilon$ is rather accurately approximated 
by the time one map $\Phi_{h_6}^1$ generated by the Hamiltonian~(\ref{Eq:h6non}).
When $\varepsilon$ changes sign an equilibrium of $h_6$
bifurcates following the scenario sketched in Figure~\ref{Fig:h6bif}.
For $\varepsilon>0$, separatrices of the saddle point form
a loop around the elliptic fixed point.

\begin{figure}
\boxed{\includegraphics[width=0.3\textwidth,height=0.2\textwidth]{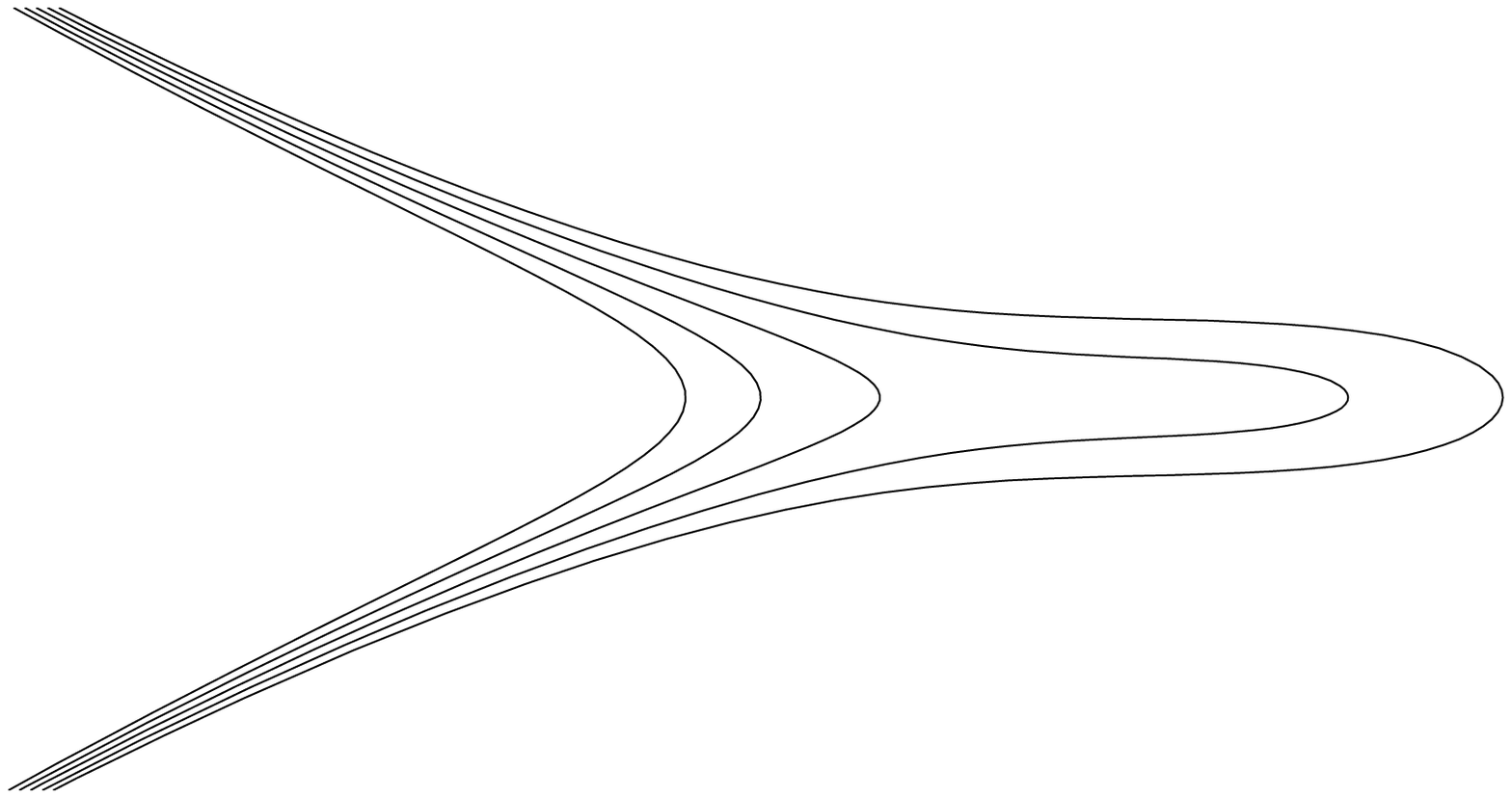}}
\boxed{\includegraphics[width=0.3\textwidth,height=0.2\textwidth]{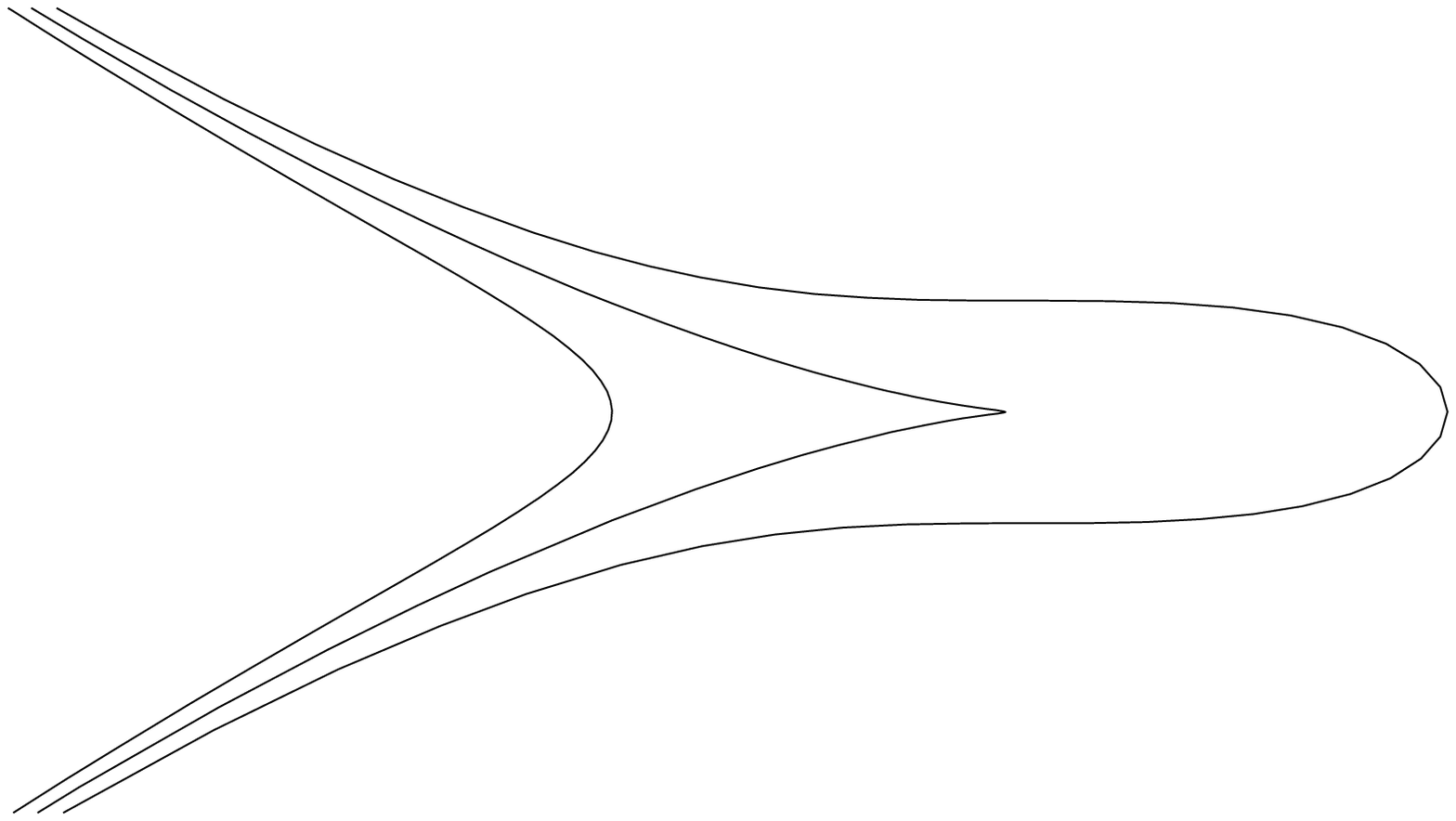}}
\boxed{\includegraphics[width=0.3\textwidth,height=0.2\textwidth]{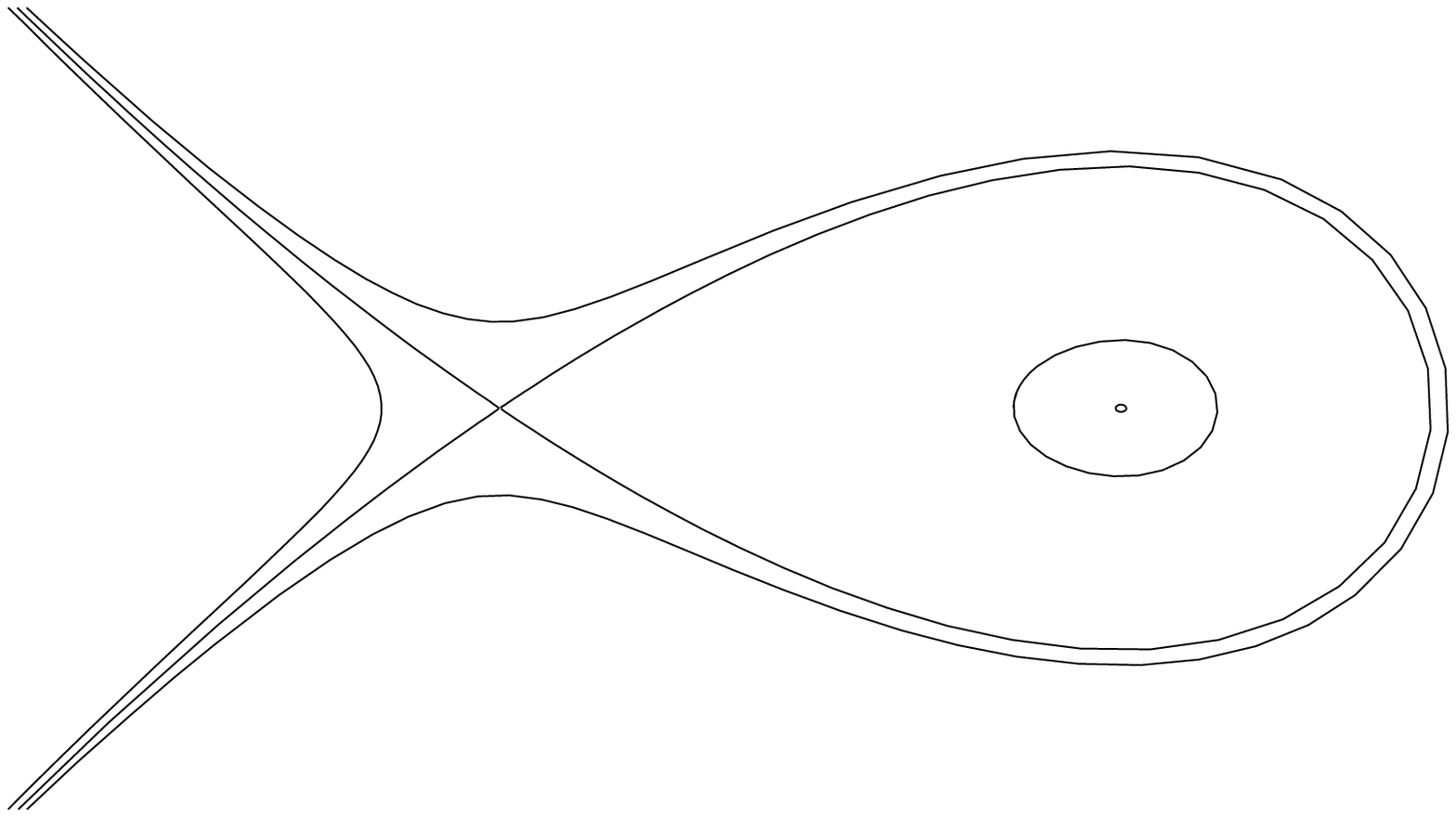}}
\caption{Bifurcation of the normal form: a small separatrix loop is created after
a Hamiltonian saddle-centre bifurcation\label{Fig:h6bif}}
\end{figure}

Under these assumptions $f_\varepsilon$ with $\varepsilon>0$ 
has two fixed points which collide at the origin when $\varepsilon=0$ and then disappear. 
One of the fixed points is elliptic (centre) and another one is hyperbolic (saddle). 
For this reason the bifurcation is sometimes called a Hamiltonian saddle-centre bifurcation.

The bifurcation of the map is very similar to the bifurcation
of the leading order normal form (\ref{Eq:h6non}). 
Unlike the normal form, we do not expect a generic 
map to be integrable and therefore the separatrices
of the map may split. We will derive an asymptotic formula
which describes the splitting of the separatrices. This formula implies that
in a generic family of maps the separatrices intersect transversally for all sufficiently
small $\varepsilon$.

We note that the area preserving H\'enon map satisfies assumptions of our theorem.

The leading order of the asymptotic expansion was described in~\cite{Gel00,Gel02}.
The goal of this paper is to present a complete proof
of a refined version of this asymptotic formula.
The proof is based on the programme of~\cite{G99} and is
a development of the original idea proposed by Lazutkin in \cite{Lazutkin1984}.

Let us state the main result of the paper. First we note that
many of the series involved in our proof
involve powers of $\varepsilon^{1/4}$. It is therefore
convenient to introduce a new small parameter $\delta=\varepsilon^{1/4}$.
We will occasionally use both of them indicating the type of
a power series expansion for the corresponding function.

In this paper we will study separatrices of a
family of hyperbolic fixed points:
\[
\mathbf{p}_\delta=f_\varepsilon(\mathbf{p}_\delta)
\]
such that $\mathbf{p}_0=\mathbf{0}$.
Let $\lambda_\delta\ge1$ be the largest of the two multipliers of $f_\varepsilon$
at $\mathbf{p}_\delta$.
We will prove that $\mathbf{p}_\delta$ and $\lambda_\delta$ depend analytically
on $\delta$. Moreover, $\log\lambda_\delta$ is of the same order of smallness
as $\delta$ itself. We will see that 
\begin{equation}\label{Eq:lambda}
\log\lambda_\delta=\sqrt[4]{4ab}\,\delta+O(\delta^3)\,.
\end{equation}

It is convenient to represents stable and unstable separatrices $W^\pm(\mathbf{p}_\delta)$
using solutions of the following finite-difference equation 
(for a more detailed discussion see the review~\cite{GL01})
\begin{equation}\label{Eq:mainfde}
\psi^\pm(t+\log\lambda_\delta)=f_\varepsilon\circ \psi^\pm(t)\,.
\end{equation}
We have suppressed explicit dependence of $\psi^\pm$ on $\delta$
to shorten the notation. An application of $f_\varepsilon$ 
increases the parameter $t$ by $\log\lambda_\delta>0$.
Therefore the conditions
\[
\lim_{t\to-\infty}\psi^-(t)= \mathbf{p}_\delta
\qquad\mbox{and}\qquad
\lim_{t\to+\infty}\psi^+(t)= \mathbf{p}_\delta
\]
imply that $\psi^-$ represents the unstable separatrix
and $\psi^+$ represents the stable one.
It is easy to see that these conditions do not defined the functions $\psi^\pm$ uniquely.

The most convenient parameterization is obtained using analytic linearizations of restrictions of 
$f_\varepsilon$ onto $W^-(\mathbf{p}_\delta)$ and $W^+(\mathbf{p}_\delta)$ respectively.
These parametrisations can be represented in the form
$\psi^-(t)=\Psi^-(e^t)$ and $\psi^+(t)=\Psi^+(e^{-t})$ where $\Psi^\pm(z)$ 
are analytic at $0$ and $\Psi^\pm(0)=\mathbf p_\delta$.
This choice reduces the freedom in $\psi^\pm$ to a translation $t\mapsto t+t_0$.
Our proof of their existence does not use the linearization but relies
on a contraction mapping argument which has the advantage of providing
an accurate approximation for the functions.

From the technical point of view a substantial part of the present paper is
dedicated to a detailed study of the analytic continuation for $\psi^\pm$.

Let $\gamma\in W^-(\mathbf{p}_\delta)\cap W^+(\mathbf{p}_\delta)$ 
be an homoclinic point. Then
\[
\gamma=
\psi^+(t_s)=\psi^-(t_u)
\]
for some $t_u$ and $t_s$. The vectors $\dot\psi^-(t_u)$ and $\dot\psi^+(t_s)$
are tangent respectively to $W^-(\mathbf{p}_\delta)$ and $W^+(\mathbf{p}_\delta)$ 
at $\gamma$.
The area of the parallelogram generated by these tangent vectors
is called {\em the homoclinic invariant of $\gamma$}:
\[
\omega(\gamma)=\Omega\left(\dot\psi^-(t_s),\dot\psi^+(t_u)\right)
\,,
\]
where $\Omega$ stands for the standard area form.
The homoclinic invariant has some advantages over alternative 
measures for the separatrices splitting (see the review \cite{GL01}
for a discussion).

\begin{theorem}[Main Theorem]
There are constants $a_k\in\mathbb{R}$ such that the homoclinic invariant 
of each of the two primary homoclinic orbits has the form
\begin{equation}\label{Eq:mainexpansion}
\omega \asymp\pm \frac{2\pi }{\log \lambda_\delta ^{2}}%
{\mathrm e}^{-\frac{2\pi ^{2}}{\log \lambda_\delta }}
\sum_{k\ge 0}a_{k}\delta ^{2k}.
\end{equation}

\end{theorem}

This theorem implies that for a generic family the separatrices intersect
transversally. Indeed, we note that in (\ref{Eq:mainexpansion}) the pre-factor 
is exponentially small compared to $\delta$ due to the smallness 
of $\log\lambda_\delta$ estimated in~(\ref{Eq:lambda}).
Nevertheless for small $\delta>0$ the theorem implies 
transversality of the homoclinic points
provided $a_0\ne0$. We note that $a_0$ describes the splitting of complex
parabolic invariant manifolds
for function $f_0$ and does not vanish generically (on an open dense subset). 
In \cite{GS01} it was proved that $a_0\ne0$ for the H\'enon map.
Finally, numerical evaluation of $a_0$ for a given $f_0$ 
is a relatively easy task.

\medskip

The difficulty of this theorem is due to the exponential
smallness of the separatrices splitting which is to
be detected on the background of much larger effects.

\medskip

The proof consists of the following main steps:
\begin{enumerate}
\item We show that $f_\varepsilon$ can be formally represented as
a time-one map of a Hamiltonian flow and provide a normal form
theory for the flow.

\item We develop a general theory of close-to-identity maps which provides
sufficient condition for the existence of saddle points and their separatrices.
Moreover, we prove that the separatrices of the map can be quite accurately
approximated by separatrices of a flow which approximately interpolates the map.

\item We show that 
In the complexified time domain
the separatrix of the interpolating flow
is close to the separatrix of $f_\varepsilon$ in a set which has 
a non-empty intersection with a $\delta$-neighbourhood of~$\pm i\pi$.

\item
We show that in the latter region an approximation
based on the separatrices of the parabolic fixed point of $f_0$
provides an even more accurate approximation.
This approximation distinguishes between the stable and unstable
separatrices of $f_\varepsilon$.

\item
We show that flow-box time coordinates can be introduced.

\item
Flow-box coordinates are used to get an asymptotic expansion for the splitting
of separatrices.
\end{enumerate}

\medskip

The rest of the paper is dedicated to proving the main theorem
and is structured in the following way. Section~\ref{Se:overview}
contains a detailed overview of the proof leaving technical
estimates to corresponding sections. Section~\ref{Se:normalform}
describes the construction of the normal form for the bifurcation
up an arbitrary order. Section~\ref{Se:closetoid} develops
a general theory of close to identity maps, which can be
of independent interest since it covers a much wider class
of maps than that studied in other parts of this paper. In Section~\ref{Se:FS} we 
provide a detailed
description of a formal solution of the separatrix equation (\ref{Eq:mainfde}).

We will prove that $f_\varepsilon$ can be analytically
interpolated by an autonomous analytic flow and establish
properties for the corresponding energy-time coordinates.
Note that this property does not imply integrability of $f_\varepsilon$
since the domain in question is not invariant.
Several sections are included to provide details on approximation of $\psi^-(t)$
on various subsets of the complex plane.

\begin{figure}
\begin{center}
{\includegraphics[width=0.90\textheight,height=0.98\textwidth,angle=90]{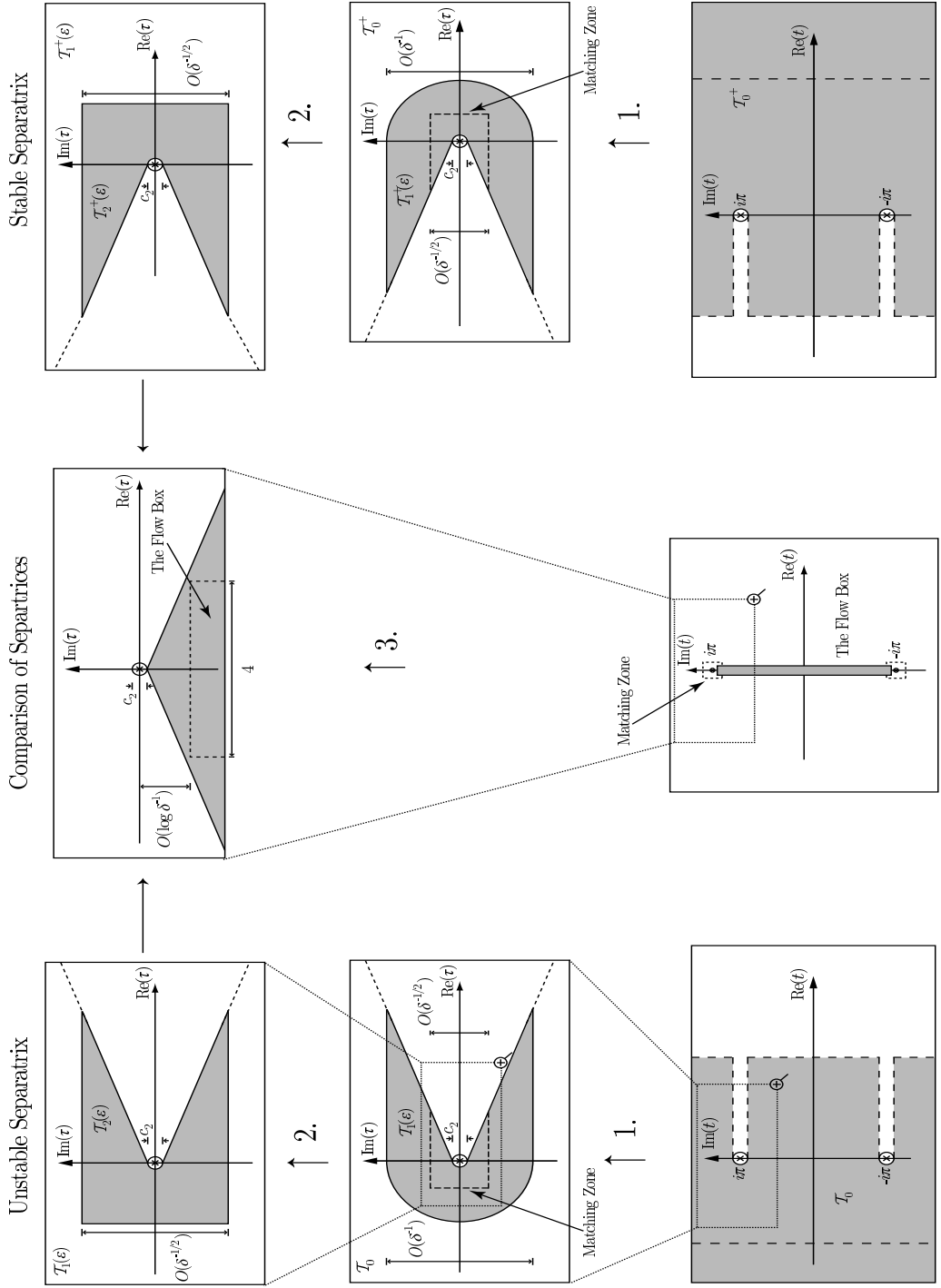}}
\end{center}
\caption{Overview of the proof: Domains of approximation}
\end{figure}

\section{Overview of the proof\label{Se:overview}}

In this section we will describe the main steps of the proof
postponing proofs of technical statements to subsequent sections.
Before proceeding to the proof we need to introduce some definitions.

In this paper we will use power series in $x,y$ and $\varepsilon$. 
The arguments are easier when terms of a similar magnitude are grouped
together. We consider $x$ to be of order 2, $y$ to be of
order 3 and $\varepsilon$ of order $4$. We say that a polynomial $Q_p$  
is quasi-homogeneous of order $p$ if it satisfies the identity
\begin{equation}\label{Eq:quasihomo}
Q_p(\lambda ^{2}x,\lambda ^{3}y,\lambda ^{4}\varepsilon )=\lambda
^{p}Q_p(x,y,\varepsilon )\,.
\end{equation}
Any power series can be represented as a sum of quasi-homogeneous polynomials.

In order to shorten the notation, we define
\begin{equation*}
\hat O_n:=O\left(\left(\max\{\,|x|^{1/2},|y|^{1/3},\varepsilon^{1/4}\,\}\right)^n\right)
\end{equation*}
and write $\hat O_{n,n+1}$ for a vector with components $(\hat O_{n},\hat O_{n+1})$.
It is easy to see
%
%
that $Q_p=\hat O_p$.

\subsection{Standard scaling and limit flow}

The change of the variables
\begin{equation}\label{Eq:standardscaling}
x=\delta ^{2}X,\quad y=\delta ^{3}Y,\quad \varepsilon =\delta ^{4}
\end{equation}
is called {\em the standard scaling}. In the new variables
the map $f_{\varepsilon }$ takes the form of a close to identity map
\begin{equation}\label{Eq:scaledmap}
F_{\delta }=\mathrm{id}+\delta G_{\delta }.
\end{equation}
We note that $G_{0}$ is called {\em the limit flow}. 
Its time-$\delta$ map $\Phi_{G_0}^\delta=\mathrm{id}+\delta G_0+O(\delta^2)$
and therefore is $O(\delta^2)$ close to $F_\delta$.

Taking into account that $f_0(0)=0$ and the equality (\ref{Eq:Df0})
we obtain
$$
G_0=\left(\begin{array}{c} Y\\-aX^2+b\end{array}\right)
$$
for some constants $a,b$. The same constants enter 
the non-degeneracy assumption (\ref{Eq:nondegeneracy}).
Indeed,  the limit flow is Hamiltonian with the Hamilton function
$$
H_{0}=\frac{Y^{2}}{2}+\frac{aX^{3}}{3}-bX
$$
which coincides, after reverting the standard scaling and scaling time, 
with (\ref{Eq:h6non}).
The dynamics of the limit flow is defined by the system of Hamiltonian equations:
\[
\dot X=Y,\qquad \dot Y=-a X^2+b.
\]
The limit flow has a saddle equilibrium at $(-\sqrt{\frac ba},0)$.
At this point the linearised vector-field has a positive eigenvalue 
$\mu_0=\sqrt[4]{4ab}$. 
The separatrix of the limit flow can be found explicitly by integrating the
Hamiltonian equations:
\[
X_0(t)=-\sqrt{\frac{b}{a}}+\frac{3\sqrt{ab}}{ \cosh^2\frac{\mu_0 t}{2}},
\qquad
Y_0(t)= -3\sqrt{ab}\mu_0\frac{\sinh\frac{\mu_0 t}{2}}{ \cosh^3\frac{\mu_0 t}{2}}\,.
\]
A substitution shows that these functions satisfy the Hamiltonian system.
Obviously they converge to the saddle equilibrium as $t\to\pm\infty$.

It is convenient to scale time and define the separatrix solution
of the limit flow by
\begin{equation}
\label{Eq:phi0}
\varphi_0(t)=(X_0,Y_0)(t/\mu_0)\,.
\end{equation}
This function satisfies the differential equation
\[
\dot \varphi_0=\mu_0^{-1}\mathrm{J}\nabla H_0(\varphi_0)\,.
\]
We will consider the separatrix solution for complex values of $t$.
We note that $\varphi_0$ is analytic for all complex $t$ except for 
second order poles at $i\pi (2k+1)$ with $k\in\mathbb Z$.
In particular this solution is analytic in the strip $|\Im(t)|<\pi$
and has poles at the points $t_*=\pm\pi i$ on its boundary.
These singularities play the central role in our analysis.

\subsection{Formal interpolation and Formal separatrix\label{Se:fs}}

We use $h_\varepsilon^n=\sum_{p= 6}^{n+5}h_{p}(x,y,\varepsilon )$
to denote a partial sum of a formal series 
\begin{equation}\label{Eq:formalHam}
h_{\varepsilon }=\sum_{p\geq 6}h_{p}(x,y,\varepsilon )
\end{equation}
where $h_{p}(x,y,\varepsilon)$ are quasi-homogeneous polynomials of order $p$.
In Section~\ref{Se:formalinterpol} we will prove that 
there exists a unique formal series such that for each $n\in\mathbb N$
\begin{equation}\label{Eq:errorfi}
f_{\varepsilon }=\Phi _{h_{\varepsilon }^n}^{1}+\hat O_{n+3,n+4}\,,
\end{equation}
where $\Phi _{h_{\varepsilon }^n}^{1}$ is the
time one map with hamiltonian $h_{\varepsilon }^n$. 
The leading order of the series is given by (\ref{Eq:h6non}).
It is a quasi-homogeneous polynomial of order~$6$ and coincides,
up to scaling of the space and time variables, with the Hamiltonian
of the limit flow.

We say that the series $h_\varepsilon$ {\em formally interpolates\/} the map $f_\varepsilon$.
We should not expect the series $h_\varepsilon$ to converge.

\begin{theorem}[Formal separatrix]\label{Thm:fs}
If $h_\varepsilon$ is a formal Hamiltonian of the form\/ {\rm (\ref{Eq:formalHam})}
and its leading order satisfies the non-degeneracy condition {\rm (\ref{Eq:nondegeneracy})},
then there is a unique formal series 
\[
\mu(\delta)=\sum_{k\ge0}\mu_k\delta^{2k+1}
\]
such that the equation
\begin{equation}\label{Eq:mainforma}
\mu(\delta)\dot {\mathbf{X}}=\mathrm{J}\nabla h_\varepsilon(\mathbf X)
\end{equation}
has a formal solution of the form
\begin{equation}\label{Sum_FormalSep}
\mathbf{X}(t,\varepsilon )=\left( 
\begin{array}{c}
\sum\limits_{p\geq 1}\delta ^{2p}X_{p}^{1}+\dot \eta _{0}\sum\limits_{p\geq
1}\delta ^{2p+1}X_{p-1}^{2}\bigskip  \\ 
\dot\eta _{0}\sum\limits_{p\geq 1}\delta
^{2p+1}Y_{p-1}^{2}+\sum\limits_{p\geq 2}\delta ^{2p}Y_{p}^{1}%
\end{array}%
\right)  
\end{equation}
where $X_{p}^{1},Y_{p}^{1},X_{p}^{2},Y_{p}^{2}$ are polynomials of order $p$
in $\eta _{0}=\cosh^{-2}\frac t2$.
\end{theorem} 
The proof of this theorem is placed in Section~\ref{Se:FS}.
We remind that $\delta=\varepsilon^{1/4}$.

%

We say that $\mathbf{X}$ is a formal separatrix.
The series $h_\varepsilon$ may diverge and consequently
there is no reason to expect that the formal separatrix converges.
On the other hand we will see that its partial sums provide 
a rather accurate approximation for the separatrices of the map $f_\varepsilon$.

Let us consider the domain
\begin{equation}
\mathcal{T}_{0}=\left\{\; t\in \mathbb{C} \;:\;
\Re(t)\leq r_1
\ \mbox{ and }
\ \varphi _{0}(t-s)\in \mathcal{D},\ \forall s\geq 0
\;\right\} 
\end{equation}
where $r_1>0$ is an arbitrary constant and $\mathcal D$ is a bounded
domain. Later we will assume it to be a sufficiently large ball centred around the origin. 
Of course, some constants in future estimates will depend on the choice of~$r_1$
and $\mathcal D$.

\begin{lemma}\label{Le:sepint}
For any $n\in\mathbb N$, the Hamiltonian $h_{\delta }^{n}$
has a saddle equilibrium with a Lyapunov exponent $\mu _{n,\delta }>0$
in a neighbourhood of the origin.
Moreover, there exists a separatrix solution $\varphi _{\delta}^{n}$ 
of the equation
\begin{equation}\label{Eq:sepintflow}
\dot{\varphi}_\delta^{n}=\mu _{n,\delta }^{-1}
\left. \mathrm{J}\nabla h_{\delta }^{n}\right\vert _{\varphi _{\delta }^{n}}
\end{equation}
such that uniformly for $t\in\mathcal{T}_{0}$
\begin{equation*}
\varphi _{\delta }^{n}(t)=\mathbf{X}^n(t,\delta)+O(\delta ^{n})\,,
\end{equation*}
where $\mathbf{X}^n$ denotes the sum of the first $n$ orders in $\delta$ 
of the formal series {\rm (\ref{Sum_FormalSep})}.
\end{lemma}

\begin{proof} The proof is straightforward: this lemma is a statement 
about polynomial vector fields and after the standard scaling
the saddle of the equation is non-degenerate.
\end{proof}

Let $\mathbf{X}$ be the formal separatrix given by $\left( \ref%
{Sum_FormalSep}\right) $. We will study it close to the singularity by
substituting $t=i\pi +\tau \log \lambda_\delta $ and expanding $X_{k}^{1,2},Y_{k}^{1,2}$ into
their Laurent series. We also re-expand $\log \lambda_\delta =\mu(\delta)$.
The result is a formal series of the form
\begin{equation}\label{Eq:formal_laurent}
\mathbf{X}=\left( \sum_{m\geq 0}\delta ^{2m}\tau ^{2m-2}\sum_{k\geq
0}x_{mk}\tau ^{-k},\sum_{m\geq 0}\delta ^{2m}\tau ^{2m-3}\sum_{k\geq
0}y_{mk}\tau ^{-k}\right),  
\end{equation}%
where $x_{mk}$ and $y_{mk}$ are real coefficients. See Section~\ref{Se:laurentexp} 
for a derivation of (\ref{Eq:formal_laurent}). It is convenient
to introduce formal power series in $\tau$ by setting
\begin{equation}\label{Eq:psim_hat}
\hat\psi_m=\left( \tau ^{2m-2}\sum_{k\geq 0}x_{mk}\tau ^{-k},
\tau ^{2m-3}\sum_{k\geq 0}y_{mk}\tau ^{-k}\right)
\end{equation}
for $m\ge0$. Then
\begin{equation}\label{Sum_AssymptoticBC}
\mathbf{X}=\sum_{m\geq 0}\delta ^{2m}\hat \psi_m(\tau)\,.
\end{equation}
We will use these formal series in the complex matching
procedure described in Section~\ref{Se:complmatch}.

\subsection{Separatrices for close-to-identity maps\label{Se:approx1}}

Let us consider a family of close to identity maps (\ref{Eq:scaledmap}).
For the purpose of this section it is not necessary to assume that the map
is obtained as a result of the standard scaling.

The implicit function theorem implies that if the limit flow $G_{0}$ 
has a non-degenerate equilibrium then $F_{\delta }$ has an
analytic family of fixed points which tend to the equilibrium when $\delta
\rightarrow 0$.
If the equilibrium is a saddle then the fixed point of $F_\delta$ is also a saddle
(see Section~\ref{Se:fpm} for formal statements and proofs).

The separatrix of $F_\delta$ is close to the separatrix of the limit flow (see \cite{FS1990}).
We will use a more accurate approximation provided by a separatrix of an (approximately) 
interpolating flow. Let us state it more formally.

Let $\mathcal D\subset\mathbb C^2$ be a bounded domain
and assume the origin is inside $\mathcal D$.

\begin{theorem}[Approximation theorem]\label{Thm_LAT}
Let $F_{\delta }$ be an analytic family of area preserving maps 
such that the limit flow has a saddle equilibrium at the origin.
Let $H_{\delta }^{n}$ be a Hamiltonian function such
that 
\begin{equation}\label{Eq:closenessFPhi}
F_{\delta }=\Phi _{H_{\delta }^{n}}^{\delta}+O(\delta ^{n+1})
\end{equation}
on $\mathcal D$ and $\phi _{\delta }^{n}$ be a separatrix solution
of the equation
\[
\dot \phi _{\delta }^{n}=\mu_{n,\delta}^{-1}\mathrm{J}\nabla H_\delta^n
\]
such that
\[
\phi _{\delta }^{n}(t)=\varphi_0(t)+O(\delta)
\]
for all $t\in\mathcal T_0$.
Then there exists a parametrisation $\Psi^{-}$ of the local unstable 
separatrix  of the map $F_\delta$ which satisfies the equation
\begin{equation}
\Psi ^{-}(t+\log \lambda _{\delta })=F_{\delta }(\Psi ^{-}(t))
\end{equation} 
and 
\begin{equation}\label{Eq:nonsingestimate}
\Psi ^{-}(t)=\phi _{\delta }^{n}(t)+O(\delta ^{n})
\end{equation} 
uniformly on the set $\mathcal T_0$.
\end{theorem}

The proof of this theorem consists of two steps. First the theorem
is proved for the local separatrix which corresponds to
a half-plane $\Re t<-r_0$ as described in Section~\ref{Se:prooflocalsep}.
Then an extension lemma (Lemma~\ref{LemmaExtension}) stated
and proved in Section~\ref{Se:extension} is applied.

For completeness, we also provide a proof of the existence 
for the Hamiltonian $H_\delta^n$ in Section~\ref{Se:formint}.

\subsection{Parametrisation the separatrices\label{Se:ext}}

We note that the existence of a parameterisation which satisfies equation
(\ref{Eq:mainfde}) follows from the general theory \cite{GL01}.
The parametrization is not unique and is defined up to a translation 
$t\mapsto t+t_0(\varepsilon)$.
The next theorem states that there is a parameterisation which
is close to the separatrix of the interpolating flow.

Let $\varphi _{\delta }^{n}$ be defined by Lemma~\ref{Le:sepint}. 

\begin{theorem}[Existence and Local approximation]\label{Thm:psim}
Equation (\ref{Eq:mainfde}) has a solution
\begin{equation}\label{Eq:psimappr}
\psi ^{-}(t)=\varphi _{\delta }^{n}(t)+O_{n+2,n+3}(\delta )
\end{equation} 
uniformly for $t\in\mathcal T_0$.
\end{theorem}

\begin{proof}
The theory of the previous section applies to the map $F_\delta$ obtained
after the standard scaling of the map $f_\varepsilon$. 
Let $\mathcal D$ be a large ball of a fixed radius. 
Since the map $f_\varepsilon$ is defined in an $\varepsilon$ independent
neighbourhood of the origin, the domain of the scaled map $F_\delta$ 
contains the ball $\mathcal D$ provided $\varepsilon$ is sufficiently
small.

The limit flow separatrix $\varphi_0$ is given explicitly by (\ref{Eq:phi0}) 
and has poles at $i\pi (2k+1)$ for each $k\in\mathbb Z$. 
Therefore the set $\mathcal T_0$ takes the shape
similar to the one shown on Figure~\ref{Fig:T0u}: it is a half plane
$\Re t<r_1$ without ``shadows" of small disks centred around
the singular points of $\varphi_0$.

\begin{figure}
\begin{center}
{\includegraphics[width=0.7\textwidth]{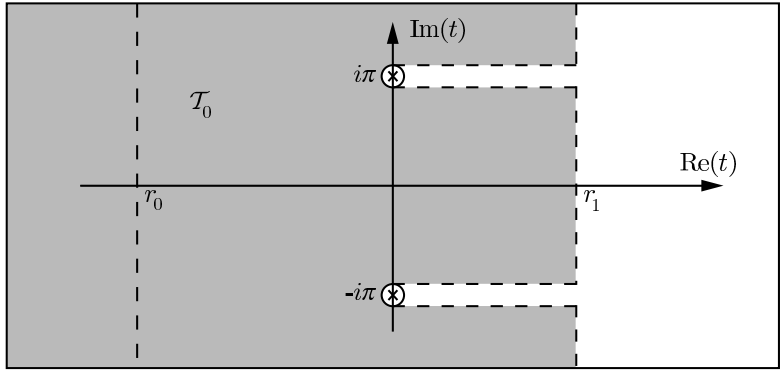}}
\end{center}
\caption{Domain $\mathcal T_0$.
\label{Fig:T0u}
}
\end{figure}

The interpolating Hamiltonian is obtained from (\ref{Eq:formalHam})
after the change of variables (\ref{Eq:standardscaling}).
Although the change of variables is not symplectic it has a constant
Jacobian, and therefore the Hamiltonian flow of $h_\delta^n$ is transformed into
a Hamiltonian flow with the Hamiltonian function
\begin{equation}\label{Eq:Hdnp}
H_{\delta }^{n}=\sum_{k=1}^{n}\delta ^{k}h_{5+k}(X,Y,1).
\end{equation}
Indeed, under the standard scaling the symplectic form becomes 
$dx\wedge dy=\delta ^{5}dX\wedge dY$, to work with the standard symplectic
form in the scaled variables we scale the Hamiltonian by $\delta ^{-5}$.
The upper bound (\ref{Eq:closenessFPhi}), which is necessary to apply
Theorems~\ref{Thm_LAT}, follows from (\ref{Eq:errorfi}) 
and~(\ref{Eq:standardscaling}).

To complete the proof we reverse the standard scaling which transforms the
error term $O(\delta^{n})$ of equation (\ref{Eq:nonsingestimate}) 
into $O_{n+2,n+3}(\delta )$.
\end{proof}

\subsection{Extension towards the singularity}

The theory presented in the previous subsection provides 
a rather accurate approximation for the invariant manifolds.
However, it is not sufficient to distinguish
between the stable and unstable separatrices of the map $F_\delta$.
Instead of further improving  the accuracy, we will study
the separatrices for values of $t$  closer to the singular 
points of the limit flow separatrix. In this region the splitting of the
separatrices is larger and consequently easier to detect.
This extension is sensitive to the form of the map
and cannot be performed with the same level of generality
as the estimates of Section~\ref{Se:approx1}.
So from now on we restrict our consideration to 
our original map $f_\varepsilon$. 

Let us fix two small positive constants $c_1$ and $\beta$,
and consider a domain $\mathcal{T}_{1}(\varepsilon )$
shown on Figure~\ref{Fig:T1u} (left), which also depends on a constant $c_2>0$.
The constant $c_2$ is to be chosen sufficiently large to ensure that
$\varphi_0$, the separatrix of the limit flow written in the unscaled
variables, does not leave a small $\varepsilon$-independent 
ball centred around the origin.

It is convenient to work with a time parameter centred at the
singularity and to scale the time step to one, hence we let 
\begin{equation}\label{Eq:tau}
\tau=\frac{t-i\pi}{\log \lambda _{\delta }}.
\end{equation}
The domain $\mathcal T_1(\varepsilon)$ expressed in the terms of $\tau$ 
is shown on Figure~\ref{Fig:T1u} (right).

\begin{figure}
\begin{center}
\includegraphics[width=0.45\textwidth]{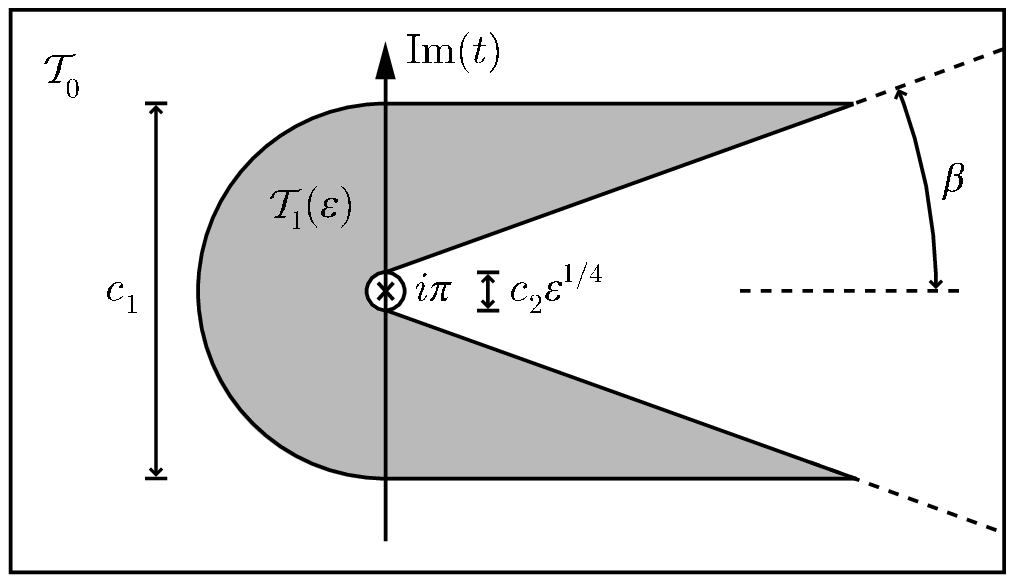}\kern0.08\textwidth
\includegraphics[width=0.45\textwidth]{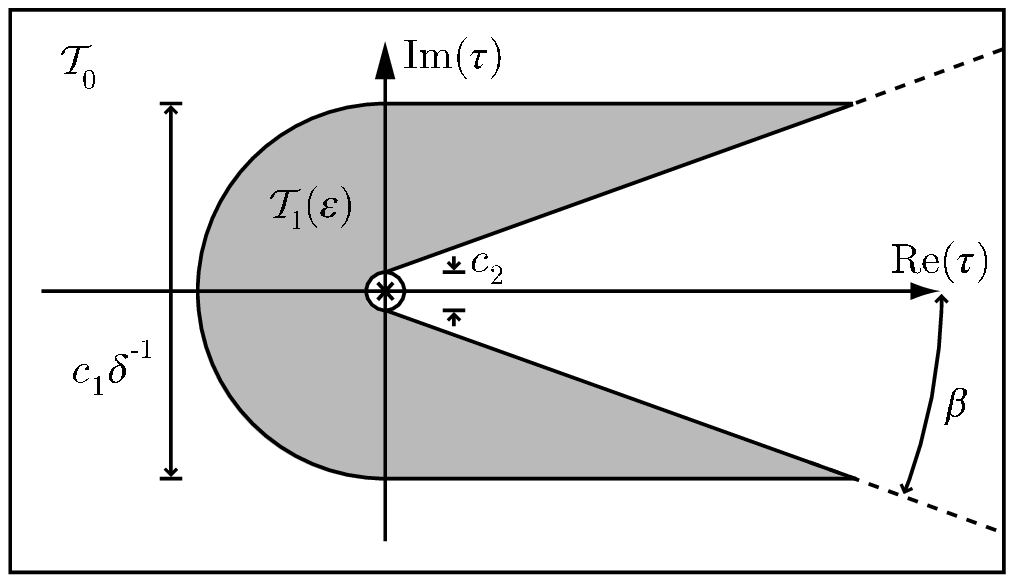}
\end{center}
\caption{An amplification of a neighbourhood of $t=i\pi$ shows the domain  
$\mathcal T_1(\varepsilon)$ in which the unstable separatrix
is still reasonably well approximated by the separatrix of the limit flow
(left). The same domain on the plane of~$\tau$ (right).
\label{Fig:T1u}}
\end{figure}

\begin{theorem}
\label{Thm:FirstApprox_T1}
Let $\psi^-$ be defined by Theorem~\ref{Thm:psim}.
There is $c_2>0$ such that
\begin{equation}
\psi ^{-}(\tau)=\varphi _{\delta }^{n}(\tau)+O_{n,n+1}(\tau ^{-1})
\label{Sum_Approx_in_T1}
\end{equation}
uniformly for all $\tau \in \mathcal{T}_{1}(\varepsilon )$ and $\varepsilon\in(0,\varepsilon_0)$.
\end{theorem}

We note that $t=i\pi+\tau\log\lambda_\delta$ and
we have overloaded notation by writing $\psi^-(\tau)$ 
instead of $\psi^-(i\pi+\tau\log\lambda_\delta)=\psi^-(t)$.

The estimate\/ {\rm (\ref{Sum_Approx_in_T1})} is uniform which means that
the constant in the term $O_{n,n+1}(\tau ^{-1})$ is independent from
both $\varepsilon$ and $\tau$. It is important because
the size of $\mathcal{T}_{1}(\varepsilon )$ increases
as $\varepsilon\to0$
and includes $\tau$ up to the order of $\delta^{-1}=\varepsilon^{-1/4}$. 
We note that in Theorem \ref{Thm:FirstApprox_T1}
the error of the approximation is of order $O(\delta ^{n})$ on the left boundary of 
$\mathcal{T}_{1}(\varepsilon )$ but it is just of order $O(1)$ on the boundary
near the central circle. 

Therefore the theorem suggests that the distance between the separatrix
of the map and separatrix of the interpolating flow may
gradually increase as the parameter $t$ comes closer to $i\pi$.

\subsection{Complex matching\label{Se:complmatch}}

In this section we construct another approximation for $\psi^-$
which provides higher accuracy in the central part of $\mathcal T_1(\varepsilon)$.
Let us start by looking for a formal solution
\begin{equation}
\hat\psi(\tau )=\sum_{k\geq 0}\delta ^{2k}\psi _{k}(\tau )
\label{Sum_MatchingSeries}
\end{equation}%
of the separatrix equation%
\begin{equation}
\hat\psi(\tau +1)=f_{\varepsilon }(\hat\psi (\tau )).
\label{Sum_MatchingEqn}
\end{equation}%
We assume that the new time $\tau$ is related to the original $t$ by (\ref{Eq:tau}).
We formally substitute the series (\ref{Sum_MatchingSeries})
into the equation (\ref{Sum_MatchingEqn}) and expand the right hand side
in powers of $\delta^2$. Collecting terms of equal order in $\delta$
we get at the leading order
\begin{equation}
\psi _{0}(\tau +1) =f_{0}(\psi _{0}(\tau )), 
\label{Eq:psi0} 
\end{equation}
and the following system of equations for all other orders:
\begin{eqnarray}
\psi _{1}(\tau +1) &=&Df_{0}(\psi _{0}(\tau ))\psi _{1}(\tau ),
 \notag 
\\
\psi _{2}(\tau +1) &=&Df_{0}(\psi _{0}(\tau ))\psi _{2}(\tau )+R_{2}(\psi_{0},\psi _{1}), 
\notag \\[6pt]
&&\vdots   \notag
\end{eqnarray}
In general, the equation corresponding to $\delta ^{2m}$ with $m\geq 2$ can be
written as%
\begin{equation}\label{Eq:psim}
\psi _{m}(\tau +1)=Df_{0}(\psi _{0})\psi _{m}(\tau )+R_{m}(\psi _{0},\psi
_{1},\ldots ,\psi _{m-1}),
\end{equation}
where $R_{m}$ is a polynomial in $\left( \psi _{1},\ldots ,\psi_{m-1}\right) $. 
We stress that the dependence on $\psi _{0}$ is not necessarily polynomial. 
For example, \[R_{2}=\psi_1^t D^{2}f_{0}(\psi _{0})\psi _{1}+f_{1}(\psi _{0})\]
where ${}^t$ is used to denote the transposition and $f_1$ is a coefficient of the
Taylor series $f_\varepsilon=\sum_{k\ge0}\varepsilon^kf_k$.

An analytic solution $\psi^-_m$ of the equation (\ref{Eq:psim}) can be
selected by the following asymptotic condition:
\begin{equation}\label{Eq:psim'}
\psi^-_m(\tau)\asymp \hat\psi_m(\tau)\qquad \mbox{
as $\tau\to\infty$ in the sector  $\beta_0<\arg \tau<2\pi-\beta_0$}\,.
\end{equation}
The formal series $\hat\psi_m(\tau)$ comes from the re-expansion of the formal separatrix
given by (\ref{Sum_AssymptoticBC}). In fact, $\hat\psi_m$ is a formal solution of 
equation (\ref{Eq:psim}).

The method of fixing a solution in one region by initial or boundary conditions
which come from a neighbouring region is known as {\em ``complex matching"}.

\begin{figure}
\begin{center}
\centerline{\includegraphics[width=0.6\textwidth]{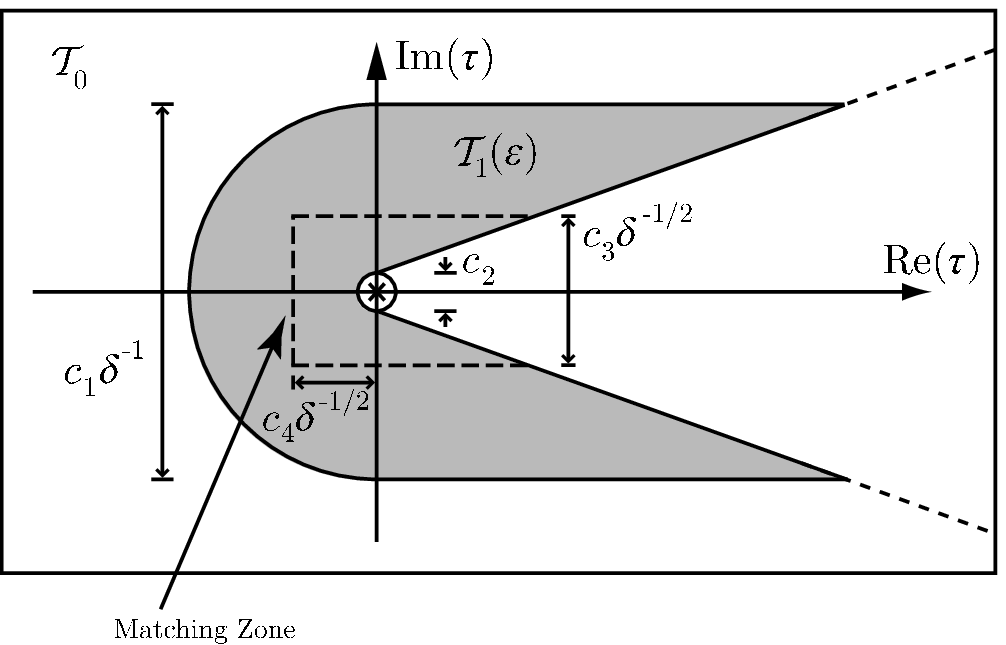}}
\end{center}
\caption{Complex matching.}
\end{figure}

We will now continue with the approximation inside the matching zone, we
call this domain $\mathcal{T}_{2}(\varepsilon )$:
\begin{equation*}
\mathcal{T}_{2}(\varepsilon )=\left\{ \tau \in \mathcal{T}_{1}(\varepsilon
):\func{Re}\tau >-\delta ^{1/2},\,-\delta ^{1/2}<\func{Im}\tau <\delta
^{1/2}\right\} .
\end{equation*}

\begin{figure}
\begin{center}
{\includegraphics{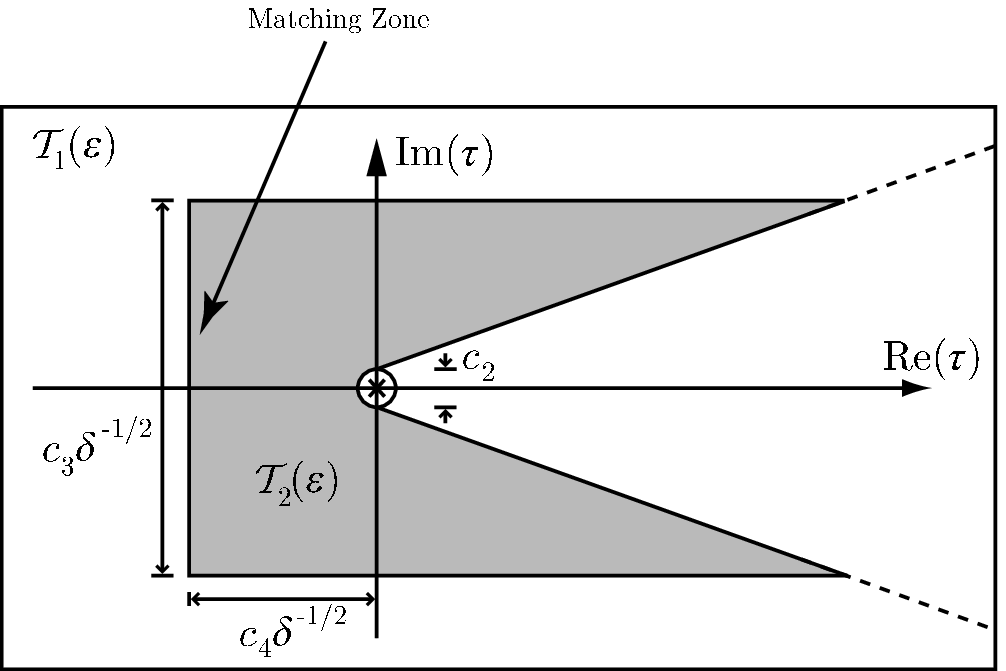}}
\end{center}
\caption{Complex matching.}
\end{figure}

\begin{theorem}[Second approximation theorem]\label{Thm:secap}
\begin{equation*}
\psi ^{-}(\tau )=\sum_{k=0}^{n-1}\delta^{2k} \psi_k^-(\tau)+O(\delta^{n})
\end{equation*}
uniformly in $\mathcal{T}_{2}(\varepsilon )$.
\end{theorem}


We note that this theorem implies the following.

\begin{corollary}
\label{Cor:unstappr}
There is a parametrisation of the unstable separatrix such that 
\begin{equation}\label{Eq:unstappr}
\psi ^{-}(\tau )=\sum_{k=0}^{n-1}\delta ^{2k}\psi _{k}^-(\tau )
+O((\delta
\tau )^{2n}).
\end{equation}
uniformly in $\mathcal{T}_{2}(\varepsilon )$.
\end{corollary}

\begin{proof}
The asymptotic boundary conditions implies $\psi^-_k(\tau)=O(\tau^{2k})$.
Then the theorem with $n$ replaced by $2n$ implies
\begin{eqnarray*}
\psi ^{-}(\tau ) &=&\sum_{k=0}^{2n-1}\delta ^{2k}\psi _{k}^-(\tau )+O(\delta
^{2n}) \\
&=&\sum_{k=0}^{n-1}\delta ^{2k}\psi _{k}^-(\tau )+O((\delta\tau) ^{2n}).
\end{eqnarray*}
\end{proof}

\subsection{The stable manifold}

So far we have only considered the unstable separatrix.
In order to estimate the splitting of separatrices 
we also need to study the stable one. 
We note that the stable manifold is an unstable one for the inverse map.

Therefore the study of the stable separatrix is analogous to the
study of the unstable one but with time reversed.
The transformation $t\mapsto-t$ swaps the upper and lower
half-planes of $\mathbb C$ and it is convenient
to use real-analytic symmetry to swap these back.
Let $\mathcal{T}_{0}^+,\mathcal{T}_{1}^+(\varepsilon ),\mathcal{T}_{2}^+(\varepsilon )\subset\mathbb C$
denote domains obtained by reflecting
$\mathcal{T}_{0},\mathcal{T}_{1}(\varepsilon ),\mathcal{T}_{2}(\varepsilon )$
in the imaginary axis respectively. The limit flow and approximately
interpolating flows are integrable and their stable and unstable
separatrices coincide. Therefore the same separatrix solution arise
and the same formal expansion (\ref{Sum_AssymptoticBC}) 
is used for the complex matching of the stable separatrix. 
However, the matching is done on the right hand side of the
singularity with asymptotic boundary conditions as $\tau \rightarrow
+\infty $.

\begin{figure}
\begin{center}
{\includegraphics[width=0.7\textwidth]{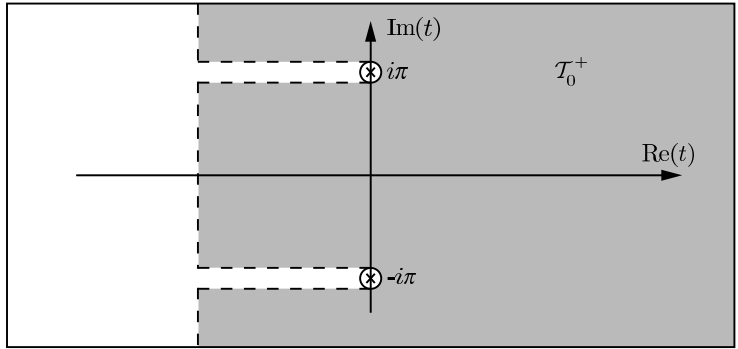}}
\end{center}
\caption{The shadowed region represents the domain of validity 
for the local approximation theorem for the stable separatrix.
\label{Fig:StableT0}}
\end{figure}

Similarly to Theorem~\ref{Thm:psim}, we conclude that for every $n\in\mathbb N$
there is a solution $\psi^+$ of the equation (\ref{Eq:mainfde}) such that 
\begin{equation}\label{Eq:psipappr}
\psi^{+}(t)=\varphi_\delta^{2n}(t)+O_{2n+2,2n+3}(\delta)
\end{equation}
uniformly for $t\in\mathcal T_0^{+}$ (see Figure~\ref{Fig:StableT0}).
Note that we directly started with $n$ replaced by $2n$.
For this solution, instead of (\ref{Sum_Approx_in_T1}) we get
\begin{equation}\label{Sum_Approx_in_T1p}
\psi^{+}(t)=\varphi_\delta^{2n}(t)+O_{2n,2n+1}(\tau^{-1})
\end{equation}
uniformly in $\mathcal T_1^+(\varepsilon)$ and $\varepsilon\in(0,\varepsilon_0)$
and instead of Corollary~\ref{Cor:unstappr} we get
\begin{corollary} \label{Cor:stappr}
Uniformly in $\mathcal T_2^+(\varepsilon)$ 
\begin{equation}\label{Eq:stappr}
\psi ^{+}(\tau ) =\sum_{k=0}^{n-1} \delta^{2k}\psi ^{+}_k(\tau )+O((\delta
\tau )^{2n}) 
\end{equation}%
where $\psi _{k}^{+}(\tau )$ is the solution of the asymptotic boundary
condition problem\/ {\rm (\ref{Eq:psi0})} or\/ {\rm (\ref{Eq:psim})}
subject to the asymptotic boundary conditions 
\[
\psi^+_k(\tau)\asymp \hat\psi_k(\tau)\qquad \mbox{
as $\tau\to\infty$ in the sector  $-\pi+\beta_0<\arg \tau<\pi-\beta_0$}\,.
\]
\end{corollary}

We note that in general $\psi _{0}^{+}\neq \psi _{0}^{-}$~\cite{Gel00,GN2008}.

\subsection{Upper bounds for the splitting of separatrices}

We assume that both $\psi^-$ and $\psi^+$ are chosen to be closed to
$\varphi^{2n}_\delta$. Then comparing the estimates (\ref{Eq:psimappr})
and (\ref{Eq:psipappr})
\[
\psi^-(t)-\psi^+(t)=O(\delta^{2n})
\]
uniformly in $\mathcal T_0\cap \mathcal T_0^+$.
Then we use the estimates (\ref{Sum_Approx_in_T1p}) (with $n$ replaced by $2n$) 
and (\ref{Sum_Approx_in_T1p}) to conclude
\[
\psi^-(t)-\psi^+(t)=O_{2n,2n+1}(\tau^{-1})
\]
uniformly in $\mathcal T_1(\varepsilon)\cap \mathcal T^+_1(\varepsilon)$.
Taking into account the difinitions of the domains
we see that the union of validity domains of these two estimates
includes the rectangle defined by the inequalities
$|\Im{t}|<\pi-\delta^{1/2}$ and $|\Re{t}|<\delta^{1/2}$
and for these values of $t$
\begin{equation}
\psi^+(t)-\psi^-(t)=O(\tau^{-n})=O(\delta^{n})\,.
\end{equation}
In the innermost region $\mathcal T_2(\varepsilon)$, closer to the singularity at $i\pi$, 
we use a different argument. In this region the estimate has to be
be close to optimal: we will use the upper bound 
to show that the square of the distance 
between the stable and unstable separatrices
is negligible
compared to the leading term of the splitting 
(on a line with $\Im\tau=-\sigma\log\delta$).

The upper bound for the splitting near the singularity 
is obtained in the following way:
by Corollaries \ref{Cor:unstappr} and \ref{Cor:stappr}
\begin{equation}\label{Eq:dpsi'}
\psi^+(\tau)-\psi^-(\tau)=\sum_{k=0}^{n-1}
\delta^{2k}\left(\psi_k^+(\tau)-\psi_k^-(\tau)\right) + 
O\left((\delta\tau)^{2n}\right)
\end{equation}
uniformly in $\mathcal T_2(\varepsilon)\cap\mathcal T_2^+(\varepsilon)$.
We use that
\begin{equation}\label{Eq:dpsik}
\psi_k^+(\tau)-\psi_k^-(\tau)=O(\tau^{m_k}e^{-2\pi i\tau})
\end{equation}
provided $-\pi+\beta<\arg\tau<-\beta$ and $\Im\tau<-c_2$.%
Consequently
\begin{equation}\label{Eq:dpsi}
\psi^+(\tau)-\psi^-(\tau)=
O(\tau^4e^{-2\pi i\tau})+O((\delta\tau)^{2n})
\end{equation}
provided $-c_2>\Im\tau>-\delta^{-1/2}$ and $-\pi+\beta<\arg\tau<-\beta$.

\subsection{The flow box}

In order to provide a quantitative description
for the difference between $\psi^+$ and $\psi^-$ 
we use the method based on flow box coordinates.
We will prove a theorem which essentially says 
that the restriction of the map $f_\varepsilon$
on a non-invariant subset can be analytically 
interpolated by an autonomous 
Hamiltonian flow with one degree of freedom. 
The subset in question contains large (but again non-invariant) 
segments of the stable and unstable separatrices.
Simultaneously we introduce ``energy-time'' coordinates $(T,E)$ 
for the interpolating flow.

Consider a domain in $\mathbb C^2$ defined by
\begin{equation}
\mathcal U_\delta=\mathcal{T} \times  \mathcal{E} =
\left\{\,
\left\vert \func{Re}(T)\right\vert <2,\ \left\vert \func{Im}(T)
\right\vert <\frac{\pi}{\log\lambda_\delta}-c_3
\,\right\} 
\times 
\left\{\,  \left\vert E\right\vert
<E_{0}(\delta )
\,
\right\} .  \label{DomainS}
\end{equation}%

\begin{figure}
\begin{center}
{\includegraphics[width=0.65\textwidth]{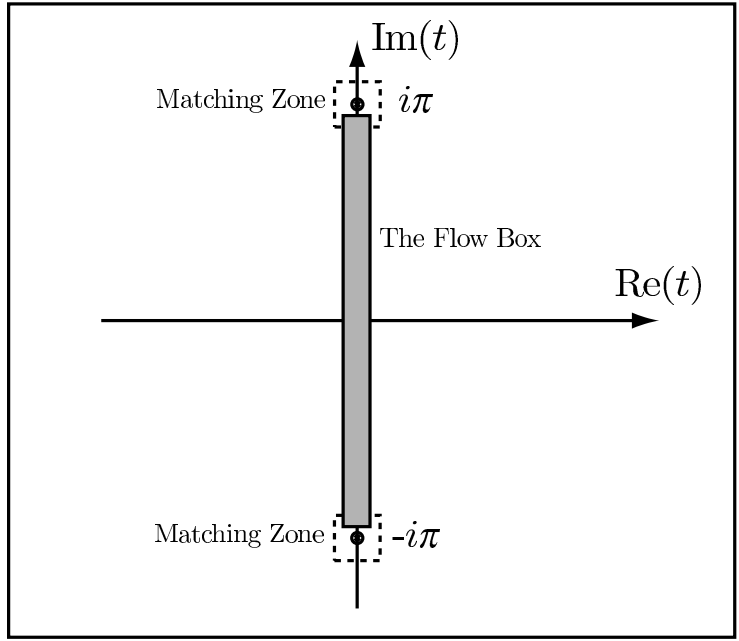}}
\end{center}
\caption{Time component of the flow box coordinates 
($t:=T\log\lambda_\delta$).}
\end{figure}

\begin{theorem}[The flow box theorem]\label{Thm:flowbox} 
If $c_3$ is sufficiently large, there exists a symplectic diffeomorphism $S:\mathcal U_\delta\to
\mathcal V_\delta:=S(\mathcal U_\delta)$ such that
\begin{enumerate}
\item
$\left.f_{\varepsilon }\right|_{\mathcal V_\delta}=\Phi _{E}^{1}$, 
where $E:\mathcal V_\delta\to\mathbb C$ is the second component of the map~$S^{-1}$.
\item
$S(T,0)= \psi ^{-}(T\log \lambda_\delta )$ \ \ (Normalisation).
\item
$\left\Vert S^{-1}\right\Vert _{C^{2}}$ is bounded uniformly for $\varepsilon\in(0,\varepsilon_0)$.
\end{enumerate}
\end{theorem}


\medskip

It is convenient to consider $S^{-1}$ as a 
symplectic coordinate map defined on $\mathcal V_\delta$.
In coordinates $(T,E)$, the Hamiltonian equations
with the Hamiltonian function $E$ take the form
\[
T'=1\,,\qquad E'=0
\]
and can be easily integrated. The corresponding time-one map 
has the form $(T,E)\mapsto(T+1,E)$. It coincides with
$S^{-1}\circ \Phi^1_E\circ S $. In the original coordinates the map 
$\Phi_E^1=f_\varepsilon$. Consequently the first property of $S$
means that $S$ conjugates $f_\varepsilon$ and the translation:
\begin{equation}
(T+1,E)=S^{-1}\circ f_{\varepsilon }\circ S(T,E).  \label{EqnS}
\end{equation}
Let us use $(T(x,y),E(x,y))$ for the components of $S^{-1}$
and $(x(T,E),y(T,E))$ for the components of $S$. 
Equation (\ref{EqnS}) implies
\begin{eqnarray}
T\circ f_\varepsilon(x,y)&=&T(x,y)+1\,,\label{Eq:T1}
\\
E\circ f_\varepsilon(x,y)&=&E(x,y)\,.\label{Eq:E1}
\end{eqnarray}
By the inverse function theorem, the Jacobian matrix $DS^{-1}=(DS)^{-1}$.
In coordinate form this equality implies that
\begin{equation*}
\left( 
\begin{array}{cc}
\frac{\partial T}{\partial x} & \frac{\partial T}{\partial y} \\[6pt] 
\frac{\partial E}{\partial x} & \frac{\partial E}{\partial y}%
\end{array}%
\right) =\left( 
\begin{array}{cc}
\frac{\partial x}{\partial T} & \frac{\partial x}{\partial E} \\[6pt] 
\frac{\partial y}{\partial T} & \frac{\partial y}{\partial E}
\end{array}%
\right) ^{-1}=\left( 
\begin{array}{cc}
\frac{\partial y}{\partial E} & -\frac{\partial x}{\partial E} \\[6pt] 
-\frac{\partial y}{\partial T} & \frac{\partial x}{\partial T}%
\end{array}%
\right) ,
\end{equation*}
where we used that the Jacobian matrix has the unit determinant
to get the second equality.
The second row gives the following Hamiltonian equations:
\[
\frac{\partial y}{\partial T}
=-
\frac{\partial E}{\partial x}
\qquad\mbox{and}\qquad
\frac{\partial x}{\partial T}
=
\frac{\partial E}{\partial y}\,.
\]
Substituting $E=0$ into the equations,
using the normalisation condition in the form
\[
\left(x(T,0),y(T,0)\right) =\psi ^{-}(T\log\lambda_\delta)
=(\psi_1 ^{-}(T\log\lambda_\delta),\psi_2 ^{-}(T\log\lambda_\delta))
\]
and switching to derivation with respect to $t$ 
(where $t=T\log\lambda_\delta$) 
we get
\begin{eqnarray*}
\left.
 \frac{\partial E}{\partial x}
\right\vert _{\psi ^{-}}&=&
-\log\lambda_\delta\,
\dot \psi _{2}^{-}\,,
\\
\left.
 \frac{\partial E}{\partial y}
\right\vert _{\psi ^{-}}&=&
\log\lambda_\delta\,
\dot \psi _{1}^{-}\,.
\end{eqnarray*}
The last equation has the form of a Hamiltonian system with the Hamiltonian $E(x,y)$
after a time rescaling. For future reference we write this identity
using the vector notation:
\begin{equation}\label{Eq:nablaE}
\nabla E(\psi^-)=-
\log\lambda_\delta \,{\mathrm J}\dot\psi^-
\end{equation}
where ${\mathrm J}=\left(\begin{array}{rr} 0&1\\-1&0\end{array}\right)$
is the standard symplectic matrix.

\begin{remark} The following facts are not directly used in the proof.
It is easy to see that equation (\ref{EqnS}) alone
does not define $S$ uniquely. Indeed, 
consider a substitution $(T,E)\mapsto (\widetilde{T},\widetilde{E})$
of the form
\begin{equation*}
(\widetilde{T},\widetilde{E})=(T+a(T,E),b(T,E))
\end{equation*}%
where $a$ and $b$ are $1$-periodic in $T$.
A map $\tilde S$ obtained as a result of this substitution
also satisfy (\ref{EqnS}). Although the uniqueness
is not important for our study, we note that the substitutions of
this form exhaustively describe the freedom in the definition of $S$.

The normalisation property of Theorem~\ref{Thm:flowbox} 
does not eliminate this freedom completely but 
only fixes $S$ on the line $E=0$ in a way which makes comparison
between the stable and unstable separatrices more convenient.
\end{remark}

\subsection{Splitting function\label{Se:splitfunc}}

Let us define the splitting function by%
\footnote{The domain of this function is to be studied. } 
\begin{equation}\label{defTheta}
\Theta (T)=E\circ \psi ^{+}(T\log \lambda_\delta ).  
\end{equation}
We can use $\Theta $ to

\begin{enumerate}
\item Find the homoclinic points.

\item Find the homoclinic invariant $\omega$.

\item Calculate the lobe areas.

\item Calculate the distances between $\psi ^{+}$ and $\psi ^{-}$.
\end{enumerate}

Since $E=0$ corresponds to the unstable separatrix,
zeroes of the splitting functions correspond to
homoclinic points, i.e.,
if $\Theta(T_1)=0$ for some $T_1$ then $\psi^{+}(t_1)\in W^-$ for $t_1=T_1\log\lambda$.
Therefore there is $t_2$ such that $\psi^+(t_1)=\psi^-(t_2)$.

\medskip

Let us consider the derivative of the splitting function
with respect to $T$ at $T=T_1$:
\begin{equation*}
\frac{d\Theta }{dT}(T_1)
= 
\left.\nabla  E^t\right\vert _{\psi ^{+}(t_1)} 
\frac{d \psi^{+}}{d T}(T_1\log \lambda ) 
=
\log \lambda \,
\left.\nabla  E^t\right\vert _{\psi ^{-}(t_2)} 
\dot{\psi}^{+}(t_1) \,,
\end{equation*}
where we used the equality $\psi^+(t_1)=\psi^-(t_2)$. We use (\ref{Eq:nablaE})
to eliminate the gradient of~$E$:
\begin{equation*}
\frac{d\Theta }{dT}(T_1)=- \log^2\! \lambda \,
({\mathrm J}\dot\psi^-)^t
\dot{\psi}^{+}
=
\log^2\! \lambda \,
(\dot\psi^-)^t{\mathrm J}\dot{\psi}^{+}
\end{equation*}
since $\mathrm{J}^t=-\mathrm{J}$. Taking into account the definition of the symplectic
form we get
\begin{equation*}
\frac{d\Theta }{dT}(T_1)=- \log^2\! \lambda\, \Omega\left(\dot\psi^-,\dot{\psi}^{+}\right).
\end{equation*}
We remind ourselves that the homoclinic invariant, also known as the Lazutkin invariant, is defined by
\begin{equation*}
\omega =\Omega (\dot{\psi}^{+},\dot{\psi}^{-}).
\end{equation*}
Therefore we have established that the homoclinic invariant is proportional to
the gradient of the splitting function 
\begin{equation}\label{Eq:omegaTheta}
\omega =\frac{1}{\log ^{2}\lambda }\;\Theta'(T_{1})
\end{equation}
at its zero, $\Theta (T_{1})=0$.

\medskip

Next, we note that $\Theta (T)$ is periodic. Indeed using the definition (\ref{defTheta}),
equation (\ref{Eq:mainfde}) with $t=T\log\lambda $ and the equality (\ref{Eq:E1})
we obtain
\begin{eqnarray*}
\Theta (T+1) &=&E\circ \psi ^{+}(\left( T+1\right) \log \lambda )=E\circ
f_{\varepsilon }\circ \psi ^{+}(T\log \lambda ) \\
&=&E\circ \psi ^{+}(T\log \lambda )=\Theta (T)\,.
\end{eqnarray*}%
Hence we can write the Fourier series
\begin{equation*}
\Theta (T)=\sum_{k\in\mathbb Z} \Theta _{k}e^{2\pi i k T}.
\end{equation*}

\begin{lemma} Let $c_1>0$ be fixed
and let $\rho=\pi-c_1$. If $\delta$ is sufficiently small,
the domain of the function $\Theta$ contains the strip
$$
\{\,T: |\Im T|\le \rho/\log\lambda_\delta\,\}\,.
$$
Moreover, for all $s\in\mathbb R$
\begin{equation}\label{Eq:thetareal}
\Theta(s)= \Theta_1e^{2\pi i s}+\Theta_{1}^*e^{-2\pi i s}+
O(e^{-4\pi\rho/\log\lambda_\delta})
\end{equation}
where  $\Theta_{1}\equiv\Theta_{1}(\delta)$ is a first
Fourier coefficient of $\Theta$.
The asymptotic formula {\rm (\ref{Eq:thetareal})} can be differentiated
with respect to $s$, the error term does not change its order.
\end{lemma}

\begin{proof}
We note that on the strip the absolute value of $\Theta(T)$  
does not exceed $E_0(\delta)$ defined in (\ref{DomainS}).
Writing the classical integrals for Fourier
coefficients of a periodic function and shifting the integration path
to the boundary of the analyticity strip, we derive
\begin{equation}\label{Eq:kFourier}
\left\vert \Theta _{k}\right\vert \leq 
E_0(\delta) \mathrm e^{-\left\vert
k\right\vert \frac{2\pi \rho }{\log\lambda_\delta}}.
\end{equation}
Since $\log\lambda_\delta$ is of the order of $\delta$ all Fourier coefficients
are exponentially small (the zero order coefficient being an exception).
Moreover, the coefficients with $|k|\ge2$ are negligible compared to
the size of the separatrix splitting we intend to detect.

Now we show that $\Theta_0$ is also negligible.
We define an auxiliary function  
$g(s)=T\circ \psi ^{+}(s\log\lambda)$ which describes the $T$ component of the
stable separatrix in the flow box coordinates. 
Using the definition (\ref{defTheta}),
equation (\ref{Eq:mainfde}) with $t=T\log\lambda $ and the equality (\ref{Eq:T1})
we obtain
\begin{eqnarray*}
g(s+1)&=&
T\circ \psi ^{+}((s+1)\log\lambda_\delta)=T\circ f_{\varepsilon }\circ \psi
^{+}(s\log\lambda_\delta)\\
&=&T\circ \psi ^{+}(s\log\lambda_\delta)+1=g(s)+1.
\end{eqnarray*}
Consequently $g'$ is periodic and its mean value is 1
\begin{equation}\label{Eq:gp0}
\int_{T_1}^{T_1+1}(1-g'(s))ds=0
\end{equation}
for any $T_1$ from its domain.

Let $T_1$ be a zero of the function $\Theta$.
Then $\mathbf p_1=\psi^+(t_1)=\psi^-(t_2)$ is an homoclinic
point. The segments of the stable and unstable separatrices
with their ends at $\mathbf p_1$ and $f_\varepsilon(\mathbf p_1)$
form a closed loop on the plane. The total algebraic area of this loop
vanishes. Since $S$ is symplectic, this area can be evaluated in
the flow box coordinates. The segment of the unstable separatrix
is a straight line which connects the points $(T_2,0)$ and $(T_2+1,0)$
and therefore the area is given by the integral
\begin{equation}\label{Eq:area0}
\int_{T_1}^{T_1+1}\Theta (s)g'(s)ds=
\int_{T_1}^{T_1+1}\Theta (s)dg(s)
=0.
\end{equation}

\begin{figure}
\begin{center}
{\includegraphics[width=0.7\textwidth]{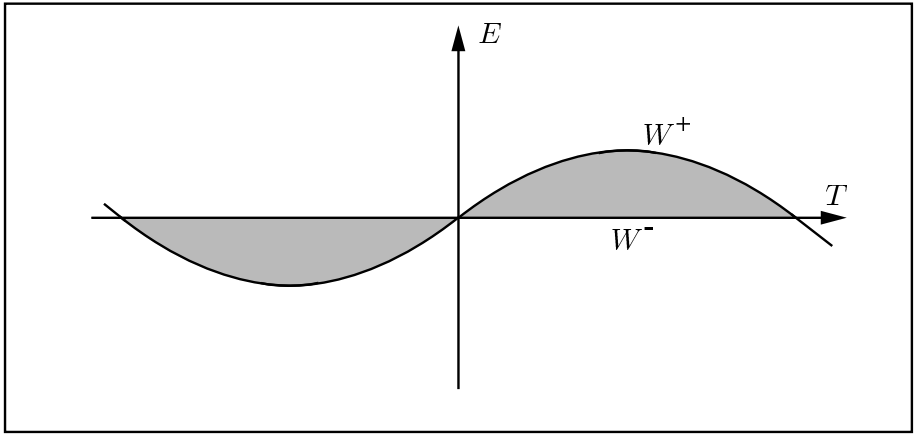}}
\end{center}
\caption{Unstable separatrix is the "x-axis" and the Stable one is a graph
over it; The area is zero since the map is area-preserving.}
\end{figure}

Now we can come back to estimates for the zero order Fourier coefficient
of the function $\Theta$ which is given by the integral
\[
\Theta_0=\int_{T_1}^{T_1+1}\Theta(s)ds\,.
\]
Taking into account (\ref{Eq:area0}) we rewrite it
\[
\Theta_0=\int_{T_1}^{T_1+1}\Theta(s)(1-g'(s))ds\,.
\]
Then using equation (\ref{Eq:gp0}) and the fact that $\Theta_0$ is constant
we get:
\[
\Theta_0=\int_{T_1}^{T_1+1}(\Theta(s)-\Theta_0)(1-g'(s))ds\,.
\]
Under the integral we see the product of two functions, each one
is periodic and of zero mean. The Fourier series arguments show that
\begin{equation}\label{Eq:expbounds}
\Theta(s)-\Theta_0=
O(e^{-2\pi\rho/\log\lambda_\delta})
\qquad \mbox{and}\qquad
1-g'(s)=
O(e^{-2\pi\rho/\log\lambda_\delta})
\end{equation}
since $s$ is real.
Consequently, $\Theta_0$ is extremely small:
\begin{equation*}
\Theta_0=O(e^{-4\pi\rho/\log\lambda_\delta}),
\end{equation*}
as the exponent in the right hand side is almost double
the expected leading order of the separatrices splitting.

In particular, combining the last upper bound
with (\ref{Eq:kFourier}) we conclude that on the real axis
\begin{equation}\label{Eq:ThetaBound}
|\Theta(s)|\le\sum_{k\in\mathbb Z}|\Theta_k|=O(e^{-2\pi\rho/\log\lambda_\delta})\,.
\end{equation}
Moreover, this function is almost sinusoidal since
if we keep apart the two largest Fourier coefficients
the sum of all others is very small and for $s\in\mathbb R$
we get (\ref{Eq:thetareal}).

Of course the  Fourier coefficients $\Theta_1$ and $\Theta_{-1}=\Theta_{1}^{\ast}$ 
depend on $\delta$ and are exponentially small 
due to the upper bound (\ref{Eq:kFourier}) with $|k|=1$. 
Nevertheless the the error term is much smaller than the upper bound
for $\Theta_{\pm1}$. In the next section we will construct
an asymptotic expansion for $\Theta_{-1}$ and conclude that
generically it really dominates the error. 
\end{proof}

Finally we show how to use the splitting function to estimate 
the lobe area between the stable and unstable manifolds.
Suppose $T_1$ and $T_2$ are two zeroes of $\Theta$.
Then the lobe area is represented by the following integral:
\[
\int_{T_1}^{T_2}\Theta(s)g'(s)ds=
\int_{T_1}^{T_2}\Theta(s)ds
+\int_{T_1}^{T_2}\Theta(s)(g'(s)-1)ds\,.
\]
Taking into account the upper bound (\ref{Eq:expbounds})
and (\ref{Eq:ThetaBound}) for the functions under the second integral 
we conclude that the lobe area equals to
\[
\int_{T_1}^{T_2}\Theta(s)ds
+
O(e^{-4\pi\rho/\log\lambda_\delta})\,.
\]

\subsection{Asymptotic expansion for the homoclinic invariant}

\begin{figure}
\begin{center}
\includegraphics[width=0.7\textwidth]{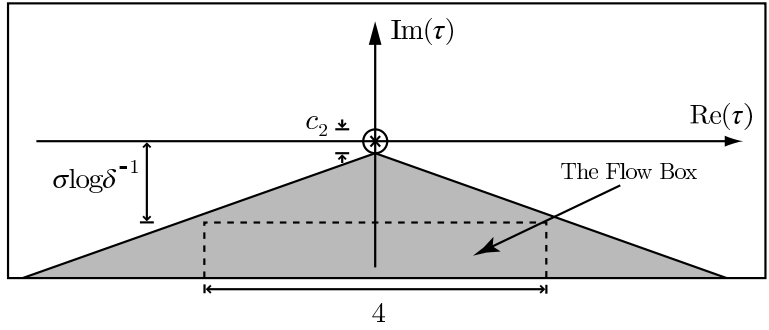}
\end{center}
\caption{Domain used to estimate the first Fourier
coefficient of the splitting function}
\end{figure}

Consider a line in the complex plane of the variable $T$ defined by
\[
\ell(\delta)=\left\{\,\Im T=\rho_1(\delta)\,,
\ 
|\Re T|\le1\,\right\}\,,
\]
where 
\[
\rho_1(\delta)=\frac{\pi}{\log\lambda_\delta}-\sigma\log\delta^{-1}
\]
and $\sigma=\frac n\pi$ is a positive constant.
We fix the following relation between the parameters: $t=T\log\lambda_\delta$ and
\begin{equation}
\label{Eq:Tttau}
\tau =\frac{t-i\pi }{\log \lambda_\delta }=T - \frac{i\pi }{\log \lambda_\delta }.
\end{equation}
We estimate the splitting function defined in (\ref{defTheta})
by expanding the energy $E$ in Taylor series centred at
a point $\psi^-$ of the unstable manifold. 
Using $E\circ\psi^-=0$, equation (\ref{Eq:nablaE}) 
for the gradient of $E$ and the uniform boundedness of the second derivatives
of $E$ we find that
\begin{equation*}
\Theta (T)=\log \lambda_\delta \,\det \left( \frac{d \psi ^{-}}{dt};
\psi ^{-}-\psi ^{+}\right) +O\left(|\psi ^{-}-\psi ^{+}|^2\right)\,.
\end{equation*}
The functions in the right hand side are evaluated at a point $\tau $ with $%
\func{Im}\tau =-\sigma \log \delta ^{-1}.$ Using (\ref{Eq:dpsi'}) with $n$
replaced by $2n$ gives us the following estimate for the difference of $\psi^{\pm}$%
\[
\psi ^{+}-\psi ^{-}=\sum_{k=0}^{2n-1}\delta ^{2k}(\psi _{k}^{+}-\psi
_{k}^{-})+O(\delta ^{4n}\log ^{4n}\delta ^{-1}).
\]%
We note that on $\ell(\delta)$ 
\begin{equation}
\psi_{k}^{+}-\psi _{k}^{-}=O(\delta ^{2n}\log^{m_{k}}\delta^{-1})
\label{Bound_deltasep}
\end{equation}
by the upper bound (\ref{Eq:dpsik}), hence it follows that%
\begin{eqnarray*}
\psi ^{+}-\psi ^{-}&=&\sum_{k=0}^{n-1}\delta ^{2k}(\psi _{k}^{+}-\psi
_{k}^{-})+O(\delta ^{4n}\log ^{m}\delta ^{-1})\qquad{and}\\
\psi ^{+}-\psi ^{-}&=&O(\delta ^{2n}\log ^{m}\delta ^{-1}),
\end{eqnarray*}
where $m=\max \{4n,m_{0},\ldots ,m_{2n-1}\}$. 
On a sufficiently large neighbourhood of $\ell(\delta)$
we have $|\psi^{-}(\tau)|<1$, therefore the Cauchy estimate implies
$\left\vert \frac{d\psi ^{-}}{d \tau }\right\vert <1$ and
consequently
\begin{eqnarray*}
\Theta (T) &=&\det \left( \frac{d\psi ^{-}}{d\tau },\psi ^{-}-\psi
^{+}\right) +O({|\psi^{-}-\psi^{+}|^2}) \\
&=&\det \left( \frac{d\psi ^{-}}{d\tau },\sum_{k=0}^{n-1}\delta ^{2k}(\psi
_{k}^{+}-\psi _{k}^{-})\right) +O(\delta ^{4n}\log ^{2m}\delta ^{-1}).
\end{eqnarray*}%
Now, since $\psi^{-}$ is analytic the series (\ref{Eq:unstappr}) can be 
differentiated term-wise, so using (\ref%
{Bound_deltasep}) we obtain%
\[
\Theta (T)=\det \left( \sum_{k=0}^{n-1}\delta ^{2n}
\frac{d\psi_{k}^{-}}{d\tau}
,\sum_{k=0}^{n-1}\delta ^{2k}(\psi _{k}^{+}-\psi _{k}^{-})\right)
+O(\delta ^{4n}\log ^{\widetilde{m}}\delta ^{-1}),
\]%
where $\widetilde{m}=\max \left\{2 m,(2n-1)m_{0}\right\} .$ We use that the
determinant is bi-linear to rewrite this estimate in the following form%
\begin{equation}
\Theta (T)=\sum_{k_{1}=0}^{n-1}\sum_{k_{2}=0}^{n-1}\delta
^{2(k_{1}+k_{2})}\det \left( \frac{d\psi _{k_{1}}^{-}}{d\tau },\psi
_{k_{2}}^{-}-\psi _{k_{2}}^{+}\right) +O(\delta ^{4n}\log ^{\widetilde{m}%
}\delta ^{-1}).  \label{theta_doublesum}
\end{equation}%
Next we estimate the derivative of $\psi_{k_1}^{-}$. Since it is an analytic 
function in a small domain $\mathcal D$ we can use the Cauchy estimate 
to estimate it using $\psi_{k_1}^{-}$ itself. Using that $\psi_{k_1}^{-}$ is 
asymptotic to, see (\ref{Eq:psim'}), the series given by (\ref{Eq:psim_hat}) yields
\[
\left\vert \frac{d\psi _{k_{1}}^{-}}{d\tau }\right\vert =O(\tau
^{m_{k_{1}}})=O(\log ^{m_{k_{1}}}\delta ^{-1})
\]%
which together with the estimate (\ref{Bound_deltasep}) gives%
\[
\det \left( \frac{d\psi _{k_{1}}^{-}}{d\tau },\psi _{k_{2}}^{-}-\psi
_{k_{2}}^{+}\right) =O(\delta ^{2n}\log ^{m_{k_{1}}+m_{k_{2}}}\delta ^{-1}).
\]%
The last estimate shows that terms in (\ref{theta_doublesum}) with $%
k_{1}+k_{2}\geq n$ are not larger than the error term. We collect terms with 
$\left( k_{1}+k_{2}\right) <n$ to get%
\begin{equation}\label{eqn2_Theta}
\Theta (T)=\sum_{k=0}^{n-1}\delta ^{2k}\sum_{\substack{ k_{1}+k_{2}=k \\ %
k_{1},k_{2}\geq 0}}\det \left( \frac{d\psi _{k_{1}}^{-}}{d\tau },\psi
_{k_{2}}^{-}-\psi _{k_{2}}^{+}\right) +O(\delta ^{4n}\log ^{\widehat{m}%
}\delta ^{-1}),
\end{equation}%
where $\widehat{m}=\max_{k_{1}+k_{2}=n}\{\widetilde{m},m_{k_{1}}+m_{k_{2}}\}$.

Let us introduce the auxiliary functions
\begin{equation}\label{Eq:thetak}
\theta _{k}(\tau )=\sum_{k_{1}+k_{2}=k}\det \left( \frac{d\psi _{k_{1}}}{%
d\tau }(\tau);\psi _{k_{2}}^{-}(\tau )-\psi _{k_{2}}^{+}(\tau )\right) .
\end{equation}
We note that these functions are independent from the parameter $\delta$.
The next lemma gives a more accurate estimate for these functions.

\begin{lemma}
For every integer $k\ge0$ 
there exist $\omega _{k}\in \mathbb{C}$ and $N_k\in\mathbb N$ such that%
\begin{equation}
\theta _{k}(\tau )=\omega _{k}e^{-2\pi i\tau }+O(\tau ^{N_{k}}e^{-4\pi i\tau
}),  \label{eqnTheta_k}
\end{equation}%
where $\tau \in \{-\pi +\beta _{1}<\arg \tau <-\beta _{1},$ $\func{Im}\tau <c\}$.
\end{lemma}

According to the lemma we have 
\begin{equation}
\theta _{k}(\tau )=\omega _{k}e^{-2\pi i\tau}
+O(\delta ^{4n}\log ^{N_{k}}\delta ^{-1})
\end{equation}
on the line $\ell(\delta )$.
Substituting this and (\ref{Eq:thetak}) into equation (\ref{eqn2_Theta}), 
we obtain
\[
\Theta (T)=\sum_{k=0}^{n-1}\delta ^{2k}\omega _{k}e^{-2\pi i\tau }+O(\delta
^{4n}\log ^{M}\delta ^{-1}),
\]%
where $M=\max \{m,N_{0}\}$.
Now we can
evaluate the first Fourier coefficients of $\Theta$:
\begin{eqnarray*}
\Theta_{-1}&=&\int_{i\rho_1(\delta)}^{1+i\rho_1(\delta)}e^{2\pi i T}\Theta (T)\,dT
\\&=&
e^{-\frac{2\pi^2}{\log\lambda_\delta}}
\int_{-i \sigma\log\delta^{-1}}^{1-i \sigma\log\delta^{-1}} e^{2\pi i\tau}\left(
\sum_{k=0}^{n-1}\delta ^{2k}\omega _{k}e^{-2\pi i\tau }
+
O\left(\delta^{4n}\log ^{M}\delta ^{-1}\right)
\right)
d\tau\,,
\end{eqnarray*}
where the exponential factor comes from the change of the variables
(\ref{Eq:Tttau}). Expanding the parenthesis and taking into account that 
$|e^{2{\pi}i{\tau}}|=\delta^{-2n}$ on the integration path we get 
\begin{equation*}
\Theta _{-1}=\left( \sum_{k=0}^{n-1}\delta ^{2k}\omega _{k}+
O\left(\delta^{2n}\log ^{M}\delta ^{-1}\right)
\right) e^{-\frac{2\pi ^{2}}{\log \lambda }}\,.
\end{equation*}
With this bound on the first Fourier coefficient we are now ready to 
estimate the size of $\Theta(t)$ for $t$ real. In particular, if $s\in \mathbb{R}$ then
(\ref{Eq:thetareal}) with $\rho=\frac34\pi$ implies
\begin{eqnarray*}
\Theta (s) &=&\Theta _{-1}e^{-2\pi i s}+\Theta_{-1}^{\ast }e^{2\pi i s} +
O\left(e^{-\frac{3\pi^2}{\log\lambda_\delta}}\right)
\\
&= &\sum_{k=0}^{n-1}\delta ^{2k}\left( \omega _{k}e^{-2\pi i s}
+\omega _{k}^{\ast }e^{2\pi i s}\right)
e^{-\frac{2\pi ^{2}}{\log \lambda }}+O\left(\delta ^{2n}e^{-\frac{2\pi ^{2}}{\log
\lambda }}\right) \,.
\end{eqnarray*}
There are two cases to consider. If $\omega_k$ vanish for all $k$, this formula
simply gives an upper bound for the derivative of $\Theta$ and, as a result,
for the homoclinic invariant of any primary homoclinic trajectory.
If, on the other hand, some $\omega_{k}\neq{0}$ the leading term is larger than 
the error. In this case  it is more convenient to rewrite the formula in the form 
\begin{equation*}
\Theta (s)=\sum_{k=0}^{n-1}a_{k}\delta ^{2k}\cos \left( 2\pi s
+\sum_{k=0}^{n-1}\varphi _{k}\delta ^{2k}\right) e^{\frac{-2\pi ^{2}%
}{\log \lambda }}+O\left(\delta ^{2n}e^{-\frac{2\pi ^{2}}{\log \lambda }}\right),
\end{equation*}%
where $a_{0}=\left\vert \omega _{0}\right\vert $ and at least one of the amplitudes
does not vanish. The implicit function theorem then implies that the 
function $\Theta$ has exactly two zeroes per period. The 
derivative of $\Theta$ at the zeroes is 
$\pm 2\pi e^{\frac{-2\pi ^{2}}{\log \lambda }}\sum_{k=0}^{n-1}a_{k}\delta ^{2k}$.
Therefore there are exactly two primary homoclinic orbits and
the relation (\ref{Eq:omegaTheta}) implies our main result:
\begin{equation*}
\omega =\pm \frac{2\pi }{\log \lambda ^{2}}
\sum_{k=0}^{n-1}a_{k}\delta ^{2k}%
e^{-\frac{2\pi ^{2}}{\log \lambda }}
+O\left(\delta ^{2n}e^{-\frac{2\pi ^{2}}{\log \lambda }}\right).
\end{equation*}

\section{Normal form for the bifurcation\label{Se:normalform}}

\subsection{Formal series and quasi-homogeneous polynomials}

In this section we will mainly be interested in transformations given in the form
of formal power series. We will consider the series in powers of 
the space variables $(x,y)$ combined with expansions in the parameter $\varepsilon$. 
The series have the form
\[
g(x,y,\varepsilon)=\sum_{k,l,m}c_{klm}x^ky^l\varepsilon^m\,.
\]
A formal series is treated as a collection of coefficients.
Formal series form an infinite dimensional vector space.
Addition, multiplication, integration
and differentiation are defined in a way compatible with the common
definition on the subset of convergent series.

The series involve several variables and it is convenient
to group terms which are ``of the same order". A usual 
choice is to consider $x$ and $y$ to be of the same order.
But for purpose of this paper it is much more convenient
to assume
\begin{itemize}
\item $x$ is of order 2;
\item $y$ is of order 3;
\item $\varepsilon$ is of order 4.
\end{itemize}
Then a monomial $x^ky^l\varepsilon^m$ is considered to be
of order $2k+3l+4m$. We can write
\[
g(x,y,\varepsilon)=\sum_{p\ge 0} g_p(x,y,\varepsilon)\qquad\mbox{where}\qquad
g_p(x,y,\varepsilon)=\sum_{2k+3l+4m=p}c_{klm}x^ky^l\varepsilon^m
\]
is a quasi-homogeneous polynomial of order $p$.
We stress that this notation does not refer 
to a resummation of the divergent series but simply indicates the order
in which  the coefficients are to be treated.

In order to give a rigorous background for manipulation with
formal series we define a metric on the space of formal series ${\mathfrak H}$. 
Let $g$ and $\tilde g$ be two formal
series and let $p$ denote the lowest (quasi-homogeneous) 
order of $g-\tilde g$. If $g\ne\tilde g$ then $p$ is
finite and we let
$$
d(g,\tilde g)=2^{-p},
$$
otherwise we assume 
$$
 d(g,g)=0\,.
$$
It is a straightforward to check that
$(\mathfrak H,d)$ is a complete metric space.
Moreover polynomials are dense in $(\mathfrak H,d)$.
Hence we can define formal convergence and formal continuity
on the space of formal series. In particular an operator
is formally continuous if each coefficient of a series
in its image is a function of a finite number of
coefficients of a series in its argument.

Let $\chi$ and $g$ be two formal power series. 
We note that any of the series
involved in next definitions may diverge.
The linear operator defined by the formula
\begin{equation}\label{Def_L}
L_{\chi }g=\left\{g, \chi\right\} 
\end{equation}
is called {\em the Lie derivative\/} generated by $\chi$.
We note that if $\chi$ starts from an order $p$ and $g$
starts with an order $q$, then the series $L_{\chi }g$
starts with the order $p+q-5$ as the Poisson bracket
involves differentiation with respect to $x$ and $y$. 
If $p\ge6$ the lowest order in $L_{\chi }g$ is at least $q+1$.

We define the exponent of $L_\chi$ by
\begin{equation}\label{Eq:formalexp}
\exp(L_\chi)g=\sum_{k\ge0}\frac1{k!}L_\chi^kg\,,
\end{equation}
where $L_\chi^k$ stands for the operator $L_\chi$
applied $k$ times. 
The lowest order in the series $L_\chi^kg$
is at least $q+k$ and consequently every coefficient
of $\exp(L_\chi)g$ depends only on a finite number of
coefficients of $\chi$ and $g$.

If the lowest order of $\chi$ is at least 6 the series
(\ref{Eq:formalexp}) converges with respect to the metric $d$,
i.e., each coefficient of the result is function of finite number
of coefficients of $\chi$ and $g$.

\medskip

We consider consider the formal series
\[
\Phi_\chi^1=\left(\exp(L_\chi) x,
\exp(L_\chi) y
\right).
\]
We say that $\Phi^1_\chi$ is a Lie series generated by the
formal Hamiltonian $\chi$. If $\chi$ is a polynomial
the series converge on a poly-disk and coincide with
a map which shifts points along trajectories of the 
Hamiltonian system with Hamiltonian function $\chi$.
For this reason sometimes we will call this series
a time one map of the Hamiltonian system $\chi$.

We note that it is easy to construct the formal series
for the inverse map:
\[\Phi_\chi^{-1}=\left(\exp(-L_\chi) x,
\exp(-L_\chi) y
\right)\,.
\]
Then $\Phi_{\chi}^1\circ\Phi_{\chi}^{-1}(x,y)=(x,y)$.
We also note that
\[
g\circ \Phi_\chi^1=\exp(L_\chi)g\,.
\]
These formulae are well known to be valid for convergent
series and can be extended onto ${\mathfrak H}$ due to the
density property.

\subsection{Formal interpolation\label{Se:formalinterpol}}

We study the following family of area-preserving maps
\begin{equation*}
F_{\varepsilon }:\left( 
\begin{array}{c}
x \\ 
y%
\end{array}%
\right) \mapsto \left( 
\begin{array}{c}
x+y+f(x,y,\varepsilon ) \\ 
y+g(x,y,\varepsilon )%
\end{array}%
\right) 
\end{equation*}
and the Taylor series of $f(x,y,0)$ and $g(x,y,0)$
do not contain constant and linear terms.
Since $F_{\varepsilon }$ is area-preserving we have%
\begin{equation*}
\det DF_{\varepsilon } =\left\vert 
\begin{array}{cc}
1 +\partial _{x}f & 1+\partial _{y}f \\ 
\partial _{x}g & 1 +\partial _{y}g%
\end{array}%
\right\vert  
=1
\end{equation*}
which is equivalent to
\begin{equation}
\partial _{x}f+\partial _{y}g+\left\{ f,g\right\}-\partial _{x}g\equiv0. 
 \label{Parabolic_ap}
\end{equation}
Next theorem states that $F_\varepsilon$ can be formally
interpolating by an autonomous Hamiltonian flow.

\begin{theorem}
Let $F_{\varepsilon }$ be family of area-preserving maps such that 
\begin{equation*}
DF_{0}(0)=\left(
\begin{array}{cc}
1 & 1 \\ 
0 & 1%
\end{array}%
\right) 
\end{equation*}%
then there exists a unique (up to adding a formal series
in powers $\varepsilon$ only) formal Hamiltonian 
$h_{\varepsilon }$ such that 
\begin{equation}\label{Eq:forint}
\Phi _{h_{\varepsilon }}^{1}=F_{\varepsilon }.
\end{equation}
\end{theorem}

\begin{proof}
The Taylor series for $F_\varepsilon: (x,y)\mapsto (x_{1},y_{1})$ 
is written as a sum of quasi-homogeneous polynomials:
\begin{equation}
\left\{ 
\begin{array}{l}
x_{1}=x+ y+\sum\limits_{p\geq 4}f_{p}(x,y,\varepsilon ), \\ 
y_{1}=y+\sum\limits_{p\geq 4}g_{p}(x,y,\varepsilon ),
\end{array}
\right.   \label{Lesson7_1}
\end{equation}
where $f_{p}$ and $g_{p}$ are quasi-homogeneous polynomials: 
\begin{equation*}
\left\{ 
\begin{array}{c}
f_{p}(x,y,\varepsilon )=\sum\limits_{2k+3l+4m=p}f_{klm}x^{k}y^{l}\varepsilon
^{m}, \\[12pt]
g_{p}(x,y,\varepsilon )=\sum\limits_{2k+3l+4m=p}g_{klm}x^{k}y^{l}\varepsilon
^{m}.
\end{array}
\right. 
\end{equation*}
Similarly we look for $h_{\varepsilon }$ in the form of
a sum of quasi-homogeneous polynomials using
the same quasi-homogeneous ordering for terms:
\begin{equation*}
h_{\varepsilon }(x,y)=
\sum_{p\geq 6}h_{p}(x,y,\varepsilon)
\qquad\mbox{where}\quad
h_{p}(x,y,\varepsilon)=\sum_{2k+3l+4m=p}h_{klm}x^{k}y^{l}\varepsilon ^{m}.
\end{equation*}%
The time-$1$ map of this hamiltonian is given by
the Lie series
\begin{equation*}
\Phi _{h_{\varepsilon }}^{1}(x,y)=\left( 
\begin{array}{c}
x+L_{h_{\varepsilon }}x+\sum\limits_{k\geq 2}\frac{1}{k!}L_{h_{\varepsilon
}}^{k}x \\ 
y+L_{h_{\varepsilon }}y+\sum\limits_{k\geq 2}\frac{1}{k!}L_{h_{\varepsilon
}}^{k}y%
\end{array}%
\right) ,
\end{equation*}
where the operator $L_{h_{\varepsilon }}$ is defined by
\begin{equation*}
L_{h_{\varepsilon }}(\varphi )=\left\{ \varphi ,h_{\varepsilon }\right\} =%
\frac{\partial \varphi }{\partial x}\frac{\partial h_{\varepsilon }}{%
\partial y}-\frac{\partial \varphi }{\partial y}\frac{\partial
h_{\varepsilon }}{\partial x}.
\end{equation*}
Now we use induction to show that there is a formal Hamiltonian
$h_\varepsilon$ such that 
$\Phi _{h_{\varepsilon }}^{1}=F_{\varepsilon }$ at all orders. 
This is an equality between two formal series, i.e., all coefficients of two
series coincide. We note the first component of the series
starts with the quasi-homogeneous order 2 and the second one starts with 3.
Therefore it is convenient to consider an order $p$ in the first
component simultaneously with the order $p+1$ in the second one.
In this situation we say that we consider a term of the order $(p,p+1)$. 

Equality (\ref{Eq:forint}) is equivalent to the following infinite system 
\begin{eqnarray}
L_{h_{p+3}}x+\sum_{k\geq 2}\frac{1}{k!}\left[ L_{h_{\varepsilon }}^{k}x%
\right] _{p} &=&\left\{ 
\begin{array}{ll}
f_{p}\quad &\mbox{if }p\geq 4 \\ 
y\quad &\mbox{if }p=3%
\end{array}%
\right.   \label{Lesson7_3} \\
L_{h_{p+3}}y+\sum_{k\geq 2}\frac{1}{k!}\left[ L_{h_{\varepsilon }}^{k}y%
\right] _{p} &=&g_{p+1},  \label{Lesson7_3a} 
\end{eqnarray}
where $[\cdot]_p$ is used to denote terms of the quasi-homogeneous order $p$
of a formal series. We note that 
\begin{equation*}
\func{order}(L_{h_{k}}x)=k-3\qquad\mbox{and}\qquad
\func{order}(L_{h_{k}}y)=k-2\,.
\end{equation*}
First we check that the equations can be solved for $p=3$.
It is easy to see that 
\begin{eqnarray*}
L_{h_{6}}x &=&\frac{\partial h_{6}}{\partial y}\,,\\
L_{h_{6}}y &=&-\frac{\partial h_{6}}{\partial x}.
\end{eqnarray*}
Then the equations (\ref{Lesson7_3}), (\ref{Lesson7_3a})
take the form
\begin{eqnarray}
 \frac{\partial h_{6}}{\partial y}&=&y,  \label{Lesson7_4} \\
-\frac{\partial h_{6}}{\partial x} &=&g_{4}(x,y,\varepsilon ).  \notag
\end{eqnarray}
Since $g_4$ is a quasihomogeneous polynomial of order $4$ it is independent from $y$ which
is of order $3$. Consequently
\begin{equation*}
g_{4}(x,y,\varepsilon )=g_{4}(x,\varepsilon )\,.
\end{equation*}
Then system (\ref{Lesson7_4}) has a unique solution of 
the form
\begin{equation*}
h_{6}=\frac{y^{2}}{2}+G_{6}(x,\varepsilon ),
\end{equation*}
where $G_6$ is a quasi-homogeneous polynomial 
of order $6$ in $x$ and $\varepsilon$.

Let us now proceed with the induction step. Assume we have 
$h_{6},\ldots,h_{p+2}$ and want to compute $h_{p+3}$ for $p>3$. 
To make the combinatorics easier
we introduce the notation
\begin{equation*}
L_{s}\varphi =\left\{ \varphi ,h_{s+5}\right\} .
\end{equation*}
Then%
\begin{equation*}
L_{h_{\varepsilon }}\varphi =\left\{ \varphi ,h_{\varepsilon }\right\}
=\sum_{p\geq 6}\left\{ \varphi ,h_p\right\} =\sum_{s\geq
1}L_{s}\varphi \,.
\end{equation*}
Note that $L_{s}$ maps a quasi-homogeneous polynomial of order $j$ into
a quasi-homogeneous polynomial of order $j+s$. So we have
\begin{eqnarray*}
L_{h_{\varepsilon }}^{k}\varphi _{j} &=&
\left( \sum_{s\geq 1}L_{s}\right) ^{k}\varphi_j
=\sum_{\substack{ s_{1}+s_{2}+\ldots +s_{k}\geq k \\ %
s_{1},s_{2},\ldots ,s_{k}\geq 1}}L_{s_{1}}\ldots L_{s_{k}}\varphi _{j} \\
&=&\sum_{m\geq k}\sum_{s_{1}+s_{2}+\ldots +s_{k}=m}L_{s_{1}}\ldots
L_{s_{k}}\varphi _{j}.
\end{eqnarray*}%
Now let us consider the components of $\Phi _{h_{\varepsilon }}^{1}$ at
order $(p,p+1),$ the first component is%
\begin{equation*}
\left[ L_{h_{\varepsilon }}x+\sum_{k\geq 2}\frac{1}{k!}L_{h_{\varepsilon
}}^{k}x\right] _{p}=L_{p-2}x+\sum_{k=2}^{p-2}\frac{1}{k!}\sum_{s_{1}+s_{2}+%
\ldots s_{k}=p-2}L_{s_{1}}\ldots L_{s_{k}}x
\end{equation*}%
and the second component%
\begin{equation*}
\left[ L_{h_{\varepsilon }}y+\sum_{k\geq 2}\frac{1}{k!}L_{h_{\varepsilon
}}^{k}y\right] _{p+1}=L_{p-2}y+\sum_{k=2}^{p-2}\frac{1}{k!}%
\sum_{s_{1}+s_{2}+\ldots s_{k}=p-2}L_{s_{1}}\ldots L_{s_{k}}y.
\end{equation*}%
We want%
\begin{equation*}
\left( 
\begin{array}{c}
1^{st}\text{ }comp \\ 
2^{nd}\text{ }comp%
\end{array}%
\right) =\left( 
\begin{array}{c}
f_{p} \\ 
g_{p+1}%
\end{array}%
\right) .
\end{equation*}%
We have%
\begin{eqnarray*}
L_{p-2}x &=&\left\{ x,h_{p+3}\right\} , \\
L_{p-2}y &=&\left\{ y,h_{p+3}\right\} .
\end{eqnarray*}%
Thus%
\begin{equation}
\label{Eq:hp3}
\left( 
\begin{array}{c}
\frac{\partial h_{p+3}}{\partial y} \\ 
-\frac{\partial h_{p+3}}{\partial x}%
\end{array}%
\right) =\left( 
\begin{array}{c}
f_{p}-\sum\limits_{k=2}^{p-2}\frac{1}{k!}\sum\limits_{s_{1}+\ldots
+s_{k}=p-2}L_{s_{1}}\ldots L_{s_{k}}x \\ 
g_{p+1}-\sum\limits_{k=2}^{p-2}\frac{1}{k!}\sum\limits_{s_{1}+\ldots
+s_{k}=p-2}L_{s_{1}}\ldots L_{s_{k}}y%
\end{array}%
\right) .
\end{equation}
The Hamiltonian $h_{p+3}$ is defined from this equation
uniquely up to a function of $\varepsilon$
if and only if the right hand side is divergence free.
The last condition follows from the area-preservation property
due to the following lemma.

\begin{lemma}
If expansions of two area-preserving maps coincide up to the order $\left( p-1,p\right) $
then the divergences of the terms of the order $(p,p+1)$ are equal.
\end{lemma}

\begin{proof}
We use the area-preserving property (\ref{Parabolic_ap})
and collect the terms of order $p-2$ 
\begin{equation*}
\partial _{x}f_{p}+\partial _{y}g_{p+1}-\partial _{x}g_{p}
+\sum_{\substack{ k+l=p+3 \\ k,l\geq 4}}\left\{ f_{k},g_{l}\right\} =0.
\end{equation*}
Consequently,
\begin{eqnarray*}
 \partial _{x}f_{p}+\partial _{y}g_{p+1}  &=&\partial
_{x}g_{p}-\sum_{\substack{ k+l=p+3 \\ k,l\geq 4}}\left\{ f_{k},g_{l}\right\} 
\\
&=&
\partial_{x}\tilde g_{p}-\sum_{\substack{ k+l=p+3 \\ k,l\geq 4}}\left\{\tilde f_{k},\tilde g_{l}\right\} 
=\partial _{x}\widetilde{f}_{p}+\partial _{y}\widetilde{g}_{p+1}.
\end{eqnarray*}
\end{proof}

We finish the proof by the following observation.
Let $\tilde h=\sum_{k=6}^{p+3} h_k$. For any
choice of $h_{p+3}$ the maps $\Phi^1_{\tilde h}$
and $F_\varepsilon$ satisfy the assumptions of the previous lemma
therefore their orders $(p,p+1)$ have the same divergence.
That is the solvability condition for the equation (\ref{Eq:hp3}).
\end{proof}

\subsection{Simplification of the interpolating Hamiltonian}

The Hamiltonian $h_\varepsilon$ can be simplified with the help
of a formal canonical change of variables.

\begin{theorem}\label{Thm:parabolic_simpl}
If $h(x,y,\varepsilon )=\sum_{p\geq 6}h_{p}(x,y,\varepsilon )$ 
such that the leading order has the form
\begin{equation}\label{Eq:h6}
h_{6}(x,y,\varepsilon )=
\frac{y^{2}}{2}+a\frac{x^{3}}3-b\varepsilon x
\end{equation}
then there exists a formal canonical substitution 
such that the Hamiltonian takes the form
\begin{equation*}
\widetilde{h}(x,y,\varepsilon )=\frac{y^{2}}{2}+\sum_{p\geq
6}u_{p}(x,\varepsilon ).
\end{equation*}
\end{theorem}

\begin{proof} We prove the theorem by induction. The order
$6$ is already in the desired form, just let $u_6=ax^{3}/3-b\varepsilon x$. 
Then suppose we transformed the Hamiltonian to the desired form
up to the order $p$ with $p\ge 6$. For any quasi-homogeneous polynomial
$\chi_p$ of order $p$ we have
\begin{eqnarray*}
\widetilde{h} &=&{\mathrm{e}}^{L_{\chi_p}}h=
h+\left\{ h,\chi _{p}\right\} +\hat O_{2p-10+6} \\
&=&h+\left\{ h_{6},\chi _{p}\right\} +\hat O_{p+2}
\end{eqnarray*}
We see that this change does not affect terms of order less or equal $p$.
Collecting all terms of order $p+1$ we get
\begin{equation}
\tilde{h}_{p+1}=h_{p+1}+\left\{ h_{6},\chi _{p}\right\}
.
\label{Lesson8_4}
\end{equation}%
In order to complete the proof 
we show that there is a polynomial $\chi _{p}$ such that $\tilde h_{p+1}$
is independent from $y$. 
The situation is different for $p$ even and $p$ odd.
We will show that if $p$ is even then for any $h_{p+1}$ 
there is $\chi_p$ such that
\begin{equation}\label{Eq:homeq0}
h_{p+1}+\left\{ h_{6},\chi _{p}\right\}=0\,.
\end{equation}
The polynomial $\chi_p$ is defined up to 
addition of an arbitrary polynomial of $h_6$ and $\varepsilon$.
On the other hand, if $p$ is odd then there
are two unique polynomials $u_{p+1}=u_{p+1}(x,\varepsilon)$
and $\chi_{p+1}=\chi_{p+1}(x,y,\varepsilon)$
such that
\begin{equation}\label{Eq:homeq1}
h_{p+1}+\left\{ h_{6},\chi _{p}\right\}=u_{p+1}\,.
\end{equation}
In order to prove these statements we note that since 
$y$ is counted as a term of order three
and $x$, $\varepsilon$ as terms of order two and four respectively,
a quasihomogeneous polynomial of even order contains only
even powers of $y$ and one of an odd order contains only odd powers.
Let $p+1=3q+r$ with $r\in\{0,1,2\}$. The largest possible power 
of $y$ in $h_{p+1}$ does not exceed $q$ and we can write 
\begin{equation*}
h_{p+1}=\sum_{l=0}^{q/2} y^{q - 2 l} s_{r+6l}(x,\varepsilon)
\qquad\mbox{and}\qquad
\chi _{p}=\sum_{l=0}^{(q-1)/2} y^{q-2l-1}\sigma_{r+6l+2}(x,\varepsilon),
\end{equation*}
where $s_j$ and $\sigma_j$ are quasi-homogeneous polynomials
of $x$ and $\varepsilon$. Taking into account
\[
\left\{ h_{6},\chi _{p}\right\}=
\frac{\partial u_{6}}{\partial x}\frac{\partial
\chi _{p}}{\partial y}-ay\frac{\partial \chi _{p}}{\partial x}
\]
and collecting terms in $h_{p+1}+\{h_6,\chi_p\}$ which have
of the same order in $y$ (excluding the $y$-independent terms)
we get
\begin{eqnarray*}
-a\frac{\partial \sigma_{r+2}}{\partial x}+s_{r}
&=&0\,,
\\
-a
\frac{\partial \sigma_{r+6l+2}}{\partial x}+
(q-2l+1)\frac{\partial u_{6}}{\partial x}
\sigma_{r+6l-4}
+s_{r+6l}
&=&0\,.
\end{eqnarray*}
The first equation defines $\sigma_{r+2}$. If $r\ne 2$ the solution is unique,
otherwise a constant times $\varepsilon$ can be added.
Then $\sigma_{r+6l+2}$ are to be chosen recursively  for $1\le l\le q/2$
to satisfy the second equation.

The theorem follows by induction in $p$.
\end{proof}

\section{Formal separatrix\label{Se:FS}}

In this section we will describe properties of
an asymptotic expansion for the separatrices
providing a proof for the statements of Section~\ref{Se:fs}.

\subsection{Auxiliary functions}

In this subsection we describe a useful class of functions.
A function belongs to this class if
\begin{enumerate}
\item it is periodic with the imaginary period $2\pi i.$

\item it is analytic on the entire $\mathbb{C}$ except for poles at $\pi i(2k+1)$, $k\in \mathbb{Z}.$

\item it vanishes as $\func{Re}t\rightarrow \pm \infty $.
\end{enumerate}
Any function which satisfies these three assumptions is a polynomial,
without a constant term, of two ``base" functions%
\footnote{We note that $\tanh\frac t2$ is $2\pi i$ periodic and
 analytic on $\mathbb C$ except for simple poles at $i\pi(2k+1)$
with $k\in\mathbb Z$. Suppose a function $\varphi(t)$
satisfies properties 1 and 2 and is bounded as $\Re t\to\pm\infty$.
Comparing Laurent expansions around any of the poles, we can construct 
a polynomial $p$ such that $\varphi(t)-p(\tanh\frac t2)$ has no
singularities and, consequently, is constant.
The condition 3 adds a restriction on coefficients of the polynomial
($p(1)=p(-1)=0$).}
\begin{equation}
\eta _{0} =\frac{1}{\cosh ^{2}\frac{t}{2}}
\qquad\mbox{and}\qquad
\eta _{1} =-\frac{\sinh \frac{t}{2}}{\cosh ^{3}\frac{t}{2}}\,.
\end{equation}
Equivalently, the function can be written in the form $P(\eta_0)+\eta_1Q(\eta_0)$
where $P,Q$ are polynomials of $\eta_0$ only and $P(0)=0$.
%
%
%
Indeed, a straightforward substitution shows that $\eta_0$, $\eta_1$ satisfy
the differential equation 
\begin{eqnarray}
\dot{\eta}_{0} &=&\eta _{1},  \label{HamEqnMotion} \\
\dot{\eta}_{1} &=&
\eta _{0}-\frac{3}{2}\eta _{0}^{2}.  \notag
\end{eqnarray}%
This equation is Hamiltonian with the Hamiltonian function
\begin{equation*}
H=\frac{1}{2}\left( \eta _{1}^{2}-\eta _{0}^{2}+\eta _{0}^{3}\right) .
\end{equation*}
The Hamiltonian is constant along solutions of the differential equation.
Since both $\eta_0$ and $\eta_1$ vanish as $\func{Re}\to\infty$, we get
the identity
\begin{equation}
\eta _{1}^{2}-\eta _{0}^{2}+\eta _{0}^{3}=0.  \label{eta1_squared}
\end{equation}
This last equality and (\ref{HamEqnMotion}) imply that
\begin{equation}
\dot{\eta}_{0}^{2}=\eta _{1}^{2}=\eta _{0}^{2}-\eta _{0}^{3}\,.
\label{EtaDotSq}
\end{equation}
Using this identity we can exclude all even powers of $\eta_1$ 
and reduce all odd powers to the power one.
These identities will be used to simplify formulae involved in the construction
of the formal separatrix.

\subsection{Formal separatrix of the flow}

In this section we prove Theorem~\ref{Thm:fs}. We remind that $\delta=\varepsilon^{1/4}$.
By Theorem \ref{Thm:parabolic_simpl} there is a formal canonical change 
of variables which transforms the Hamiltonian $h_\varepsilon$
to the simpler form
\begin{equation}
H_{\varepsilon }(x,y)=\frac{y^{2}}{2}+U(x,\varepsilon ),  \label{NFHamiltonian}
\end{equation}%
where%
\begin{equation}
U(x,\varepsilon )=\sum_{\substack{ m\geq 0,k\geq 1  \\ k+2m\geq 3 }}
u_{km}x^{k}\varepsilon ^{m}  \label{NFPotential}\,.
\end{equation}
After this change equation (\ref{Eq:mainforma}), which we for convenience repeat 
\begin{equation}
\mu(\delta)\dot {\mathbf{X}}=\mathrm{J}\nabla H_\varepsilon(\mathbf X),
\end{equation}
takes the form
\begin{equation}\label{Eq:system}
\left\{ 
\begin{array}{l}
\mu(\varepsilon ) \dot{x}=y, \\ 
\mu(\varepsilon ) \dot{y}=-U^{\prime }(x,\varepsilon ),%
\end{array}%
\right. 
\end{equation}%
where $U^{\prime }\equiv \frac{\partial U}{\partial x}$. 
This system can be conveniently rewritten in an equivalent form
as a scalar equation of second order:
\begin{equation*}
(\mu(\varepsilon ))^{2}\ddot{x}=-U^{\prime }(x,\varepsilon ).
\end{equation*}
In order to simplify notation we introduce the new auxiliary series
$$
\beta(\varepsilon)=\frac12(\mu(\varepsilon))^2\,.
$$
Since $\mu$ was an odd series in $\delta$, the series $\beta$ is even.
The equation takes the form
\[
2\beta(\varepsilon) \ddot{x}=-U^{\prime }(x,\varepsilon )
\]
Multiplying with $\dot{x}$ and integrating with respect to $t$ we obtain%
\begin{equation}
 \beta (\varepsilon )\dot{x}^{2}=-U(x,%
\varepsilon )+c(\varepsilon ),  \label{eqnToSolve}
\end{equation}%
where $c(\varepsilon )$ is a constant of integration. Now we use the following
ansatz for solving this equation:
\begin{eqnarray*}
x(t,\varepsilon )&=&\sum_{k\geq 1}\delta ^{2k}x_{k}(t),
\\
\beta (\varepsilon )&=&\sum_{k\geq 1}a_{k}\delta^{2k}\,,
\end{eqnarray*}
and%
\begin{equation*}
c(\varepsilon )=\sum_{k\geq 3}c_{k}\delta ^{2k}.
\end{equation*}%
Substituting these series into equation (\ref{eqnToSolve})
we get
\begin{equation}
\left( \sum_{k\geq 1}a_{k}\delta ^{2k}\right) \left( \sum_{k\geq 1}\delta
^{2k}\dot{x}_{k}(t)\right) ^{2}+\sum_{k+2m\geq 3}u_{km}\left( \sum_{j\geq
1}\delta ^{2j}x_{j}\right) ^{k}\delta ^{4m}=\sum_{k\geq 3}c_{k}\delta ^{2k}.
\label{eqnToSolveSeries}
\end{equation}%
We recall that this equation is considered in the class of formal series
which assumes that two formal series are equal if and only if their coefficients
coincide. We will use induction to show that the equation can be satisfied
at every order in $\delta^2$.

We note that in (\ref{eqnToSolveSeries}) the least order of $\delta$ is $6$.
Collecting all terms of this order we get the following equation
\begin{equation}
a_{1}\dot{x}_{1}^{2}+u_{30}x_{1}^{3}+u_{11}x_{1}=c_{3}.
\label{eqnToSolve_n3}
\end{equation}
This equation looks very similar to (\ref{EtaDotSq}) but has different
coefficients in front of its terms. This suggests $x_1$ to be of the form
\begin{equation}
x_{1}=b_{0}+b_{1}\eta _{0},  \label{def_x1}
\end{equation}%
where $b_{0},b_{1}\in \mathbb{R}$ are to be determined. We insert $x_{1}$
into $\left( \ref{eqnToSolve_n3}\right) $ which gives%
\begin{equation*}
a_{1}(b_{1}\dot{\eta}_{0})^{2}+u_{30}(b_{0}+b_{1}\eta
_{0})^{3}+u_{11}(b_{0}+b_{1}\eta _{0})=c_{3}.
\end{equation*}%
Using equation $\left( \ref{EtaDotSq}\right) $ to replace $\dot{\eta}_{0}^{2}
$ and expanding the parentheses we rewrite this expression as follows%
\begin{eqnarray}
c_{3} &=&a_{1}b_{1}^{2}\left( \eta _{0}^{2}-\eta _{0}^{3}\right)
+u_{30}b_{0}^{3}+u_{30}3b_{0}^{2}b_{1}\eta _{0}+u_{30}3b_{0}b_{1}^{2}\eta
_{0}^{2}  \label{c3eqn} \\
&&+u_{30}b_{1}^{3}\eta _{0}^{3}+u_{11}b_{0}+u_{11}b_{1}\eta _{0}.  \notag
\end{eqnarray}%
We now collect terms in $\left( \ref{c3eqn}\right) $ of equal powers in $%
\eta _{0}$ and obtain the following system of equations
\begin{equation}
\left\{ 
\begin{array}{l}
-a_{1}b_{1}^{2}+u_{30}b_{1}^{3}=0,\bigskip  \\
a_{1}b_{1}^{2}+u_{30}3b_{0}b_{1}^{2}=0,\bigskip  \\
3u_{30}b_{0}^{2}b_{1}+u_{11}b_{1}=0,\bigskip  \\
c_{3}=u_{11}b_{0}+u_{30}b_{0}^{3}.%
\end{array}%
\right.
\end{equation}%
The last equation defines $c_{3}$. The first three lines imply
\begin{equation}
\left\{ 
\begin{array}{l}
b_{0}=\pm \sqrt{-\frac{u_{11}}{3u_{30}}},\bigskip  \\ 
b_{1}=-3b_{0},\bigskip  \\ 
a_{1}=-3u_{30}b_{0.}%
\end{array}%
\right. .  \label{n3_coeff}
\end{equation}%
We make the following two remarks. Note that, to the leading order, 
$b_{0}$ is the $x$ coordinate of the equilibrium point and $c_{3}$ is its energy.

The sign of $b_{0}$ should be chosen to ensure $a_{1}>0$
since the latter coefficient is the first term of the expansion of $\beta(\varepsilon )$
which is a square of a real series. 

We have seen that equation $\left( \ref{eqnToSolveSeries}\right) $ has a
solution at the leading order $\delta ^{6}$, and continue by induction. 
We explain the first step of the induction in more detail 
in order to illustrate the method used in the general step.

Let us collect the terms of order $\delta ^{8}$ 
which gives the equation
\begin{equation}
2a_{1}\dot{x}_{1}\dot{x}_{2}+
a_{2}\dot{x}_{1}^{2}+u_{11}x_{2}+u_{21}x_{1}^{2}+3u_{30}x_{1}^{2}x_{2}+u_{40}x_{1}^{4}=c_{4}.
\label{eqnToSolve_n4}
\end{equation}
We note that this is a linear non-homogeneous equation in $x_2$.
Therefore it defines $x_2$ uniquely up to addition of a solution 
of the corresponding homogeneous equation.

Using the form of $x_{1}$ given in equation (\ref{def_x1}) 
and the relationship (\ref{EtaDotSq}) we rearrange this equation as
\begin{eqnarray}
c_{4} &=&2a_{1}\dot{x}_{1}\dot{x}_{2}+\left( 3u_{30}x_{1}^{2}+u_{11}\right)
x_{2}  \label{eqnToSolve_n4'} \\
&&+a_{2}b_{1}^{2}(\eta _{0}^{2}-\eta _{0}^{3})+u_{21}(b_{0}+b_{1}\eta
_{0})^{2}+u_{40}(b_{0}+b_{1}\eta _{0})^{4}.  \notag
\end{eqnarray}

Before solving this equation for each power of $\eta _{0}$ separately we
make some observations which will simplify our analysis. First we note that
the second line of $\left( \ref{eqnToSolve_n4'}\right) $ is a polynomial in $%
\eta _{0}$ of order $4$. The first line of $\left( \ref{eqnToSolve_n4'}%
\right) $ tells us that $x_{2}$ should have a pole of order $4$, hence we
assume that $x_{2}$ takes the following form%
\begin{equation*}
x_{2}=b_{20}+b_{21}\eta _{0}+b_{22}\eta _{0}^{2}.
\end{equation*}%
Now let us simplify the first line of $\left( \ref{eqnToSolve_n4'}\right) $
beginning with the second term%
\begin{eqnarray*}
\left( 3u_{30}x_{1}^{2}+u_{11}\right)  &=&3u_{30}(b_{0}+b_{1}\eta
_{0})^{2}+u_{11} \\
&=&3u_{30}b_{0}^{2}+6u_{30}b_{0}b_{1}\eta _{0}+3u_{30}b_{1}^{2}\eta
_{0}^{2}+u_{11}.
\end{eqnarray*}%
Using $\left( \ref{n3_coeff}\right) $ we conclude that%
\begin{equation}\label{internal_ref}
\left( 3u_{30}x_{1}^{2}+u_{11}\right) = 3u_{30}b_{1}b_{0}(2\eta _{0}-3\eta
_{0}^{2})
\end{equation}%
To ease the notation we let%
\begin{equation}
A=3u_{30}b_{1}b_{0},  \label{def_A}
\end{equation}%
and we notice that $A\neq 0$ by virtue of $\left( \ref{n3_coeff}\right)$.
Next we simplify the first term in the first line of $\left( \ref%
{eqnToSolve_n4'}\right) $%
\begin{eqnarray*}
\dot{x}_{1}\dot{x}_{2} &=&\left( b_{21}+2b_{22}\eta _{0}\right) \dot{\eta}%
_{0}b_{1}\dot{\eta}_{0} \\
&=&b_{1}\left( b_{21}+2b_{22}\eta _{0}\right) \left( \eta _{0}^{2}-\eta
_{0}^{3}\right) ,
\end{eqnarray*}%
where in the last step we used $\left( \ref{EtaDotSq}\right) .$ Then
equation $\left( \ref{eqnToSolve_n4'}\right) $ reads%
\begin{eqnarray}
c_{4} &=&2a_{1}b_{1}\left( b_{21}+2b_{22}\eta _{0}\right) \left( \eta
_{0}^{2}-\eta _{0}^{3}\right)   \label{eqnToSolve_n4''} \\
&&+A(2\eta _{0}-3\eta _{0}^{2})\left( b_{20}+b_{21}\eta _{0}+b_{22}\eta
_{0}^{2}\right)   \notag \\
&&+a_{2}b_{1}^{2}(\eta _{0}^{2}-\eta _{0}^{3})+u_{21}(b_{0}+b_{1}\eta
_{0})^{2}+u_{40}(b_{0}+b_{1}\eta _{0})^{4}.  \notag
\end{eqnarray}

Let us now solve $\left( \ref{eqnToSolve_n4''}\right) $ for $x_{2},c_{4}$
and $a_{2}.$ The strategy is to solve $\left( \ref{eqnToSolve_n4''}\right) $
in each power of $\eta _{0}$ separately in the following order
\begin{enumerate}
\item
$c_{4}$ is defined from the order $0$ in~$\eta _{0}$.
\item
$b_{20}$ is defined from the order $1$ in~$\eta _{0}$.
\item
$b_{22}$ is defined from the order $4$ in~$\eta _{0}$.
\item
$b_{21}$ and $a_{2}$ solve a system of equations obtained from the orders 
$2$ and $3$ in~$\eta _{0}$.
\end{enumerate}
At order $0$ in~$\eta _{0}$ we get
\begin{equation}
c_{4}=u_{21}b_{0}^{2}+u_{40}b_{0}^{4},  \label{n4_coeff2}
\end{equation}%
which defines the coefficient $c_{4}.$ Next we proceed with the terms of
order $1$ in $\eta _{0}$%
\begin{equation*}
2Ab_{20}+u_{21}2b_{0}b_{1}+u_{40}4b_{0}^{3}b_{1}=0,
\end{equation*}%
which gives%
\begin{equation}
b_{20}=-\frac{u_{21}2b_{0}b_{1}+u_{40}4b_{0}^{3}b_{1}}{2A}.
\label{n4_coeff3}
\end{equation}%
Since $A$ is nonzero equation $\left( \ref{n4_coeff3}\right) $ defines the
coefficient $b_{20}$ uniquely. At order $4$ in $\eta _{0}$ we have%
\begin{equation*}
-(4a_{1}b_{1}+3A)b_{22}+u_{40}b_{1}^{4}=0,
\end{equation*}%
hence%
\begin{equation}
b_{22}=\frac{u_{40}b_{1}^{4}}{4a_{1}b_{1}+3A}.  \label{n4_coeff1}
\end{equation}%
The denominator in $\left( \ref{n4_coeff1}\right) $ is non-zero since, using 
$\left( \ref{n3_coeff}\right) $ and $\left( \ref{def_A}\right) ,$%
\begin{equation*}
4a_{1}b_{1}+3A=-12u_{30}b_{0}b_{1}+9u_{30}b_{0}b_{1}=9u_{30}b_{0}^{2},
\end{equation*}%
and $u_{30}\neq 0$ by assumption and $b_{0}\neq 0$ by $\left( \ref{n3_coeff}%
\right) .$ Therefore $\left( \ref{n4_coeff1}\right) $ determines $b_{22}$
uniquely. The terms of order $3$ in $\eta _{0}$ are%
\begin{equation*}
-a_{2}b_{1}^{2}+4u_{40}b_{0}b_{1}^{3}+(-2a_{1}b_{1}-3A)b_{21}+4a_{1}b_{1}b_{22}+2Ab_{22}=0,
\end{equation*}%
and the terms of order $2$ in $\eta _{0}$ are%
\begin{equation*}
-3Ab_{20}+2a_{1}b_{1}b_{21}+2Ab_{21}+a_{2}b_{1}^{2}+u_{21}b_{1}^{2}+u_{40}6b_{0}^{2}b_{1}^{2}=0.
\end{equation*}%
These two terms define the following system of equations in $b_{21}$ and $%
a_{2}$%
\begin{equation}
\left( 
\begin{array}{cc}
2a_{1}b_{1}+2A & b_{1}^{2} \\ 
-2a_{1}b_{1}-3A & -b_{1}^{2}%
\end{array}%
\right) \left( 
\begin{array}{c}
b_{21} \\ 
a_{2}%
\end{array}%
\right) =\left( 
\begin{array}{c}
3Ab_{20}-u_{21}b_{1}^{2}-u_{40}6b_{0}^{2}b_{1}^{2} \\ 
-4a_{1}b_{1}b_{22}-2Ab_{22}-u_{40}4b_{0}b_{1}^{3}%
\end{array}%
\right) .  \label{n4_coeff4}
\end{equation}%
Call the matrix in the left-hand side $M.$ This system of equations has a
unique solution since%
\begin{equation*}
\det M=Ab_{1}^{2}=-81u_{30}b_{0}^{4}\neq 0.
\end{equation*}%
Therefore $b_{21}$ and $a_{2}$ is the unique solution of $\left( \ref%
{n4_coeff4}\right) $. We have now shown that $x_{2},a_{2}$ and $c_{4}$ are
uniquely defined.

Now let us continue with the inductive step at an arbitrary order of $\delta
^{2}$. For a general $n$ the terms corresponding to $\delta ^{2n}$ are%
\begin{equation*}
\sum_{\substack{ k+l+m=n \\ k,l,m\geq 1}}a_{k}\dot{x}_{l}\dot{x}%
_{m}+\sum_{k+2m\geq 3}u_{km}\sum_{\substack{ j_{1}+\ldots +j_{k}=n-2m \\ %
j_{1},\ldots ,j_{k}\geq 1}}x_{j_{1}}\ldots x_{j_{k}}=c_{n}.
\end{equation*}%
This equation has the form%
\begin{equation}
2a_{1}\dot{x}_{1}\dot{x}_{n}+(3u_{30}x_{1}^{2}+u_{11})x_{n}+\widetilde{P}%
(x_{1},\ldots ,x_{n-1})+a_{n}\dot{x}_{1}^{2}=c_{n},  \label{eqnToSolve_n}
\end{equation}%
where $\widetilde{P}(x_{1},\ldots ,x_{n-1})$ is a polynomial. We note that $%
\widetilde{P}$ is a (known) polynomial of order $n+2$ in $\eta _{0}$ and
that all coefficients are unique since our induction assumption is that $%
x_{1},\ldots ,x_{n-1}$ are unique. We will show by induction that there
exist a unique $a_{n}$ and a unique%
\begin{equation}
x_{n}=\sum_{m=0}^{n}b_{nm}\eta _{0}^{m}.  \label{def_xn}
\end{equation}%
That is, we will show that $a_{n}$ and all $b_{nm}$ are uniquely defined.

Using equality (\ref{internal_ref}) which states that%
\begin{equation*}
(3u_{30}x_{1}^{2}+u_{11})=A(2\eta _{0}-3\eta _{0}^{2}),
\end{equation*}%
where $A=3u_{30}b_{1}b_{0}\neq 0$ (see equation $\left( \ref{def_A}\right) $%
) and that $x_{n}$ has the form $\left( \ref{def_xn}\right) $ we simplify $%
\left( \ref{eqnToSolve_n}\right) $ as follows%
\begin{eqnarray}
c_{n} &=&2a_{1}b_{1}(\eta _{0}^{2}-\eta _{0}^{3})\left( b_{n1}+2b_{n2}\eta
_{0}+\ldots nb_{nn}\eta _{0}^{n-1}\right)  \label{eqnToSolve_n'} \\
&&+A(2\eta _{0}-3\eta _{0}^{2})\left( b_{n0}+\ldots +b_{nn}\eta
_{0}^{n}\right)  \notag \\
&&+a_{n}b_{1}^{2}(\eta _{0}^{2}-\eta _{0}^{3})+\widetilde{P}(x_{1},\ldots
,x_{n-1}).  \notag
\end{eqnarray}%
We will now solve this equation for each power in $\eta _{0}$ separately.
The strategy is the following
\begin{enumerate}
\item
The power $0$ will define $c_{n}$ uniquely.
\item
The power $1$ will define $b_{n0}$ uniquely.
\item
The power $n+2$ will define $b_{nn}$ uniquely.
\item
The powers $j\in \left\{ 4,\ldots ,n+1\right\} $ (starting from $j=n+1$
and proceeding in decreasing order) will define the coefficients $b_{n,j-2}$
uniquely.
\item Finally, at powers $2$ and $3$ we will obtain a system of equations in $%
b_{n1}$ and $a_{n}$ which we show has a unique solution.
\end{enumerate}
In what follows we use $\left[ \widetilde{P}\right]
_{k}$ to denote the terms in the polynomial $\widetilde{P}(x_{1},\ldots
,x_{n-1})$ which are of order $k$ in $\eta _{0}.$ Let us now start solving $%
\left( \ref{eqnToSolve_n'}\right) $ at order $0$ in $\eta _{0}$. At this
order we have%
\begin{equation*}
c_{n}=\left[ \widetilde{P}\right] _{0},
\end{equation*}%
which defines $c_{n}$ uniquely.
Next, at order $1$ in $\eta _{0}$ we have%
\begin{equation*}
2Ab_{n0}+\left[ \widetilde{P}\right] _{1}=0,
\end{equation*}%
thus%
\begin{equation*}
b_{n0}=-\frac{\left[ \widetilde{P}\right] _{1}}{2A}.
\end{equation*}%
Since $A\neq 0$ this equation defines $b_{n0}$ uniquely.

At order $n+2$ in $\eta _{0}$ we get%
\begin{equation*}
-2a_{1}nb_{1}b_{nn}-3Ab_{nn}+\left[ \widetilde{P}\right] _{n+2}=0,
\end{equation*}%
thus%
\begin{equation*}
b_{nn}=\frac{\left[ \widetilde{P}\right] _{n+2}}{2na_{1}b_{1}+3A}.
\end{equation*}%
By equation $\left( \ref{n3_coeff}\right) $ $a_{1}b_{1}=A,$ thus the
denominator is equal to $A(2n+3)$ which is non-zero, hence $b_{nn}$ is
defined uniquely.

Now we continue with the orders $j\in \left\{ 4,\ldots ,n+1\right\} $. At
order $j$ we solve the equation%
\begin{equation*}
b_{n,j-2}\left( 2a_{1}b_{1}(j-2)+3A\right) -b_{n,j-1}(2a_{1}b_{1}(j-1)+2A)-
\left[ \widetilde{P}\right] _{j}=0,
\end{equation*}%
thus%
\begin{equation*}
b_{n,j-2}=\frac{b_{n,j-1}(2a_{1}b_{1}(j-1)+2A)+\left[ \widetilde{P}\right]
_{j}}{2a_{1}b_{1}(j-2)+3A}.
\end{equation*}%
Since we are proceeding in decreasing $j$ the coefficient $b_{n,j-1}$ in the
right-hand side has already been determined. The denominator%
\begin{equation*}
2a_{1}b_{1}(j-2)+3A=A(2(j-2)+3)
\end{equation*}%
is non-zero since $j-2$ is positive for all $j\in $ $\left\{ 4,\ldots
,n+1\right\} .$ Therefore the coefficients $b_{n,j-2}$ are uniquely
determined for $j\in \left\{ 4,\ldots ,n+1\right\} $.

Finally we have to solve for the remaining coefficients $a_{n}$ and $b_{n1}.$
They are obtained from the equations at order $2$ and $3$ which are given by%
\begin{equation*}
2a_{1}b_{1}b_{n1}+2Ab_{n1}-3Ab_{n0}+a_{n}b_{1}^{2}+\left[ \widetilde{P}\right]
_{2}=0
\end{equation*}%
and%
\begin{equation*}
4a_{1}b_{1}b_{n2}-2a_{1}b_{1}b_{n1}+2Ab_{n2}-3Ab_{n1}-a_{n}b_{1}^{2}+\left[ 
\widetilde{P}\right] _{3}=0,
\end{equation*}%
respectively. We observe that this system of equations may be written as%
\begin{equation}
\left( 
\begin{array}{cc}
2a_{1}b_{1}+2A & b_{1}^{2} \\ 
-2a_{1}b_{1}-3A & -b_{1}^{2}%
\end{array}%
\right) \left( 
\begin{array}{c}
b_{n1} \\ 
a_{n}%
\end{array}%
\right) =\left( 
\begin{array}{c}
3Ab_{n0}-\left[ \widetilde{P}\right] _{2} \\ 
-4a_{1}b_{1}b_{n2}-2Ab_{n2}-\left[ \widetilde{P}\right] _{3}%
\end{array}%
\right) .  \label{n_system}
\end{equation}%
The determinant of the matrix in the left-hand side is $%
Ab_{1}^{2}=-81u_{30}b_{0}^{4}$ which is non-zero, hence $b_{n1}$ and $a_{n}$ is
the unique solution of $\left( \ref{n_system}\right) .$

We have now shown that $b_{n0},\ldots ,b_{nn}$ and $a_{n}$ and $c_{n}$ are
uniquely defined, hence $x_{n}$ is a unique solution of $\left( \ref%
{eqnToSolve_n}\right) $.

It follows by induction that $x_{m}$ is uniquely defined for all $m.$

\medskip

Now we can reconstruct the second component of the solution 
of system~(\ref{Eq:system}). Since
\[
x=\sum\limits_{k\geq 1}\delta ^{2k}x_{k}(t)
\]
the second component of the solution is restored by
\begin{equation*}
y= \mu (\varepsilon ) \dot{x}=
\left( 2\sum\limits_{k\geq1}a_{k}\delta ^{2k}\right) ^{\frac{1}{2}}
\sum\limits_{k\geq 1}\delta ^{2k}\dot x_{k}(t) .%
\end{equation*}
Since $x_{k}(t)$ is a polynomial in $\eta _{0}$ of order $k$
we conclude that the solution of the system (\ref{Eq:system}) can be written in the form
\begin{equation}
\left\{ 
\begin{array}{l}
x=\sum\limits_{k\geq 1}\delta ^{2k}x_{k}(t),\bigskip  \\ 
y=\eta_{1}(t)\sum\limits_{k\geq 1}\delta ^{2k+1}y_{k-1}(t),%
\end{array}%
\right.   \label{solution}
\end{equation}%
where $y_{k-1}$ is a polynomial in $\eta _{0}$ of order $k-1$.

\medskip

The solution $\left( \ref{solution}\right) $ is a formal solution of the
system of equations $\left( \ref{Eq:system}\right) .$ Equation $\left( \ref%
{Eq:system}\right) $ was obtained from $\left( \ref{Eq:mainforma}\right) $ by a
change of variables%
\begin{equation}
\mathbf{x}=\phi _{\chi _{\varepsilon }}^{-1}(\mathbf{X})
\label{VarChange_formal}
\end{equation}%
where $\chi _{\varepsilon }$=$\sum_{p\ge6}\chi _{p}$ where $\chi _{p}$ is a
quasi-homogeneous polynomial of order $p$. To finish the proof of the theorem
we have to invert this change of variables and consider the structure of the
resulting formal series $\mathbf{X}$. As the change of variables is close-to-identity (in
the quasi-homogeneous sense) it is easy to invert $\left( \ref%
{VarChange_formal}\right) .$ For our purposes it is sufficient to consider
the change in its most general form, hence%
\begin{equation}
\mathbf{X}=\phi _{\chi _{\varepsilon }}^{1}(\mathbf{x})=\left( 
\begin{array}{c}
x+\sum\limits_{2k+3l+4m\geq 3}d_{klm}x^{k}y^{l}\delta ^{4m}\medskip  \\ 
y+\sum\limits_{2k+3l+4m\geq 4}f_{klm}x^{k}y^{l}\delta ^{4m}%
\end{array}%
\right) ,  \label{VarChange_formal2}
\end{equation}%
where $d_{klm},f_{klm}$ are coefficients and $\left( x,y\right) $ are given
by $\left( \ref{solution}\right) .$

Let us now consider the structure of $\mathbf{X},$ that is, let us consider
the sums in the right hand side of the previous equation. Obviously $\delta
^{4m}$ is an even power of $\delta .$ We also note that any power of the
series $x$ is a series of the same type as $x$ itself, i.e. a series in even
powers of $\delta .$ This is not the case with $y\,.$ An even power of the
series $y$ gives a series in \emph{even} powers of $\delta $ while an odd
power of $y$ gives a series in \emph{odd} powers of $\delta .$ Hence $%
x^{k}y^{l}\delta ^{4m}$ is a series in even powers of $\delta $ if $l$ is
even and in odd powers if $l$ is odd.

Furthermore, since $\left( x,y\right) $ are series of polynomials in $\eta
_{0}(t)$ and $\dot{\eta}_{0}(t)$ any power of them are series of polynomials
as well. Hence we conclude that $\mathbf{X}$ is a power series in both odd
and even powers of $\delta $ and that each power of $\delta $ is multiplied
by a polynomial in $\eta _{0}(t)$ and $\dot{\eta}_{0}(t).$

A more careful study reveals even more about the structure. Consider $\left( %
\ref{solution}\right) $ again. We see that taking any power of $x$ gives a
series of terms of the type%
\[
\delta ^{2k}P_{k}(t),
\]%
where $P_{k}$ is a quasi-homogeneous polynomial in $\eta _{0}$ of order $k.$
When taking powers of $y$ we note that%
\[
\eta _{1}^{k}=\dot{\eta}_{0}^{k}=\left\{ 
\begin{array}{c}
(\eta _{0}^{2}-\eta _{0}^{3})^{k/2},\quad p\text{ even\medskip } \\ 
\dot{\eta}_{0}(\eta _{0}^{2}-\eta _{0}^{3})^{(k-1)/2},\quad p\text{ odd}%
\end{array}%
\right. .
\]%
Bearing this in mind an even power of $y$ gives a series with terms of the
type%
\[
\delta ^{2k}Q_{\leq k}(t),
\]%
where $Q_{\leq k}(t)$ is a homogeneous polynomial in $\eta _{0}(t)$ of
order \emph{at most }$k.$ Similarly, taking an odd power of $y$ gives a
series with terms of the type%
\[
\delta ^{2k+1}Q_{\leq k-1}(t),
\]%
where $Q_{\leq k-1}(t)$ is a polynomial in $\eta _{0}(t)$
of order \emph{at most }$k-1.$ Based on these observations we conclude that
the form of the formal separatrix $\mathbf{X}$ is
\begin{equation}\label{Eq:Derived_formal}
\mathbf{X}(t,\varepsilon )=\left( 
\begin{array}{c}
\sum\limits_{p\geq 1}\delta ^{2p}x_{p}^{1}+\dot{\eta}_{0}(t)\sum\limits_{p%
\geq 1}\delta ^{2p+1}x_{p-1}^{2}\medskip  \\ 
\dot{\eta}_{0}(t)\sum\limits_{p\geq 1}\delta
^{2p+1}y_{p-1}^{2}+\sum\limits_{p\geq 2}\delta ^{2p}y_{p}^{1}%
\end{array}%
\right) ,
\end{equation}%
where $x_{p}^{1},x_{p}^{2},y_{p}^{1},y_{p}^{2}$ denotes 
polynomials in $\eta _{0}(t)$ of order $p.$ This finishes the proof of the
theorem.

\subsection{Re-expansion near the singularity}\label{Se:laurentexp}

In this section we will derive equation $\left( \ref{Eq:formal_laurent}\right),$ i.e.
we will expand the formal separatrix $\mathbf{X}$ in its Laurent series
around the singularity $t=i\pi .$ We do this by first expanding the two
"base" functions $\eta _{0}(t)$ and $\dot{\eta}_{0}(t)$ into their Laurent
series around the singularity. The Laurent series of $\eta _{0}(t)$ 
at $t=i\pi $ is given by%
\begin{equation*}
\eta _{0}(t-i\pi )=-\frac{4}{t^{2}}\left( 1+\sum_{k=1}^{\infty
}c_{k}t^{2k}\right) ,
\end{equation*}%
where $c_{k}\in \mathbb{R}$ $\forall k\in \mathbb{N}.$ The distance to the
next singularity is $2\pi $, which is then the radius of convergence for
this series. Secondly we note that the Laurent expansion of $\eta _{1}(t)$
at $t=i\pi $ is given by%
\begin{equation*}
\eta _{1}(t-i\pi )=\frac{8}{t^{3}}\left( 1+\sum_{k=1}^{\infty
}d_{k}t^{2k}\right) ,
\end{equation*}%
where $d_{k}\in \mathbb{R}$ $\forall k\in \mathbb{N}$ and it converges for $%
\left\vert t\right\vert <2\pi .$

We are now ready to determine the Laurent expansion of $\left( \ref%
{Eq:Derived_formal}\right) .$ Let us consider its first component which we write as%
\begin{equation}
\left[ \mathbf{X}\right] _{n-1}=\sum_{k=2}^{n-1}\delta ^{k}\psi _{k}^{1}(t)
\label{FormalSep_2}
\end{equation}%
where $\psi _{k}^{1}(t)$ is the term corresponding to $\delta ^{k}$ in the
formal separatrix $\left( \ref{Eq:Derived_formal}\right) .$ $\psi _{k}^{1}\,$\ has a
pole of order $k.$ Using the Laurent expansions of $\eta_{0}(t)$ and $\eta_{1}(t)$ 
respectively, the Laurent expansion of each $\psi _{k}^{1}(t)$ is
given by%
\begin{equation*}
\psi _{k}^{1}(t)=\frac{1}{(t-i\pi )^{k}}\sum_{j=0}^{\infty }(t-i\pi
)^{2j}\psi _{kj},
\end{equation*}%
where $\psi _{kj}$ are the coefficients of the Laurent expansion of $\psi _{k}^{1}(t)$. 
Next we substitute $\tau =\frac{t-i\pi }{\log \lambda },$ where we re-expand $\log
\lambda =\delta \sum_{l\geq 0}\lambda _{l}\delta ^{2l},$ this gives%
\begin{equation*}
\psi _{k}^{1}(t)=\frac{1}{\tau ^{k}\delta ^{k}(\sum_{l\geq 0}\lambda
_{l}\delta ^{2l})^{k}}\sum_{j=0}^{\infty }\tau ^{2j}\delta ^{2j}(\sum_{l\geq
0}\lambda _{l}\delta ^{2l})^{2j}\psi _{kj}.
\end{equation*}%
Inserting this expression into $\left( \ref{FormalSep_2}\right) $ we note
that $\delta ^{-k}$ and $\delta ^{k}$ cancel each other and we obtain an expansion in
even powers of $\delta $%
\begin{equation*}
\sum_{k=2}^{n-1}\delta ^{k}\psi _{k}^{1}(t)=\sum_{k=2}^{n-1}\frac{1}{\tau
^{k}}\sum_{j=0}^{\infty }\tau ^{2j}\delta ^{2j}(\sum_{l\geq 0}\lambda
_{l}\delta ^{2l})^{2j-k}\psi _{kj}.
\end{equation*}%
Since $\lambda (\delta )$ is even, analytic at $\delta =0$ and $\lambda (0)=1
$ we use that $\lambda ^{m}$ expands as\footnote{Do we need all these properties 
to draw this conclusion?}%
\begin{equation*}
\lambda ^{m}=\left( \sum_{l=0}^{\infty }\lambda _{l}\delta ^{2l}\right)
^{m}=\sum_{l=0}^{\infty }\lambda _{m,l}\delta ^{2l}
\end{equation*}%
to obtain 
\begin{equation*}
\sum_{k=2}^{n-1}\delta ^{k}\psi _{k}^{1}(t)=\sum_{k=2}^{n-1}\frac{1}{\tau
^{k}}\sum_{j=0}^{\infty }\sum_{l=0}^{\infty }\psi _{kj}\tau ^{2j}\delta
^{2(l+j)}\lambda _{2j-k,l}.
\end{equation*}%
We will now change the summation indices in two steps to simplify this
expression. First we let $m=j+l$ which allow us to write the sums as%
\begin{equation*}
\sum_{k=2}^{n-1}\delta ^{k}\psi _{k}^{1}(t)=\sum_{k=2}^{n-1}\sum_{m\geq
0}^{{}}\delta ^{2m}\sum_{j+l=m}^{{}}\psi _{kj}\tau ^{2j-k}\lambda _{2j-k,l}.
\end{equation*}%
Next we let $p=2j-k$ which gives us the desired result%
\begin{eqnarray*}
\sum_{k=2}^{n-1}\delta ^{k}\psi _{k}^{1}(t) &=&\sum_{m\geq 0}\delta
^{2m}\sum_{k=2}^{n-1}\sum_{\substack{ 2j-k=p \\ l+j=m}}\tau ^{p}\psi
_{kj}\lambda _{p,l} \\
&=&\sum_{m\geq 0}\delta ^{2m}\sum_{p=-n+1}^{2m-2}\tau ^{p}\widetilde{\psi }%
_{m,p}^{(1)}.
\end{eqnarray*}%
The derivation of the Laurent expansion of the second component of $\left( %
\ref{Eq:Derived_formal}\right) $ is analogous. The result is
\begin{equation}
[\mathbf{X}]_{n-1}=\sum_{k=2}^{n-1}\delta ^{k}\psi _{k}(t)=\left( 
\begin{array}{c}
\sum\limits_{m\geq 0}\delta ^{2m}\sum\limits_{p=-n+1}^{2m-2}\tau ^{p}%
\widetilde{\psi }_{m,p}^{(1)}\bigskip  \\ 
\sum\limits_{m\geq 0}\delta ^{2m}\sum\limits_{p=-n+1}^{2m-3}\tau ^{p}%
\widetilde{\psi }_{m,p}^{(2)}%
\end{array}%
\right) .  \label{LaurentSeries}
\end{equation}%
This finishes the derivation of the Laurent expansion of the formal separatrix
around the singularity $t=i\pi$.

\section{Close to identity maps\label{Se:closetoid}}

In this section study an analytic family of close to identity maps
of the form
\begin{equation}\label{Eq:closetoid}
F_{\delta }(\mathbf{x})=%
\mathbf{x}+\delta G_{\delta }(\mathbf{x})
\end{equation}
where $\delta $ is a small parameter. 
Expanding a solution of the differential equation 
$\dot{\mathbf{x}}=G_{0}(\mathbf{x})$ into Taylor series in time,
we can easily check that 
$F_\delta$ is approximated by a map which shifts
a point along a trajectory of the differential equation
by time $\delta$:
\[
F_{\delta}=\Phi^\delta_{G_0}+O(\delta^2)\,.
\]
We say that $G_0$ generate {\em the limit flow\/} associated with the
map $F_{\delta }$.

If $F_{\delta }$ is a family of 
area-preserving maps then the limit flow is divergence free and 
has a (possibly local) Hamiltonian function $H_{0}$, i.e.,
\[
G_{0}=\mathrm{J}\nabla H_{0}
\]
where $\mathrm{J}$ is the standard two dimensional  symplectic matrix.

\subsection{Fixed points and their multipliers\label{Se:fpm}}

\begin{theorem}
\label{Existence_FixedPoint}
If $\mathbf{p}_{0}$ is a hyperbolic saddle of the limit flow $G_{0}$ 
then there is a unique smooth family $\mathbf{p}_\delta$ of fixed points 
of $F_{\delta }$ which converges to $\mathbf{p}_{0}$ as $\delta \rightarrow 0.$ 
Moreover, the fixed point depends analytically on $\delta $ and is hyperbolic with multiplier
$\lambda_\delta$. The multiplier depends analytically on $\delta$ and
\begin{equation*}
\lambda _{\delta }=1+\delta \mu_0+O(\delta ^{2}),
\end{equation*}%
where $\mu_0$ is an  eigenvalue of $DG_{0}(\mathbf{p}_{0})$.
\end{theorem}

\begin{proof}
The theorem easily follows from the implicit function theorem.
Indeed, by the definition of a fixed point we have
\begin{equation*}
\mathbf{p}_{\delta }=F_{\delta }(\mathbf{p}_{\delta }),
\end{equation*}%
which is equivalent to
\begin{equation}
G_{\delta }(\mathbf{p}_{\delta })=\mathbf{0}  \label{G}
\end{equation}
due to the form of the map (\ref{Eq:closetoid}).
Since $F_{\delta }$ is an analytic function, $G_{\delta }$ is analytic. 
Since $\mathbf{p}_{0}$ is an equilibrium of the limit flow
we have
\begin{equation*}
G_{0}(\mathbf{p}_{0})=\mathbf{0}\,.
\end{equation*}%
Since $\mathbf{p}_{0}$ is a saddle equilibrium of $G_0$ one of the eigenvalues
of $DG_{0}(\mathbf{p}_{0})$ is positive and the other one is negative. Therefore
\begin{equation*}
\det DG_{0}(\mathbf{p}_{0})<0\,,
\end{equation*}
and in particular $\det DG_{0}(\mathbf{p}_{0})\ne0$.
Then by the
implicit function theorem there exists an analytic solution of $\left( \ref%
{G}\right) $ for all $\delta $ such that $\left\vert \delta \right\vert
<\delta _{0}.$

Let us now consider the eigenvalues and eigenvectors of the fixed point. Let%
\begin{equation*}
\mathrm{A}_{\delta }:=DF_{\delta }(\mathbf{p}_{\delta })=\mathrm{id}+\delta\, DG_{\delta }(%
\mathbf{p}_{\delta }).
\end{equation*}
The eigenvalues of $DG_{0}(\mathbf{p}_{0})$ are simple, 
hence eigenvalues and normalised eigenvectors of 
$DG_{\delta }(\mathbf{p}_{\delta })$ are analytic in $\delta$.
The matrices $DG_{\delta }(\mathbf{p}_{\delta })$
and $\mathrm{A}_{\delta }$ have the same eigenvectors.
Moreover, if $\mu (\delta )$ is an eigenvalue of 
$DG_{\delta }(\mathbf{p}_{\delta })$, then 
$\lambda _{\delta}=1+\delta \mu (\delta )$ 
is an eigenvalue of $\mathrm{A}_{\delta }$.
\end{proof}

\subsection{Formal interpolation by a flow\label{Se:formint}}

The results of this subsection are of independent interest.
We show that the map $F_\delta$ can be formally interpolated 
by an autonomous Hamiltonian flow. 

\begin{proposition} Let $F_\delta$ be an analytic family
of area-preserving maps {\rm (\ref{Eq:closetoid})}
defined on a simply connected domain $D\subset \mathbb{R}^2$ 
(or $D\subset \mathbb{C}^2$), 
then there is a formal Hamiltonian $H_\delta$ with coefficients analytical on $D$,
which formally interpolates $F_\delta$. The Hamiltonian $H_\delta$ is
defined uniquely up to addition of a formal series in $\delta$
only.
\end{proposition}

\begin{proof}
For a function $F_\delta$ let us denote by $[F_\delta]_n$ the $n^{\mathrm{th}}$-order
coefficient of the Taylor series in $\delta$. In coordinates we write
$$
[F_\delta]_n=(f_n(x,y),g_n(x,y)).
$$
We construct the formal series
$$
H_\delta(x,y,\delta)=\sum_{k\ge0} \delta^k H_k(x,y),
$$
such that for every $n$
\begin{equation}\label{Eq:condition}
[\Phi^\delta_{H_\delta^n}]_{\ell }=[F_\delta]_\ell
\qquad\mbox{for all $1\le \ell\le n$}
\end{equation}
where $\Phi^t_{H_\delta^n}$ is the Hamiltonian flow
generated by the Hamiltonian
\begin{equation}\label{Eq:Hdn}
H_{\delta}^n(x,y)=\sum_{k=0}^{n-1} \delta^k H_k(x,y).
\end{equation}
It is convenient to write the Hamiltonian flow
generated by a Hamiltonian $H_\delta^n$
using Lie series
\[
\Phi_{H_\delta^n}^\delta(\p)=\p+\sum_{n=1}^\infty \frac{\delta^n L_{H_\delta^n}^n(\p)}{n!}\,,
\]
where $L_{H_\delta^n}$ is the Lie derivative generated by the Hamiltonian ${H_\delta^n}$.
The Lie derivative is linear in ${H_\delta^n}$.
Then taking into account (\ref{Eq:Hdn}) we get
\begin{equation}\label{Eq:Lee_n}
\left[\Phi^\delta_{H_\delta^{n}}\right]_{\ell}=
\sum_{m=1}^{\ell}\frac1{m!} \sum_{\substack{{n_1+n_2+\ldots n_m=\ell-m}\\0\le n_i\le n-1}}
 L_{H_{n_1}}\dots L_{H_{n_m}}(\p)\,.
\end{equation}
We construct $H_n$ inductively. Let $n=0$. Then (\ref{Eq:condition})
takes the form 
\[
[\Phi^\delta_{H_\delta^1}]_{1 }=[F_\delta]_1\,.
\]
Taking into account (\ref{Eq:Lee_n}) we rewrite it in the form
$$
L_{H_{0}}(\p)=G_0.
$$
Taking into account that in coordinates $\p=(x,y)$ and
\[
L_{H_{0}}(\p)=\left(\frac{\partial H_0}{\partial y},-\frac{\partial H_0}{\partial x}
\right)
\]
we see that (\ref{Eq:condition}) is satisfied for $n=0$ provided $H_0$ 
is the Hamiltonian of the limit flow.

We continue by induction. Suppose $H_0,\ldots,H_{n-1}$ are chosen in such a way
that the condition (\ref{Eq:condition}) is satisfied for some $n\ge 0$.
We show that $H_n$ can be chosen in such a way that (\ref{Eq:condition}) is satisfied for
$n$ replaced by $n+1$.

We note that (\ref{Eq:Lee_n}) implies that 
\[
\left[\Phi^\delta_{H_\delta^{n}}\right]_{\ell}=\left[\Phi^\delta_{H_\delta^{n+1}}\right]_{\ell}
\qquad\mbox{for $1\le \ell\le n$}
\]
and 
\[
\left[\Phi^\delta_{H_\delta^{n}}\right]_{n+1}=\left[\Phi^\delta_{H_\delta^{n+1}}\right]_{n+1}+L_{H_n}(\p)\,.
\]
Then the induction assumption implies that 
\[
\left[\Phi^\delta_{H_\delta^{n+1}}\right]_{\ell}=
\left[F_\delta\right]_{\ell}
\qquad\mbox{for $1\le \ell\le n$}\,.
\]
For $\ell=n+1$ the equality is achieved provided $H_n$ satisfies
the equation
\begin{equation}\label{Eq:Hn}
L_{H_n}(\p)=
\left[\Phi^\delta_{H_\delta^{n}}\right]_{n+1}-
\left[F_\delta\right]_{n+1}\,.
\end{equation}
This equation has a solution in $\mathcal D$  provided the right hand side
is divergence free. The solution is defined uniquely up
to addition of a constant. 

In order to check that the right hand side has indeed zero divergence,
 we make a simple observation.
Let $F_\delta$ and $\tilde F_\delta$ be two analytic families of
close-to-identity area preserving maps such that $[F_\delta]_\ell=[\tilde F_\delta]_{\ell}$
for all $0\le\ell\le n$. Then  
\begin{equation}\label{Eq:diveq}
\func{div}[F_\delta-\tilde F_\delta]_{n+1}=0.
\end{equation}
Indeed, let us write the maps in the coordinates:
\[
F_\delta(x,y)=\left(x+\sum_{k\ge1}\delta^ka_k(x,y), y+\sum_{k\ge1}\delta^k b_k(x,y)\right)
\]
and a similar formula for $\tilde F_\delta$. The preservation of the area is equivalent to the system
\begin{eqnarray*}
\partial_xa_1+\partial_y b_1&=&0,\\
\partial_x a_{n+1}+\partial_y b_{n+1}&=&-\sum_{\ell=1}^{n }
\{a_\ell,b_{n+1-\ell}\}, \qquad n\ge1.
\end{eqnarray*}
The system was obtained by expanding the Jacobian of $F_\delta$
in power series of $\delta$ and comparing the result with $1$.
Since $a_\ell=\tilde a_\ell$ and $b_\ell=\tilde b_\ell$ for $1\le \ell\le n$,
we immediately see that
\begin{equation}
\partial_x a_{n+1}+\partial_y b_{n+1}=
\partial_x \tilde a_{n+1}+\partial_y \tilde b_{n+1}
\end{equation}
which is equivalent to (\ref{Eq:diveq}).

Setting $\tilde F_\delta=\Phi^\delta_{H_\delta^{n}}$
and applying the last statement 
we see that the right hand side of (\ref{Eq:Hn})
has zero divergence and consequently 
$H_n$ is defined by the equation. The theorem follows by induction.
\end{proof}

\subsection{Approximation of the local separatrix\label{Se:prooflocalsep}}

In this section we prove that the local separatrix of
the map $F_{\delta }$ is close to a local separatrix
separatrix of any Hamiltonian flow which is $O(\delta ^{n+1})$ close to 
$F_{\delta }.$ The approximation is of order $O(\delta ^{n})$. We note that
the existence of a local separatrix near a hyperbolic fixed point of the map
and a hyperbolic equilibrium of the flow follows from the Hartman-Grobman
theorem.%
\footnote{The proof does not really uses the fact that $F$ is area-preserving
and the flow is Hamiltonian.}

The proof consists of two steps. First we transform the fixed
point of the map to the origin and diagonalise its linear part by
a linear symplectic change of variables. We also transform the Hamiltonian flow
into a similar form. Then we prove the theorem for the
transformed map and the transformed flow. Then we come back to the original
variables.

Theorem~\ref{Existence_FixedPoint} implies the map $F_{\delta }$ has a fixed
point $\mathbf{p}_{\delta }$ which is analytic in $\delta .$ Consider the
following change of variables%
\begin{equation*}
S:\mathbf{x}\mapsto \mathbf{p}_{\delta }+\mathrm{B}_{\delta }\mathbf{x},
\end{equation*}%
where $\mathrm{B}_{\delta }$ is a symplectic matrix which diagonalises $DF_{\delta }(%
\mathbf{p}_{\delta }).$ This change of variables transforms $F_{\delta }$ to 
\begin{equation*}
\widetilde{F}_{\delta }=S^{-1}\circ F_{\delta }\circ S\,.
\end{equation*}%
The change of variables is chosen to ensure that the fixed point
is moved to the origin
\begin{equation*}
\widetilde{F}_{\delta }(\mathbf{0})=\mathbf{0,}
\end{equation*}%
and the Jacobian of the map at the origin is diagonal:
\begin{equation*}
D\widetilde{F}_{\delta }(\mathbf{0})=\mathrm{A}_\delta:=\func{diag}\left( \lambda _{\delta },\lambda
_{\delta }^{-1}\right) .
\end{equation*}%
Note that the multipliers of the fixed point are left unchanged under any smooth change of variables.
After the change of coordinates the unstable separatrix is parametrised by the function
\begin{equation*}
\psi _{\delta }^{-}=S\circ \widetilde{\psi }{}_{\delta }^{-}
\end{equation*}
which satisfies the equation
\begin{equation}
\widetilde{\psi }{}_{\delta }^{-}(t+\log \lambda _{\delta })=\widetilde{F}%
_{\delta }(\widetilde{\psi }{}_{\delta }^{-}(t))\,.
\label{Separatrix_flow}
\end{equation}
We note that the vector field generated by $H^n_\delta$ has an equilibrium $\mathbf{q}_\delta$
which is close to $\mathbf{p}_\delta$ but does not necessary coincides with it.
In general the linear part of the flow is not exactly diagonal after the substitution $S$.
So we introduce another change of variables
\begin{equation*}
\widehat{S}:\mathbf{x}\mapsto \mathbf{q}_{\delta }+\widehat{\mathrm{B}}_{\delta }%
\mathbf{x,}
\end{equation*}%
where  $\widehat{\mathrm{B}}_{\delta }$ is a
symplectic matrix that diagonalises 
$\mathrm{J}\nabla ^{2}H_{\delta }^{n}(\mathbf{q}_{\delta }).$ 
Using the implicit function theorem and smooth dependence of eigenvectors
in a way similar to Section~\ref{Se:fpm} we conclude that
\begin{eqnarray}
\mathbf{p}_{\delta } &=&\mathbf{q}_{\delta }+O(\delta ^{n}),
\label{closeness_pandB} \\
\mathrm{B}_{\delta } &=&\widehat{\mathrm{B}}_{\delta }+O(\delta ^{n}).  \notag
\end{eqnarray}
Then we define a new Hamiltonian function
\begin{equation*}
\widetilde{H}_{\delta }^{n}=H_{\delta }^{n}\circ \widehat{S}.
\end{equation*}
Let ${\widetilde{\varphi}}_\delta^{n}$ denote the separatrix solution
of the differential equation
\begin{equation}
\mu _{n,\delta }\,\dot{\widetilde{\varphi}}
{\vphantom{widetilde{\varphi}}}_\delta^{n}
=\mathrm{J}\nabla \left. \widetilde{H}
_{\delta }^{n}\right\vert _{\widetilde{\varphi }_\delta^{n}}\,.
\label{Separatrix_map}
\end{equation}
The original separatrix solution can be easily restored by the formula
\begin{equation*}
\varphi _{\delta }^{n}=\widehat{S}\circ \widetilde{\varphi }_{\delta }^{n}.
\end{equation*}
Let us write $\tilde \Phi_\delta:=\Phi^{\log\lambda_\delta}_{\mu_\delta^{-1}\tilde H_\delta^n }$
to shorten the notation. By the definition of the Hamiltonian flow,
\begin{equation}\label{Eq:tildephi}
\widetilde{\varphi }_{\delta }^{n}(t+\log\lambda_\delta)=
\tilde\Phi_\delta
(\widetilde{\varphi }_{\delta }^{n}(t))\,.
\end{equation}
It is easy to see that 
\[
\tilde\Phi_\delta=\Phi^{\log\lambda_\delta}_{\mu_{n,\delta}^{-1}\tilde H_\delta^n }
=
\Phi^{1}_{\mu_{n,\delta}^{-1}\log\lambda_\delta\tilde H_\delta^n }
=
\hat S^{-1}\circ \Phi^{1}_{\mu_{n,\delta}^{-1}\log\lambda_\delta H_\delta^n }
\circ\hat S
=\tilde F_\delta+O(\delta^{n+1})\,,
\]
where we have taken into account that
$\mu_{n,\delta}^{-1}\log\lambda_\delta=1+O(\delta^n)$
implies 
\[
\mu_{n,\delta}^{-1}\log\lambda_\delta H_\delta^n =H_\delta^n+O(\delta^{n+1}).
\]
We have got
\begin{equation}\label{Eq:tildeFF}
\tilde\Phi_\delta=\tilde F_\delta+O(\delta^{n+1})\,,
\end{equation}
which will play a crucial role in the proof.

\medskip

Now we define
\begin{equation}
\xi (t):=\widetilde{\psi}{}_{\delta }^{-}(t)
-\widetilde{\varphi}_{\delta}^{n}(t).  \label{Def_xi}
\end{equation}
It follows from the definition of $\xi (t)$ and equations (\ref{Eq:tildephi})
and (\ref{Separatrix_flow}) that
\begin{equation}
\xi (t+\log \lambda _{\delta })=
\widetilde{F}_{\delta }(\widetilde{\varphi }_{\delta }^{n}(t)+\xi (t))
-\tilde \Phi _{\delta }(\widetilde{\varphi }_{\delta }^{n}(t)).
\label{eqn_xi}
\end{equation}
We are proving that there exists a $\xi (t)$ such that%
\begin{equation}
\xi (t)=O(\delta ^{n}e^{2t}).  \label{closeness_separatrices}
\end{equation}

First we rewrite the equation for $\xi$ in the form
\begin{equation}
\xi (t+\log \lambda _{\delta })-\mathrm{A}_{\delta }\xi (t)=R(t,\xi (t))\,,
\label{Solve_xi}
\end{equation}
where
\[
R(t,\xi (t)):=
\widetilde{F}_{\delta }(\widetilde{\varphi }_{\delta }^{n}(t)+\xi (t))
-\tilde \Phi _{\delta }(\widetilde{\varphi }_{\delta }^{n}(t))
-\mathrm{A}_{\delta }\xi (t).
\]
Expanding $\widetilde{F}_{\delta }$ in Taylor series around $\widetilde{\varphi }
_{\delta }^{n}(t)$ we get
\begin{equation}
R(t,\xi (t))=R_{0}(t)+R_{1}(t)\xi (t)+R_{2}(\xi (t)),  \label{def_R}
\end{equation}%
where%
\begin{eqnarray}
R_{0}(t) &:&=\widetilde{F}_{\delta }(\widetilde{\varphi }_{\delta
}^{n}(t))-\tilde\Phi _{\delta }(\widetilde{\varphi }_{\delta }^{n}(t)),  \notag \\
R_{1}(t) &:&=D\widetilde{F}_{\delta }(\widetilde{\varphi }_{\delta
}^{n}(t))-\mathrm{A}_{\delta },  \label{Eq:R2def} \\
R_{2}(\xi (t)) &:&=O(\left\Vert \xi (t)\right\Vert ^{2}),  \label{Order_R2}
\end{eqnarray}%
where the upper bound for $R_{2}(\xi (t))$ is a standard estimate for
the remainder of the Taylor expansion. Next we show that%
\begin{equation}
R_{0}(t)=O(\delta ^{n}e^{2t})  \label{Order_R0}
\end{equation}%
and%
\begin{equation}
R_{1}(t)=O(\delta e^{t}).  \label{Order_R1}
\end{equation}%
We begin with (\ref{Order_R1}). We note that since $\tilde F_\delta$ is close
to identity we have $\tilde F_\delta=\mathrm{id}+\delta \tilde G_\delta$ and consequently
$D\tilde F_\delta=\mathrm{id}+\delta D\tilde G_\delta$ and in particular
$\mathrm{A}_\delta=D\tilde F_\delta(0)=\mathrm{id}+\delta D\tilde G_\delta(0)$. Therefore
\[
D\widetilde{F}_{\delta }(\mathbf{x})-\mathrm{A}_{\delta }=
\delta \left(D\tilde G_\delta(\mathbf{x})- D\tilde G_\delta(0)\right)
=O(\delta  |\mathbf{x}|)\,.
\]
Substituting $\widetilde{\varphi }_{\delta }^{n}(t)$ instead of $\mathbf{x}$ and using
the upper bound $\widetilde{\varphi }_{\delta }^{n}(t)=O(e^t)$  as $\func{Re}(t)\rightarrow -\infty$,
we see the upper bound (\ref{Order_R1}) follows from the definition (\ref{Eq:R2def}).

The derivation of the upper bound for $R_{0}(t)$ is a bit more sophisticated.
First we note that
\[
\tilde F_\delta(0)=\tilde \Phi _{\delta }(0)=0
\qquad
\mbox{and}
\qquad
D\tilde F_\delta(0)=D\tilde \Phi _{\delta }(0)=\mathrm{A}_\delta\,.
\]
Consequently
\[
\tilde F_\delta(\mathbf{x})-\tilde \Phi _{\delta }(\mathbf{x})=O(|\mathrm{x}|^2)\,.
\]
Substituting $\widetilde{\varphi }_{\delta }^{n}(t)$ instead of $\mathbf{x}$
and taking into account the definition of $R_2$ we
conclude that the components of $\frac{R_{0}(t)}{\left(e^{t}\right) ^{2}}$ 
are bounded analytic function on $\func{Re}(t)\leq r_0$.
Moreover $\widetilde{\varphi }_{\delta }^{n}$ is $2\pi i$ periodic and analytic, therefore
each component of the vector $R_0(t)$ is $2\pi i$ is also periodic and analytic in $t$.
Consequently the maximum of $\frac{R_{0}(t)}{\left(e^{t}\right) ^{2}}$ is achieved at
the line $\func{Re}(t)=r_0$.\footnote%
{
The $2\pi i$ periodicity allows a substitution $z=e^t$ and an application of a maximum
modulus principle to the result.
} 
Using (\ref{Eq:tildeFF}) we get for $\Re t<r_0$
\begin{equation*}
\left\vert 
e^{-2t}R_{0}(t)
\right\vert \leq 
\sup_{\func{Re}(t)=r_0}
\frac{
\left\vert \widetilde{F}_{\delta }(\widetilde{\varphi }_{\delta}^{n}(t))-
\tilde\Phi_{\delta }(\widetilde{\varphi }_{\delta}^{n}(t))
\right\vert }{\left\vert e^{2t}\right\vert }
\leq
c_{2}\delta^{n+1}
\end{equation*}
which is equivalent to the desired estimate (\ref{Order_R0}).

Now we observe that if the following sum converges
\begin{equation*}
\xi (t)=\sum_{k\geq 1}\mathrm{A}_{\delta }^{k-1}f(t-k\log \lambda _{\delta }),
\end{equation*}
then it solves the difference equation
\[
\xi(t+\log\lambda_\delta)-\mathrm{A}_\delta\xi(t)=f(t)\,.
\]
Indeed, substituting the series into the left hand side of the equation we get 
\begin{eqnarray*}
\xi (t+\log \lambda _{\delta })-\mathrm{A}_{\delta }\xi &=&\sum_{k\geq 1}\mathrm{A}_{\delta
}^{k-1}f(t-(k-1)\log \lambda _{\delta })  \\
&&-\sum_{k\geq 1}\mathrm{A}_{\delta }^{k}f(t-k\log \lambda _{\delta })=f(t) .
\end{eqnarray*}
Thus instead of solving $\left( \ref{Solve_xi}\right) $ we solve the
``integral equation"%
\begin{equation}
\xi (t)=\sum_{k=1}^{\infty }
\mathrm{A}_{\delta }^{k-1}R(t-k\log \lambda _{\delta},\xi (t-k\log \lambda _{\delta })),  
\label{Solve_xi_integral}
\end{equation}%
where $R(t,\xi (t))$ is given by (\ref{def_R}). 
We will show that we can solve $\left( \ref{Solve_xi_integral}\right)$
by applying contraction mapping arguments in a ball in the Banach space
of functions which are analytic, $2\pi i$ periodic in a half plane 
$\func{Re}(t)<r_0$ and have bounded norm
\[
\|\xi\|=\sup_{\func{Re} (t)<r_{0}}\left\vert e^{-2t}\xi (t)\right\vert
\]
where $|\cdot|$ stands for the Euclidean norm in $\mathbb{C}^2$.
Let
\begin{equation}
\mathcal{B}=\left\{\; \xi :\left\Vert \xi \right\Vert  <K\delta ^{n}\;\right\} 
\label{def_ball}
\end{equation}%
where $K>0$ and $r_0\in\mathbb{R}$ will be chosen later in the proof.
We note that if $\xi \in \mathcal{B}$ then for all $\func{Re}(t)<r_0$
\begin{eqnarray*}
| \xi (t)| &\leq&\left\Vert \xi \right\Vert e^{2\func{Re}(t)} \\
				 &<&  K\delta ^{n}e^{2\func{Re}(t)}\,.
\end{eqnarray*}%
We will show that there exist $K$ and $r_{0}$ such that the ball $\mathcal{B}
$ is invariant under the operator in the right-hand side of $\left( \ref%
{Solve_xi_integral}\right) $ and that the operator is a contraction operator
on the ball $\mathcal{B}.$ Let us denote the operator in the right-hand side
of $\left( \ref{Solve_xi_integral}\right) $ by $\mathcal{F}.$ We begin by
showing that the ball $\mathcal{B}$ is invariant under $\mathcal{F}$, i.e. $%
\mathcal{F}(\xi )\in \mathcal{B}$ for any $\xi \in \mathcal{B}.$ Let $\xi
\in \mathcal{B}$ then%
\begin{equation}
\mathcal{F}(\xi )=\sum_{k=1}^{\infty }\mathrm{A}_{\delta }^{k-1}\left. \left(
R_{0}+R_{1}\xi +R_{2}(\xi)\right)
 \right\vert _{t-k\log\lambda _{\delta }}. 
 \label{operator_F}
\end{equation}
Using the upper bounds for $R_{0}(t),R_{1}(t)$ and $R_{2}(\xi (t))$ 
from equations (\ref{Order_R0}), (\ref{Order_R1}) and (\ref{Order_R2})  
respectively we obtain
\begin{eqnarray*}
\left|
\left. R_{0}\right\vert _{t-k\log \lambda _{\delta }} 
\right|
&\le&
c_{1}\delta^{n+1}e^{2\func{Re}(t)}\lambda _{\delta }^{-2k}, \\
\left|
\left. R_{1}\xi \right\vert _{t-k\log \lambda _{\delta }} 
\right|
&\le&c_{2}\delta e^{3\func{Re}(t)}\lambda _{\delta
}^{-2k}\|\xi\| , \\
\left|
\left. R_{2}(\xi)\right\vert _{t-k\log \lambda _{\delta }}
\right|
&=&c_{3}\delta ^{2n}K^{2}e^{4\func{Re}(t)}\lambda _{\delta }^{-4k},
\end{eqnarray*}%
where $c_{1},c_{2},c_{3}$ are constants. Since the norm of $\mathrm A_\delta$
does note exceed $\lambda_\delta$, we can use the above estimates to
get an upper bound for the left hand side of $\left( \ref{operator_F}\right) $:
\begin{equation}
\left\Vert \mathcal{F}(\xi )\right\Vert \leq c_{1}c_{4}\delta
^{n}e^{2r_{0}}+c_{2}c_{4}\delta ^{n}e^{3r_{0}}K\lambda_{\delta }
^{-1}+c_{3}c_{4}\lambda_{\delta }^{-2}\delta ^{2n-1}K^{2}e^{4r_{0}},
\label{closedness}
\end{equation}%
where $c_{4}=\sup_{|\delta|<\delta_0}\delta \frac{\lambda _{\delta }^{-2}}{1-\lambda _{\delta }}$ is
finite since $\frac{1}{1-\lambda _{\delta }}=O(\delta ^{-1}).$
To show that there exists a $K,r_{0}$ such that $\left\Vert \mathcal{F}(\xi
)\right\Vert \leq \delta ^{n}K$ for all $\delta <\delta _{0}$ we start by
choosing $K=4c_{1}c_{4}e^{2r_{0}}.$ Then by choosing $r_{0}=\frac{1}{3}\log 
\frac{\lambda _{\delta }}{4c_{2}c_{4}}$ we obtain that the second term in
the right-hand side is bounded by%
\begin{equation*}
c_{2}c_{4}\delta ^{n}e^{3r_{0}}K\lambda ^{-1}<\frac{1}{4}\delta ^{n}K.
\end{equation*}%
Finally letting $\delta _{0}=\sqrt[{n-2}]{\frac{c_{2}}{4c_{1}c_{3}}}$,
wee that  the last
term in the right-hand side of $\left( \ref{closedness}\right) $%
\begin{equation*}
c_{3}c_{4}\lambda _{\delta }^{-2}\delta ^{2n-1}K^{2}e^{4r_{0}}\leq \frac{1}{4%
}\delta ^{n}K
\end{equation*}%
for all $\delta <\delta _{0}.$ These choices yield%
\begin{equation*}
\left\Vert \mathcal{F}(\xi )\right\Vert \leq \frac{3}{4}\delta ^{n}K,
\end{equation*}%
hence $ \mathcal{F}(\xi ) \in \mathcal{B}$.

Next we show that $\mathcal{F}$ is a contraction. Let $\xi _{1},\xi _{2}\in 
\mathcal{B},$ and consider
\begin{eqnarray*}
\left\vert R(t,\xi _{1}(t))-R(t,\xi _{2}(t))\right\vert &=&\left\vert 
R_{1}(t)\left( \xi _{1}(t)-\xi _{2}(t)\right) 
+R_{2}(t,\xi _{1}(t))-R_{2}(t,\xi_{2}(t))\right\vert 
\\
&\leq &\left( c_{2}\delta e^{\func{Re}(t)}+c_{3}\left\Vert R_{2}\right\Vert
_{C^{2}}K\delta ^{n}e^{2\func{Re}(t)}\right) \left\vert \xi _{1}(t)-\xi
_{2}(t)\right\vert .
\end{eqnarray*}%
Then we have%
\begin{equation*}
\left\Vert \mathcal{F}(\xi _{1})-\mathcal{F}(\xi _{2})\right\Vert \leq
\sum_{k=1}^{\infty }\lambda _{\delta }^{-1+k}
\sup_{\func{Re}(t)<r_{0}}
\left\vert\left.
\left( R(\xi _{1})-R(\xi _{2})\right)\right\vert _{t-k\log\lambda _{\delta }}
\,
e^{2\func{Re}(t)}\right\vert .
\end{equation*}%
Taking into account the previously obtained upper bound for
$R_{1}$ and $R_{2}$ we get
\begin{equation*}
\left\Vert \mathcal{F}(\xi _{1})-\mathcal{F}(\xi _{2})\right\Vert \leq
c_{5}\left( c_{2}e^{r_{0}}+c_{3}\left\Vert R_{2}\right\Vert
_{C^{2}}Ke^{4r_{0}}\delta ^{n-1}\lambda _{\delta }^{-1}\right) \left\Vert
\xi _{1}-\xi _{2}\right\Vert ,
\end{equation*}%
where $c_{2},c_{3}$ and $c_{5}$ are constants. By choosing $r_{0}$
sufficiently small the contraction constant%
\begin{equation*}
\alpha =c_{5}\left( c_{2}e^{r_{0}}+c_{3}\left\Vert R_{2}\right\Vert
_{C^{2}}Ke^{4r_{0}}\delta ^{n-1}\lambda _{\delta }^{-1}\right) <\frac{1}{2}\,.
\end{equation*}%
It now follows by the contraction mapping principle that there exists a
unique solution $\xi\in\mathcal{B}$ which satisfies equation
(\ref{closeness_separatrices}).

Now we recall that
\begin{equation*}
\psi _{\delta }^{-}=S\circ \widetilde{\psi }{}_{\delta }^{-}
\end{equation*}
and
\begin{equation*}
\varphi _{\delta }^{n}=\widehat{S}\circ \widetilde{\varphi }_{\delta }^{n}
\end{equation*}
Taking the difference we get the estimate: 
\begin{eqnarray*}
\psi _{\delta }^{-}(t)-\varphi _{\delta }^{n}(t)&=&
S\circ \widetilde{\psi }{}_{\delta }^{-}(t)-
\widehat{S}\circ \widetilde{\varphi }_{\delta }^{n}(t)
=
S\circ \widetilde{\psi }{}_{\delta }^{-}(t)
-
S\circ \widetilde{\varphi }_{\delta }^{n}(t)
+O(\delta ^{n})
\\
&=&
S\left( \widetilde{\psi }{}_{\delta }^{-}(t)
-
\widetilde{\varphi }_{\delta }^{n}(t)
\right)
+O(\delta ^{n})=O(\delta^n)\,.
\end{eqnarray*}
We have proved that there is $r_0\in\mathbb R$ and a separatrix solution
of (\ref{Eq:mainfde}) such that
\begin{equation}
\psi _{\delta }^{-}(t)=\varphi _{\delta }^{n}(t)+O(\delta^n)
\end{equation}
for all $t$ with $\Re t<r_0$.

\subsection{Extension lemma\label{Se:extension}}

In the previous section we have seen that the local unstable separatrix of
the flow and the map approximates each other well for the half-plane $\func{%
Re}(t)<r_{0}$ for some $r_{0}\in \mathbb{R}.$ In this section we show that
we can extend the domain in which the separatrices stay close to each other.

Consider two close to identity maps%
\begin{eqnarray}
F &=&\mathrm{id}+\delta G,  \label{Aux_CloseMaps} \\
\widetilde{F} &=&\mathrm{id}+\delta \widetilde{G},  \notag
\end{eqnarray}%
The following lemma gives a closeness result for the difference of their first 
$O(\delta ^{-1})$ iterates.

\begin{lemma}[Extension Lemma]
\label{LemmaExtension}
Suppose two close to identity maps $F_\delta$ and $\tilde F_\delta$
are both defined on a $\delta$-independent convex
neighbourhood of a set $\overline{\mathcal{D}}\subset\mathbb{C}^2$ 
and  on this neighbourhood $\tilde F_\delta=F_\delta+O(\delta^{n+1})$.
Then for any $r>0$ and any $T>0$ 
there are two constants $\delta_0>0$ and $K$ such that
for any two points $z_0,\tilde z_0\in D$ such that 
$z_k:=F_\delta^k(z_0)\in D$
for all $k$, $0\le k\le T\delta^{-1}$, the inequality
 $|\tilde z_0-z_0|\le r\delta^n$ implies that 
\begin{equation*}
\left\vert \tilde z_{k}-z_{k}\right\vert \leq K\delta ^{n}
\qquad\mbox{for all $0\le k\le T\delta^{-1}$}
\end{equation*}
where $\tilde z_k=\tilde F_\delta^k(\tilde z_0)$.
\end{lemma}

To prove this lemma we will need the following Gronwall type inequality.

\begin{lemma}\label{Gronwall}
Let $b\geq 0$ and $(a_{n})$ be an increasing sequence of positive numbers.
If $z_{n}$ is a sequence such that $z_{n}\leqslant a_{n}+b\sum\limits_{k=1}^{n-1}z_{k}$, 
then $z_{n}\leqslant a_{n}\left( 1+b\right) ^{n-1}$.
\end{lemma}

\begin{proof}
By induction. The statement is true for $n=1$. Assume that the statement is
true for all $n\le p$, {\em i.e.} $z_{n}\leqslant a_{n}\left(1+b\right) ^{n-1}$. 
Then for $n=p+1$ we have 
\begin{eqnarray*}
z_{p+1} &\leqslant &a_{p+1}+b\sum_{k=1}^{p}z_{k}\leqslant
a_{p+1}+b\sum_{k=1}^{p}a_{k}\left( 1+b\right) ^{k-1} \\
&\leqslant &a_{p+1}\left( 1+b\sum_{k=1}^{p}\left( 1+b\right) ^{k-1}\right)
=a_{p+1}\left( 1+b\frac{1-\left( 1+b\right) ^{p}}{1-(1+b)}\right) \\
&=&a_{p+1}\left( 1+b\right) ^{p}
\end{eqnarray*}%
By induction the statement is true for all $n\in \mathbb{N}.$
\end{proof}

\medskip

\begin{proof}[Proof of Lemma~\ref{LemmaExtension}] 
Let us estimate the difference of the $N^{th}$ iterates 
of the two maps (\ref{Aux_CloseMaps}).
It is easy to see that
\[
z_N=z_0+\delta\sum_{k=0}^{N-1}G(z_k)
\qquad\mbox{and}\qquad
\tilde z_N=\tilde z_0+\delta\sum_{k=0}^{N-1}\tilde G(\tilde z_k)\,.
\]
Let $\tilde G-G=\delta^n R$, where $R$ is bounded due to the assumptions
of the lemma. Taking the difference
and applying the triangle inequality we get
\begin{eqnarray*}
\left\vert \tilde z_{N}-z_{N}\right\vert  &=&
\left\vert\tilde z_{0}-z_{0}\right\vert 
+\delta \sum_{k=0}^{N-1}
\left\vert\tilde G(\tilde z_{k})- G(z_{k}) \right\vert  \\
&\leq &\left\vert \tilde z_{0}-z_{0}\right\vert +\delta
\sum_{k=0}^{N-1}\left\vert G(\tilde z_k)-G(z_k)
\right\vert +\delta ^{n+1}\sum_{k=0}^{N-1}\left\vert R(\tilde z_{k})\right\vert .
\end{eqnarray*}
Since the domain of $G$ is convex the map is Lipschitz
\[
\left\vert G(\tilde z_{k})-G(z_k)\right\vert\le L|\tilde z_k-z_k|\,,
\]
where $L=\|DG\|$. Moreover $R$ is bounded and we conclude
\begin{equation*}
\left\vert\tilde z_{N}-z_{N}\right\vert \leq \left\vert \tilde z_{0}-z_{0}
\right\vert +\delta L \sum_{k=0}^{N-1}\left\vert\tilde z_{k}-z_{k}\right\vert +
\delta^{n+1}N\left\Vert R\right\Vert .
\end{equation*}
We recall that $b:=|\tilde z_0-z_0|\le r\delta^n$. The sequence
$a_{N}:=\delta ^{n}r+\delta ^{n+1}N\|R\| $ 
is increasing and we apply the Gronwall type
inequality of Lemma \ref{Gronwall} to obtain%
\begin{equation*}
\left\vert \tilde z_{N}-z_{N}\right\vert \leq \left(
\delta^n r
 +\delta ^{n+1}N\left\Vert R\right\Vert
\right) \left( 1+\delta L  \right) ^{N-1}.
\end{equation*}%
Then we use the textbook inequality 
\[
\left( 1+\delta
L\right) ^{\frac{1}{\delta L}}\leq e
\] 
as an upper bound for the last parenthesis
in the right-hand side and get%
\begin{equation*}
\left\vert \tilde z_{N}-z_{N}\right\vert \leq 
\delta^n\left( r+\delta N\left\Vert R\right\Vert \right) 
e^{\delta L (N-1)}.
\end{equation*}
Then for all $0\le N\leq \frac{T}{\delta }$ 
we have
\begin{equation*}
\left\vert \tilde z_{N}-z_{N}\right\vert \leq \delta ^{n} 
\left(
r+T\left\Vert R\right\Vert \right) e^{L T}.
\end{equation*}
Let $K=(r+T \|R\|)\exp(LT)$ to complete the proof. 
\end{proof}

\subsection{Application of the extension lemma}

We will now use Lemma \ref{LemmaExtension} to extend the domain in which the
local separatrix of the flow approximates the local separatrix of the map.
We extend the approximation to the domain defined by%
\begin{equation}
\mathcal{T}_{0}=\left\{ t\in \mathbb{C}:\varphi _{0}(t-s)\in \mathcal{D}%
\;\forall s\geq 0,\;\func{Re}(t)\leq T\right\} ,  \label{Set_Tzero}
\end{equation}%
where $\varphi _{0}$ is the separatrix of the limit flow
$\mu_0\dot{\varphi}_{0}=G_{0}(\varphi _{0}).$ Note that $\mathcal{T}_{0}$ is independent of 
$\delta$.

\begin{corollary}
Let $\psi _{\delta }^{-}(t)$ and $\varphi _{\delta }^{n}(t)$ be defined as
in Theorem \ref{Thm_LAT} and let $T\in \mathbb{R}$ be fixed.  Then 
$\psi _{\delta }^{-}$ is defined on $\mathcal{T}_{0}$ and
\begin{equation*}
\psi _{\delta }^{-}(t)=\varphi _{\delta }^{n}(t)+O(\delta ^{n})
\end{equation*}
uniformly in $\mathcal{T}_{0}.$
\end{corollary}

\begin{remark}
Uniformly in $\mathcal{T}_{0}$ means that%
\begin{equation*}
\left\vert \psi _{\delta }^{-}(t)-\varphi _{\delta }^{n}(t)\right\vert \leq
K(D,T)\delta ^{n}.
\end{equation*}
\end{remark}

\begin{proof}
By definition%
\begin{equation*}
\psi _{\delta }^{-}(t+\log \lambda _{\delta })=F_{\delta }(\psi _{\delta
}^{-}(t))
\end{equation*}%
and%
\begin{equation*}
\dot{\varphi}_{\delta }^{n}(t)=\mu _{n,\delta }^{-1}J\nabla H_{\delta
}^{n}(\varphi _{\delta }^{n}(t)),
\end{equation*}%
which implies that%
\begin{equation*}
\varphi _{\delta }^{n}(t+\log \lambda _{\delta })=\Phi _{\mu _{n,\delta
}^{-1}H_{\delta }^{n}}^{\log \lambda _{\delta }}(\varphi _{\delta
}^{n}(t))=\Phi _{\mu _{n,\delta }^{-1}\log \lambda _{\delta }H_{\delta
}^{n}}^{1}(\varphi _{\delta }^{n}(t)).
\end{equation*}%
Theorem \ref{Thm_LAT} also states that%
\begin{equation*}
F_{\delta }=\Phi _{H_{\delta }^{n}}^{1}+O(\delta ^{n+1}),
\end{equation*}%
hence%
\begin{equation}
\Phi _{\mu _{n,\delta }^{-1}\log \lambda _{\delta }H_{\delta
}^{n}}^{1}
=\Phi _{(1+O(\delta ^{n}))H_{\delta
}^{n}}^{1}=\Phi _{H_{\delta }^{n}}^{1}+O(\delta ^{n+1})=F_{\delta }+O(\delta
^{n+1}).  \label{Approx_F}
\end{equation}%
Let us denote this map by $\widetilde{F}_{\delta }$.
Then we have 
\begin{equation*}
\varphi _{\delta }^{n}(t+\log \lambda _{\delta })=\widetilde{F}_{\delta
}(\varphi _{\delta }^{n}(t)).
\end{equation*}%
It now follows that $F_{\delta }$ and $\widetilde{F}_{\delta }$ are two
close to identity maps such that $\widetilde{F}_{\delta }=F_{\delta
}+O(\delta ^{n+1}).$ Define%
\begin{equation*}
z_{k}:=F_{\delta }^{k}(\psi _{\delta }^{-}(t_{0}))
\end{equation*}%
and%
\begin{equation*}
z_{k}^{\ast }:=\widetilde{F}_{\delta }^{k}(\varphi _{\delta }^{n}(t_{0})).
\end{equation*}%
Since $\func{Re}(t_{0})\leq r_{0}$ it follows from Theorem \ref%
{Thm_LAT} that%
\begin{equation*}
\left\vert z_{0}-z_{0}^{\ast }\right\vert \leq O(\delta ^{n}).
\end{equation*}%
To estimate $\left\vert \psi _{\delta }^{-}(t)-\varphi _{\delta
}^{n}(t)\right\vert $ we do $\left\lfloor \frac{\func{Re}(t)-r_{0}}{\log
\lambda _{\delta }}\right\rfloor $ iterations of the maps $F_{\delta },%
\widetilde{F}_{\delta },$ hence the number of iterates is bounded by $\frac{%
L(T,r_{0})}{\delta }$ where $L$ is a constant. Finally we note that by the
definition of $\mathcal{T}_{0}$ it follows that $z_{k}^{\ast }$ belongs%
to a $\delta $%
-neighbourhood of $\mathcal{D}$ for all $0\leq k\leq \frac{L(T,r_{0})}{%
\delta }.$ The corollary now follows from Lemma \ref{LemmaExtension}.
\end{proof}

%


\begin{thebibliography}{GST00}
\bibitem{AKN}
Arnold, V.I., Kozlov, V.V., Neishtadt, A.I.,
{\em Mathematical aspects of classical and celestial mechanics. 
Dynamical systems. III.} Third edition. 
Encyclopaedia of Mathematical Sciences, 3. Springer-Verlag, Berlin, 2006.

\bibitem{Seara}
Baldom\'a, I., Seara, T. M., Breakdown of heteroclinic orbits 
for some analytic unfoldings of the Hopf-zero singularity. 
J. Nonlinear Sci. 16 (2006), no. 6, 543--582.

\bibitem{Champneys}
Champneys, A. R. Codimension-one persistence beyond all orders of homoclinic orbits 
to singular saddle centres in reversible systems. Nonlinearity 14 (2001), no. 1, 87--112.

\bibitem{Delshams}
Delshams, A., Ramírez-Ros, R., Exponentially small splitting 
of separatrices for perturbed integrable standard-like maps.
 J. Nonlinear Sci. 8 (1998), no. 3, 317--352. 

\bibitem{Deprit1969}
Deprit, Andr\'e
Canonical transformations depending on a small parameter.
Celestial Mech. 1 1969/1970 12--30. 


\bibitem{Duarte99}
Duarte, P., Abundance of elliptic isles at conservative bifurcations. 
Dynam. Stability Systems 14 (1999), no. 4, 339--356. 

\bibitem{Duarte00}
Duarte, P., Persistent homoclinic tangencies for conservative maps near the identity. 
Ergodic Theory Dynam. Systems 20 (2000), no. 2, 393--438.

\bibitem{DRR}
Dumortier, F., Rodrigues, P.R., Roussarie, R.,
{\em Germs of diffeomorphisms in the plane}. 
Lecture Notes in Mathematics, 902. Springer-Verlag, Berlin-New York, 1981. 197 pp.

\bibitem{FS1990}
Fontich, E.; Sim\'o, C. The splitting of separatrices for analytic diffeomorphisms.  
Ergodic Theory Dynam. Systems  10  (1990),  no. 2, 295--318.

\bibitem{G99}
Gelfreich V., {\em A proof of the exponentially small transversality 
of the separatrices for the standard map}, Comm. Math. Phys. 201/1,  (1999) 155--216

\bibitem{Gel00}
Gelfreich V., {\em Splitting of a small separatrix loop 
near the saddle-center bifurcation in area-preserving maps}, 
Physica D 136, no.3--4, (2000) 266--279

\bibitem{Gel02}
Gelfreich V., {\em Near strongly resonant periodic orbits in a Hamiltonian system}. 
Proc. Nat. Acad. Sci. USA, vol. 99, no. 22, (2002) 13975--13979.

\bibitem{GL01}
Gelfreich V., Lazutkin V., {\em Splitting of Separatrices: 
perturbation theory and exponential smallness}, 
Russian Math. Surveys Vol. 56, no. 3, (2001) 499--558

\bibitem{GN2008}
Gelfreich V., Naudot N. (2008) Analytic invariants associated with a parabolic fixed point in C2
    Ergodic Theory and Dynamics Systems (accepted for publication) 

\bibitem{GS01}
Gelfreich V., Sauzin D. 
{\em Borel summation and splitting of separatrices for the Henon map}, 
Annales l'Institut Fourier, vol. 51, fasc. 2, (2001) p.513--567. 

\bibitem{HakimMallock}
Hakim, V., Mallick, K.,
Exponentially small splitting of separatrices, 
matching in the complex plane and Borel summation. 
Nonlinearity 6 (1993), no. 1, 57--70. 


\bibitem{Lazutkin1984}
Lazutkin, V.F., Splitting of separatrices for the Chirikov standard map. 
J. Math. Sci., New York 128, No. 2, 2687-2705 (2005) 
(translation from a Russian preprint (1984))

\bibitem{Lombardi}
Lombardi, E.,
Oscillatory integrals and phenomena beyond all algebraic orders 
with applications to homoclinic orbits in reversible systems.
Lecture Notes in Mathematics. 1741. Berlin: Springer, 412~p (2000)


\bibitem{Ramirez}
Ram\'irez-Ros, R., Exponentially small separatrix splittings and almost 
invisible homoclinic bifurcations in some billiard tables. Phys. D 210 (2005), no. 3-4, 149--179. 

\bibitem{Simo91}
Sim\'o, C., Broer, H., Roussarie, R., A numerical survey on the Takens-Bogdanov 
bifurcation for diffeomorphisms. European Conference on Iteration Theory (Batschuns, 1989), 
320--334, World Sci. Publ., River Edge, NJ, 1991

\bibitem{Treschev98a}
Treschev, D. Width of stochastic layers in near-integrable two-dimensional symplectic maps. 
Phys. D 116 (1998), no. 1-2, 21--43.

\bibitem{Treschev98b}
Treschev, D., Splitting of separatrices for a pendulum with rapidly oscillating suspension point. 
Russian J. Math. Phys. 5 (1997), no. 1, 63--98 (1998).


\end{thebibliography}
\end{document}